\documentclass[12pt]{article}
\input epsf.tex

\usepackage{graphicx}
\usepackage{amsmath,amsthm,amsfonts,amscd,amssymb,comment,eucal,latexsym,mathrsfs}
\usepackage{stmaryrd}
\usepackage[all]{xy}

\usepackage{epsfig}

\usepackage[all]{xy}
\xyoption{poly}
\usepackage{fancyhdr}
\usepackage{wrapfig}
\usepackage{epsfig}

\theoremstyle{plain}
\newtheorem{thm}{Theorem}[section]
\newtheorem{prop}[thm]{Proposition}
\newtheorem{lem}[thm]{Lemma}
\newtheorem{cor}[thm]{Corollary}
\newtheorem{conj}{Conjecture}
\newtheorem{claim}{Claim}
\theoremstyle{definition}
\newtheorem{defn}{Definition}
\newtheorem{notation}{Notation}
\theoremstyle{remark}
\newtheorem{remark}{Remark}
\newtheorem{example}{Example}

\advance \topmargin by -\headheight
\advance \topmargin by -\headsep
\evensidemargin \oddsidemargin
  \def\C{{\mathbb{C}}}  \def\E{{\mathbb{E}}}         \def\N{{\mathbb{N}}}    \def\R{{\mathbb{R}}}        \def\Z{{\mathbb{Z}}}

\def\ba{{\bar{a}}}                       \def\bx{{\bar{x}}} \def\by{{\bar{y}}} 

  \def\bfc{{\bf{c}}}                       

\def\cA{{\mathcal{A}}}   \def\cD{{\mathcal{D}}}    \def\cH{{\mathcal{H}}}   \def\cK{{\mathcal{K}}}  \def\cM{{\mathcal{M}}} \def\cN{{\mathcal{N}}} \def\cO{{\mathcal{O}}} \def\cP{{\mathcal{P}}}  \def\cR{{\mathcal{R}}}    \def\cV{{\mathcal{V}}}    

       \def\tcH{{\tilde{\mathcal{H}}}}        \def\tcP{{\tilde{\mathcal{P}}}}

       \def\tcH{{\tilde{\cH}}}        \def\tcP{{\tilde{\cP}}}

  \def\tlambda{{\widetilde{\lambda}}} \def\tmu{{\widetilde{\mu}}}            

\newcommand{\La}{\Lambda}

\newcommand{\Si}{\Sigma}
\newcommand{\eps}{\epsilon}

\renewcommand\d{\delta}
\newcommand\g{\gamma}

\renewcommand\l{\lambda}
\newcommand\s{\sigma}

\newcommand\dom{\operatorname{dom}}

\newcommand\GL{{\operatorname{GL}}}
\newcommand\Gr{{\operatorname{Gr}}}

\newcommand\id{\operatorname{I}}
\newcommand\im{\operatorname{Im}}

\newcommand\Isom{\operatorname{Isom}}

\newcommand\Prob{\operatorname{Prob}}
\newcommand\Proj{\operatorname{Proj}}

\newcommand\Span{\operatorname{span}}

\def\cc{{\curvearrowright}}
\newcommand{\botimes}{\bar{\otimes}}
\newcommand{\dee}{\textrm{d}}
\newcommand{\resto}{\upharpoonright}

\newcommand{\ip}[1]{\langle #1 \rangle}

  \newcommand{\rL}{\textrm{L}}

\begin{document}
\title{A multiplicative ergodic theorem for von Neumann algebra valued cocycles}
\author{Lewis Bowen\footnote{This author acknowledges support from NSF grant DMS-1500389 and a Simons Fellowship.}, Ben Hayes\thanks{This author acknowledges support from NSF grants DMS-1827376 and DMS-2000105.} and Yuqing (Frank) Lin\footnote{This author acknowledges support from NSF grant DMS-1500389.}}
\maketitle

\begin{abstract}
The classical Multiplicative Ergodic Theorem (MET) of Oseledets is generalized here to cocycles taking values in a semi-finite von Neumann algebra. This allows for a continuous Lyapunov distribution.
\end{abstract}

\noindent
{\bf Keywords}: Multiplicative Ergodic Theorem, Oseledets' Theorem, Lyapunov exponents\\
{\bf MSC}: 37H15\\

\noindent
\tableofcontents

\section{Introduction}

\subsection{The finite dimensional MET}
Here is a version of the classical Multiplicative Ergodic Theorem (MET). Let $(X,\mu)$ be a standard probability space, $f:X \to X$ a measure-preserving transformation, and $c:\N \times X \to \GL(n,\R)$ a measurable cocycle:
$$c(n+m,x) = c(n, f^mx)c(m,x)\quad \forall n,m \in \N, ~\mu-a.e. ~x\in X.$$
Assume the first moment condition:
$$\int \log^+ \|c(1,x)\|~d\mu(x)<\infty,$$
Then there is a {\bf limit operator}
$$\La(x):= \lim_{n\to\infty} [c(n,x)^*c(n,x)]^{1/2n}$$
for a.e. $x$. Let $e^{\l_1(x)} >  \cdots > e^{\l_k(x)}$ be the distinct eigenvalues of $\La(x)$. Then $\l_1,\dots, \l_k$ are the {\bf Lyapunov exponents}.  They are invariant in the sense that $\l_i(f(x)) = \l_i(x)$ for a.e. $x$. If $m_i \in \N$ is the multiplicity of $\l_i$ then the {\bf Lyapunov distribution} is the discrete measure $\sum_{i=1}^k m_i \d_{\l_i}$.

Let $W_i$ be the $e^{\l_i(x)}$-eigenspace of $\La(x)$ and define
$$V_i = \sum_{j \ge i } W_j$$
so that  $V_k(x) \subset \cdots \subset V_1(x) = \R^n$ is a flag. The $V_i(x)$ are the {\bf Oseledets subspaces}. They are cocycle-invariant in the sense that $V_i(f(x)) = c(1,x)V_i(x)$ (for a.e. $x$).

Finally, for a.e. $x \in X$ and every vector $v \in V_i(x) \setminus V_{i+1}(x)$,
$$\lim_{n\to\infty} \frac{1}{n} \log \|c(n,x)v\| = \l_i(x).$$
This last condition can be expressed without reference to Lyapunov exponents by:
\begin{equation}\label{eqn:Lyapunov growth rate}
    \lim_{n\to\infty}  \| c(n,x)v \|^{1/n} = \lim_{n\to\infty} \| \La(x)^{n}v \|^{1/n}.
\end{equation}


 \subsection{Previous literature}

 Infinite-dimensional generalizations of the MET have appeared in \cite{MR647807, MR730286, MR3485402, MR2674952, MR877991, MR3400385, MR1178957}. Each of these assumes the operators $c(n,x)$ satisfy a quasi- compactness condition and consequently they trivialize away from the discrete part of the spectrum of the limit operators $\La(x)$. There are also geometric generalizations that are not directly concerned with operators on Hilbert spaces \cite{MR1729880}.

On the other hand, one does not expect there to be an unconditional generalization to infinite dimensions. For example, Voiculescu's example in \cite[Example 8.4]{MR2511586} shows there is a bounded operator $T:\ell^2(\N) \to \ell^2(\N)$ such that $|T^n|^{1/n}$ does not converge in the Strong Operator Topology. We could define the cocycle $c$ above by $c(n,x)=T^n$  to see that convergence cannot be guaranteed in the general setting of bounded operators on Hilbert spaces. Another counterexample is given in \cite{MR1178957}.


\subsection{von Neumann algebras}\label{sec: vNa intro}

The purpose of this paper is to establish a new MET in which the cocycle takes values in the group of invertible elements of a semi-finite tracial von Neumann algebra. To explain in more detail, let $\cH$ be a separable Hilbert space, $B(\cH)$ the algebra of bounded operators on $\cH$. A {\bf von Neumann algebra} is a sub-algebra $M \subset B(\cH)$ containing the identity ($\id\in M$) that is closed under taking adjoints and closed in the weak operator topology. Let $M_{+} \subset M$ be the positive operators on $M$ (these are the operators $x\in M$ satisfying $\langle x\xi,\xi\rangle \ge 0$ for all $\xi \in \cH$), and $M_{sa}$ the self-adjoint operators in $M$. We define a partial order on $M_{sa}$ by saying that $x\leq y$ if $y-x\in M_{+}$. A {\bf trace} on $M$ is a map $\tau:M_+ \to [0,\infty]$ satisfying
\begin{enumerate}
\item $\tau(x+y) =\tau(x)+\tau(y)$ for all $x,y \in M_+$;
\item $\tau(\l x) = \l \tau(x)$ for all $\l \in [0,\infty)$, $x \in M_+$ (agreeing that $0(+\infty) = 0$);
\item $\tau(x^*x)=\tau(xx^*)$ for all $x\in M$.
\end{enumerate}
We will always assume $\tau$ is
\begin{itemize}
\item faithful, which means $\tau(x^*x)=0 \Rightarrow x=0$;
\item normal, which means $\tau(\sup_i x_i) = \sup_i \tau(x_i)$ for every increasing net $(x_i)_i$ in  $M_+$;
\item semi-finite, which means for every $x \in M_+$ there exists $y \in M_+$ such that $0<y<x$ and $0<\tau(y)<\infty$.
\end{itemize}
 The pair $(M,\tau)$ is a {\bf finite tracial von Neumann algebra} if $\tau(\id)<\infty$.


The trace $\tau$ on $M$ is unique (up to scale) if and only if $M$ has trivial center. Many constructions considered here depend on the choice of trace but we will suppress this dependence from the notation and terminology.

\subsection{Example: the abelian case}\label{intro-abelian}
Fix a standard (semi-finite) measure space $(Y,\nu)$ and let $M=\rL^\infty(Y,\nu)$. For every $\phi \in M$, define the multiplication operator
$$m_\phi: \rL^2(Y,\nu) \to \rL^2(Y,\nu), \quad (m_\phi f)(y) = \phi(y)f(y).$$
The map $\phi \mapsto m_\phi$ embeds $M$ into the algebra of bounded operators on $\rL^2(Y,\nu)$. We will identify $\phi$ with $m_\phi$. Define the trace $\tau:M_+ \to [0,\infty)$ by
$$\tau(\phi) = \int \phi~\dee\nu.$$
With this trace, $(M,\tau)$ is a semi-finite von Neumann algebra. It is finite if $\nu(Y)$ is finite.


\subsection{Example: the full algebra}\label{intro-full}
Let $M=B(\cH)$ be the algebra of all bounded operators on a separable Hilbert space $\cH$. Also let $\{\xi_i\}_{i \in I} \subset \cH$ be an orthonormal basis. Define the {\bf canonical trace} $\tau_\cH:M^+ \to [0,\infty]$ by
$$\tau_\cH(a)=\sum_{i\in I} \langle a\xi_i,\xi_i\rangle.$$
It is well-known that the canonical trace does not depend on the choice of orthonormal basis. Moreover, $(B(\cH),\tau_\cH)$ is semi-finite. The multiplicative ergodic theorem for integrable cocycles $c:\N \times X \to \{\exp(x):~ x \in B(\cH), \tau_\cH(x^*x)<\infty\}$ was obtained in Karlsson-Margulis \cite{MR1729880}.

\subsection{Example: group von Neumann algebras}\label{intro-group vNa}
Let $G$ be a countable, discrete group and let $\lambda\colon G\to \mathcal{U}(\ell^{2}(G))$ be the homomorphism given by
\[(\lambda(g)\xi)(h)=\xi(g^{-1}h)\]
where $\mathcal{U}(\ell^{2}(G))\subset B(\ell^{2}(G))$ is the group of unitary operators acting on $\ell^{2}(G)$. Let
\[L(G)=\overline{\Span\{\lambda(g):g\in G\}}^{SOT}.\]
We call $L(G)$ the \emph{group von Neumann algebra of $G$}. We leave it as an exercise to verify that if $x\in L(G)$ and $\xi=x(\delta_{1})$ (where $\d_1 \in \ell^2(G)$ is the unit vector corresponding to the identity), then for any $\eta \in \ell^2(G)$, $x(\eta)=\xi*\eta$, where
\[(\xi*\eta)(g)=\sum_{h\in G}\xi(h)\eta(h^{-1}g).\]
So every operator in $L(G)$ is a convolution operator (it can in fact be shown that $L(G)$ consists precisely of the bounded convolution operators).  Set $\tau(x)=\ip{x(\delta_{1}),\delta_{1}}$. It is another exercise to verify that $(L(G),\tau)$ is a finite tracial von Neumann algebra. Moreover, if $G$ is non-abelian, then $L(G)$ is non-commutative.  

A von Neumann algebra $M$ is \textbf{diffuse} if for every nonzero projection $p\in M$, there is a nonzero projection $q\leq p$ with $q\ne p$.  See  Proposition \ref{prop:diffuse vNa TFAE} in the appendix for other equivalent ways to say that a finite von Neumann algebra is diffuse. For example, it follows from that Proposition that if $(M,\tau)$ is a finite tracial von Neumann algebra, then $M$ contains an isomorphic copy of $(\rL^{\infty}(X,\mu),\int \cdot\,d\mu)$ where $(X,\mu)$ is an atomless probability space, if and only if $M$ is diffuse. If $G$ is infinite, then $L(G)$ is diffuse. To see this, let $(g_{n})_{n=1}^{\infty}$ be a sequence of distinct elements of $G$. We claim that $\lambda(g_{n})\to_{n\to\infty}0$ in the weak operator topology. Indeed, for all $\xi,\eta\in c_{c}(G)$ we have that $\ip{\lambda(g_{n})\xi,\eta}=0$ for all large $n$. The fact that
\[\ip{\lambda(g_{n})\xi,\eta}\to_{n\to\infty}0\textnormal{ for all $\xi,\eta\in \ell^{2}(G)$}\]
then follows from the case of $\xi,\eta\in c_{c}(G)$, the fact that $\|\lambda(g_{n})\|=1$, and the density of $c_{c}(G)$ in $\ell^{2}(G)$.  So $\lambda(g_{n})\to 0$ in the weak operator topology, and  thus Proposition \ref{prop:diffuse vNa TFAE} implies that $M$ is diffuse. 

By Propositions \ref{prop:diffuse vNa TFAE} and \ref{prop: no nonzero compacts appendix}, if $G$ is infinite then $L(G)$ does not contain compact operators but it does contain a copy of $\rL^\infty(Y,\nu)$ for some non-atomic probability space $(Y,\nu)$. If $G$ is amenable and every nonidentity conjugacy class of $G$ is infinite, then by \cite{MR0454659} we know that $L(G)$ is isomorphic to the hyperfinite $\textrm{II}_{1}$-factor (see \S \ref{intro-hyperfinite} and \cite[Theorem 11.2.2]{anantharaman-popa}). For more details on group von Neumann algebras, see \cite{anantharaman-popa}.

\subsection{Main results}



\subsubsection{The limit operator}

Our first main result, Theorem \ref{thm:main} shows the existence of a limit operator.  The remaining results, Theorems \ref{thm:invariance}-\ref{thm:main2} are derived from Theorem \ref{thm:main} in \S \ref{sec:proofs of main results}. We state the result here and afterwards explain the notions of convergence and the notation used, such as $\GL^2(M,\tau)$ and $\cP$.

\begin{thm}\label{thm:main}
Let $(X,\mu)$ be a standard probability space, $f:X \to X$ an ergodic measure-preserving transformation, $(M,\tau)$ a von Neumann algebra with semi-finite faithful normal trace $\tau$. Let $M^\times \subset M$ be the group consisting of operators $x\in M$ that are bounded in the operator topology on $\cH$ with bounded inverse $x^{-1} \in M$. Let $c:\Z \times X \to M^\times \cap \GL^2(M,\tau)$ be a cocycle in the sense that
$$c(n+m,x) = c(n, f^mx)c(m,x)$$
for all $n,m \in \Z$ and a.e. $x \in X$. We assume $c$ is measurable with respect to the Strong Operator Topology on $M^\times$ (with respect to either the inclusion of $M^\times$ into $B(\cH)$ or into $B(\rL^2(M,\tau))$).

Assume the first moment condition:
$$\int_X \| \log (|c(1,x)|)\|_2 ~d\mu(x) < \infty.$$
Then for almost every $x\in X$, the following limit exists:
$$\lim_{n\to\infty} \frac{\|\log(c(n,x)^*c(n,x))\|_2}{n} = D.$$
$D$ is called the {\bf drift}. Moreover, if $D>0$ then for a.e. $x$, there exists a  limit operator $\La(x) \in \GL^2(M,\tau)$ satisfying
\begin{itemize}
\item $\lim_{n\to\infty} \frac{1}{n} d_\cP(|c(n,x)|, \La(x)^n) = 0$;
\item $\lim_{n\to\infty} |c(n,x)|^{1/n} \to \La(x)$ in $(\cP,d_\cP)$ and in measure;
\item $\lim_{n\to\infty} n^{-1} \log |c(n,x)| \to \log \La(x)$ in $\rL^2(M,\tau)$.
\end{itemize}
\end{thm}

\begin{remark}
The special case of Theorem \ref{thm:main} in which $(M,\tau)=(B(\cH),\tau_\cH)$ was proven in \cite{MR1729880}. If $M$ is infinite dimensional, then $\rL^{2}(M,\tau)$ is as well and so $M^{\times}$ does not contain any compact operators.
\end{remark}

\subsubsection{The regular representation}\label{sec: intro regular rep}

In order to explain the notation, we briefly recall the regular representation. So let $(M,\tau)$ be as in Theorem \ref{thm:main}. Let $\cN = \{x \in M:~ \tau(x^*x)<\infty\}$. Since $\tau$ is a trace, we have that $\cN=\{x\in M:~\tau(xx^{*})<\infty\}$.   By \cite[Lemma VII.1.2]{TakesakiII} $\cN$ is a $*$-closed two-sided ideal in $M$, the ideal $\cM=\Span\{x^{*}y:x,y\in \cN\}$ is linearly spanned by its positive elements and the trace extends linearly to $\cM$. So we can define an inner-product on $\cN$ by
$$\langle x, y \rangle : = \tau(x^*y).$$
The $\rL^2$-norm for $x\in \cN$ is defined by $\|x\|_2 = \tau(x^*x)^{1/2}$.

Let $\rL^2(M,\tau)$ denote the Hilbert space completion of $\cN$ with respect to this inner product.  Note that $\rL^2(M,\tau)$ does not usually coincide with $B(\cH)$: if $(M, \tau)$ is not finite then $\rL^2(M,\tau)$ does not contain the identity operator; on the other hand in many cases $\rL^2(M,\tau)$ also contains unbounded operators (see Example \ref{abelian1} and Remark 5).  For $x \in M$, the left-multiplication operator $L_x:M \to M$ defined by $L_x(y)=xy$ extends to a bounded linear operator on $\rL^2(M,\tau)$. Therefore, we may view $M$ as a sub-algebra of the algebra $B(\rL^2(M,\tau))$ of bounded linear operators on $\rL^2(M,\tau)$. This is the {\bf regular representation of $M$} (this is explained in more detail in \S \ref{sec:prelim}).

An operator (not necessarily bounded) $x$ on $\rL^2(M,\tau)$ is {\bf affiliated} with $M$ if it is closed, densely defined and commutes with every element in the commutant $M'=\{x \in B(\rL^2(M,\tau)):~xy=yx \, \forall y \in M\}$. A subspace is $V \subset \rL^2(M,\tau)$ is {\bf essentially dense} if for every $\eps>0$ there exists a projection $p \in M$ such that $\tau(\id-p) < \eps$ and $p\rL^2(M,\tau) \subset V$. Here, and throughout the paper, by a \emph{projection} we mean an \emph{orthogonal} projection, i.e. a self-adjoint idempotent.  Essentially dense subspaces are reviewed in \S \ref{sec:essentially}. An operator affiliated with $(M,\tau)$ is {\bf $\tau$-measurable} if its domain of definition is essentially dense.  Note that when $(M,\tau)$ is finite, then all affiliated operators are $\tau$-measurable.  Let $\rL^0(M,\tau)$ denote the algebra of $\tau$-measurable operators affiliated with $(M,\tau)$. This is a $\ast$-algebra, and in fact a complete topological $\ast$-algebra with respect to the measure topology. Moreover, the trace $\tau$ extends to $\tau: \rL^0(M,\tau)_+ \to [0,\infty]$ where $\rL^0(M,\tau)_+ \subset \rL^0(M,\tau)$ is the cone of positive $\tau$-measurable affiliated operators. Also if $x \in \rL^0(M,\tau)_+$ then $x^{-1/2}$ and $\log x$ are well-defined via the spectral calculus.  See \S \ref{sec:affiliated} and \S \ref{sec:measure} for details, including on the measure topology.

Let $\GL^2(M,\tau)$ consist of those elements $x \in \rL^0(M,\tau)^\times$ such that $\log |x| \in \rL^2(M,\tau)$ (where $\rL^0(M,\tau)^\times$ is the set of operators $x \in \rL^0(M,\tau)$ with $x^{-1} \in \rL^0(M,\tau)$). We prove in \S \ref{sec:GL2} that $\GL^2(M,\tau)$ is a group. Let $\cP = \GL^2(M,\tau) \cap \rL^0(M,\tau)_+$. For $x,y \in \cP$, define $d_\cP(x,y) = \| \log (x^{-1/2} y x^{-1/2})\|_2$. We prove in \S \ref{sec:P1}-\ref{sec:P2}:

\begin{thm}\label{thm:positivecone}
Let $(M,\tau)$ be as in Theorem \ref{thm:main}. Then $(\cP,d_\cP)$ is a complete CAT(0) metric space on which $\GL^2(M,\tau)$ acts transitively by isometries.
\end{thm}
These aforementioned properties allow us to apply the Karlsson-Margulis theorem, a special case of which is reproduced in \S \ref{sec:proof overview}, to obtain our result.  This extends work of Andruchow-Larotonda who previously studied the geometry of $\cP \cap M$ \cite{MR2254561}.

\begin{remark}
In general, neither of the conditions $\log |x| \in \rL^2(M,\tau)$ and $1-x \in \rL^2(M,\tau)$ implies the other. For example, suppose $(Y,\nu)$ is the unit interval with Lebesgue measure and $M=\rL^\infty(Y,\nu)$. If $z \in \rL^2(Y,\nu)$ is such that $z\ge 0$ and $z^2 \notin \rL^2(Y,\nu)$ then $x=\exp(z)$ satisfies $\log |x| \in \rL^2(M,\tau)$ but $1-x \notin \rL^2(M,\tau)$. On the other hand, the function $x(t)=e^{-t^{-1/2}}$ satisfies $1-x \in \rL^2(M,\tau)$ but $\log |x| \notin \rL^2(M,\tau)$. 
\end{remark}


\begin{example}[The abelian case]\label{abelian1}
Continuing with our running example, if $M=\rL^\infty(Y,\nu)$ then the above-mentioned inner product on $M$ is the restriction of the inner product on $\rL^2(Y,\nu)$ to $M$. Therefore, $\rL^2(M,\tau)$ is naturally isomorphic to $\rL^2(Y,\nu)$. An operator is affiliated with $M$ if and only if it is a multiplication operator of the form $m_\phi$ for some measurable $\phi:Y \to \C$. Such an operator $m_\phi$ is  $\tau$-measurable if and only if there is some $\eps>0$ such that 
$$\nu(\{y \in Y:~ |\phi(y)| > \eps\})<\infty.$$ 
The exponential map $\exp:\rL^2(Y,\nu) \to \GL^2(M,\tau)$ is a surjective homomorphism of groups (where we consider $\rL^2(Y,\nu)$ as an abelian group under addition). The kernel consists of all maps $\phi \in \rL^2(Y,\nu)$ with essential range in $2 \pi i \Z$. The restriction of $\exp$ to the real Hilbert space $\rL^2(Y,\nu; \R)$  is an isometry onto $(\cP,d_\cP)$.
\end{example}

\begin{example}[The full algebra case]
Suppose $(M,\tau)=(B(\cH),\tau_\cH)$ is as in \S \ref{intro-full}. Then $\cN \subset M$ is the algebra of Hilbert-Schmidt operators and $\rL^2(M,\tau)=\cN$. So $\GL^2(M,\tau)=\exp(\cN)$ consists of all operators of the form $\id + x$ where $x$ is Hilbert-Schmidt.

A subspace of $\rL^2(M,\tau)$ is essentially dense if and only if it equals $\rL^2(M,\tau)$. This is because every non-zero projection operator has trace at least 1 so if $\tau(\id-p)<1$ then $\id=p$. So $\rL^0(M,\tau)=B(\cH)$.
\end{example}

\begin{example}[The group case]
Suppose $(M,\tau)=(L(G),\tau)$ is as in \S \ref{intro-group vNa}. Then $\cN=M$, and the map $M\to \ell^{2}(G)$ given by $x\mapsto x(\delta_{1})$ extends to a unitary isomorphism $L^{2}(M,\tau)\cong \ell^{2}(G)$.  In this case, every $\xi\in \ell^{2}(G)$ defines a closed, densely-defined operator $\lambda(\xi)$ whose domain is $\dom(\lambda(\xi))=\{\eta\in \ell^{2}(G):\xi*\eta\in \ell^{2}(G)\}$ and so that $\lambda(\xi)(\eta)=\xi*\eta$ for $\eta\in \dom(\lambda(\xi))$. By \cite[Proposition 43.10]{ConwayOT}, the commutant of $M$ acting on $\ell^{2}(G)$ is the von Neumann algebra generated by the right regular representation $\rho\colon G\to \mathcal{U}(\ell^{2}(G))$ defined by $(\rho(g)\xi)(h)=\xi(hg)$. Since $\lambda(\xi)$ commutes with $\rho(g)$ for every $g\in G$, we know that $\lambda(\xi)$ is affiliated to $M$. Since the trace on $M$ is finite, we know that $\lambda(\ell^{2}(G))\subseteq L^{0}(M,\tau)$. In fact, we leave it as an exercise to the reader to verify that if we embed $L^{2}(M,\tau)\subseteq L^{0}(M,\tau)$ as above, then $\lambda(\ell^{2}(G))=L^{2}(M,\tau)$.
Recall that if $G$ is infinite, then $M$ is diffuse. By Proposition \ref{prop:facts about GL2/L2 diffuse case}, $L^{2}(M,\tau)$ contains unbounded operators. The statement that $L^{2}(M,\tau)$ contains unbounded operators is equivalent, by the closed graph theorem, to the statement that there are two $\ell^{2}(G)$ functions whose convolution is not in $\ell^{2}(G)$. For some infinite groups (e.g. those contain an element of infinite order), this is easy to see directly.

\end{example}




\subsubsection{Remarks on the limit operator}

\begin{remark}
Let $\|\cdot\|_\infty$ denote the operator norm. If the cocycle is uniformly bounded in operator norm (this means there is a constant $K$ such that $\|c(1,x)\|_\infty \le K$ for a.e. $x$) then $\|\La(x)\|_\infty \le K$ as well. Therefore, $\La(x) \in M$ for a.e. $x$.
\end{remark}

\begin{remark}
This theorem is a special case of a more general result (Theorem \ref{thm:main-general}) which removes the restriction of the cocycle to taking values in bounded operators.
\end{remark}

\begin{remark}
The reader might wonder whether a stronger form of convergence holds in the theorem above. Namely, whether convergence $\log \La(x)=\lim_{n\to\infty} \log \left([c(n,x)^*c(n,x)]^{1/2n}\right)$ occurs in operator norm. The answer is `no'. We provide an explicit example of this in \S \ref{sec:abelian} below with $M=\rL^\infty(Y,\nu)$. 
\end{remark}

\begin{conj}\label{conj:1}
Assume the hypotheses of Theorem \ref{thm:main}. If $(M,\tau)$ is finite then for a.e. $x$, $\log \left([c(n,x)^*c(n,x)]^{1/2n}\right)$ converges to $\log \La(x)$ almost uniformly in the sense of \cite{MR212581} (the equivalent notion of nearly everywhere convergence was first introduced in \cite[Defn 2.3]{MR54864}). This means that for every $\eps>0$ and for a.e. $x$, there exists a closed subspace $S(x) \subset \rL^2(M,\tau)$ such that the projection operator $p_{S(x)}$ satisfies $p_{S(x)} \in M$, $\tau(\id-p_{S(x)})<\eps$ and
$$\lim_{n\to\infty} \log \left([c(n,x)^*c(n,x)]^{1/2n}\right)p_{S(x)} = \log (\La(x))p_{S(x)}$$
where convergence is in operator norm.
\end{conj}

\begin{remark}
 If $M$ is diffuse, then $M$ has no nonzero compact operators, see Proposition \ref{prop: no nonzero compacts appendix} in the appendix. For example, this occurs when $M=L(G)$ for an infinite group $G$. 
This shows us that frequently (indeed for most cases of interest) $M$ does not have any nonzero compact operators, and so  the limiting operator $\Lambda$ is typically not compact. In fact, it follows from the proof of  Proposition \ref{prop:diffuse vNa TFAE} that if $M$ is diffuse and $\Lambda\ne 0,$ then it does not have a finite dimensional, nonzero eigenspace. 

\end{remark}

\subsubsection{Oseledets subspaces and Lyapunov distribution}

One of the main advantages of working with a tracial von Neumann algebra $(M,\tau)$ is that if $x \in M$ is normal (this means $xx^*=x^*x$) then $x$ has a spectral measure. If $M=\textrm{M}_n(\C)$ is the algebra of $n\times n$ complex matrices, then the spectral measure is the uniform probability measure on the eigenvalues of $x$ (with multiplicity). To define it more generally, recall that there is a projection-valued measure $E_x$ on the complex plane such that $x = \int \l ~\dee E_{x}(\l)$ \cite[Chapter IX, Theorem 2.2]{Conway}. The {\bf spectral measure of $x$} is the composition $\mu_x = \tau \circ E_x$. It is a positive measure with total mass equal to $\tau(\id)$ (where $\id$ is the identity operator). Moreover, if $p$ is any polynomial  then $\tau(p(x))=\int p~d\mu_x$.

\begin{example}[The abelian case]
If $M=\rL^\infty(Y,\nu)$ then every operator $\phi \in M$ is normal. The spectral measure of $\phi$ is its distribution $\mu_\phi$ defined by
$$\mu_\phi(R)=\nu(\{y \in Y:~ \phi(y) \in R\})$$
for all measurable regions $R \subset \C$.
\end{example}

\begin{example}[The full algebra case]
Suppose $(M,\tau)=(B(\cH),\tau_\cH)$. Then every normal Hilbert-Schmidt operator is unitarily diagonalizable. In particular, there is an orthonormal basis of eigenvectors. Therefore, the spectral measure of a normal Hilbert-Schmidt operator is discrete.
\end{example}

This definition of spectral measure extends to $x \in \rL^0(M,\tau)$. In the context of Theorem \ref{thm:main}, we define the {\bf Lyapunov distribution} to be the spectral measure $\mu_{\log\La(x)}$ of the log limit operator $\log\La(x)$. If $M=\textrm{M}_n(\C)$ is the algebra of $n\times n$ complex matrices and $\tau$ is the usual trace then this definition agrees with the previous definition.

To further justify this definition, we recall the notion of von Neumann dimension. If $S \subset \rL^2(M,\tau)$ is a closed subspace and the orthogonal projection operator $p_S$ lies in $M$ then the {\bf von Neumann dimension of $S$} is $\dim_M(S) = \tau(p_S)$. For example, the vN-dimension of $\rL^2(M,\tau)$ itself is $\tau(\id)$. This notion of dimension satisfies many desirable properties such as being additive under direct sums and continuous under increasing and decreasing limits \cite{Luck}.

\begin{example}[The abelian case]
If $M=\rL^\infty(Y,\nu)$ and if $p \in M$ is a projection operator then there is a measurable subset $Z \subset Y$ such that $p$ is the characteristic function $p=1_Z$ and the range of $p$ is the space of all $\rL^2$-functions with support in $Z$. The vN-dimension of this space is the measure $\nu(Z)$.
\end{example}

\begin{example}[The full algebra case]
Suppose $(M,\tau)=(B(\cH),\tau_\cH)$. Then a projection $p \in M$ is Hilbert-Schmidt if and only if its range is finite-dimensional. Moreover, the vN-dimension of a finite-dimensional subspace is its dimension.
\end{example}

Let
$$\mathcal{H}_{t}(x) =1_{(-\infty,t]}(\log\Lambda(x))(\rL^2(M,\tau)) \subset \rL^2(M,\tau) $$
where $1_{(-\infty,t]}(\log\Lambda(x))$ is defined via functional calculus. Alternatively, $\mathcal{H}_{t}(x)$ is the range of the projection $E_{\log \Lambda(x)}(-\infty, t]$. This is analogous to the Oseledets subspaces defined previously. 
The following theorem is proven is \S \ref{sec:invariance}:
\begin{thm}\label{thm:invariance}[Invariance principle]
With notation as above, for a.e. $x \in X$ and every $t \in [0,\infty)$,
$$c(1,x)\cH_t(x)=\cH_t(f(x)), \quad \mu_{\log \La(x)} = \mu_{\log \La(f(x))}.$$
\end{thm}
Furthermore, in \S \ref{sec:invariance} we also express $\cH_t$ as $\left\{\xi \in \rL^2(M,\tau):~ \liminf_{n\to\infty} \frac{1}{n} \log \|\La(x)^n\xi\|_2 \le t\right\}$ and prove a version of (\ref{eqn:Lyapunov growth rate}), roughly saying that growth rates along a cocycle matches those of the limit operator.



\subsubsection{Fuglede-Kadison determinants}

The {\bf Fuglede-Kadison determinant} of an arbitrary $x\in M$ is defined by
$$ \Delta(x) = \exp \left(\int_0^\infty \log(\l)~d\mu_{|x|}(\l) \right)$$
where $|x|= (x^*x)^{1/2}$ is a positive operator defined via the spectral calculus. The FK-determinant is multiplicative in the sense that $\Delta(ab)=\Delta(a)\Delta(b)$  \cite{MR0052696}. From \cite{MR2339369} it follows the definition of FK-determinant extends to operators in $\GL^2(M,\tau)$ and therefore can be applied to the limit operator $\La(x)$.

\begin{example}[The abelian case]
If $M=\rL^\infty(Y,\nu)$ then the FK-determinant of a function $\phi \in M$ is $\exp \int \log|\phi(y)|~\dee\nu(y)$.
\end{example}
\begin{example}[The full algebra case]
Suppose $(M,\tau)=(B(\cH),\tau_\cH)$.  If $\cH$ is finite-dimensional, then the FK-determinant is the absolute value of the usual determinant. If $\cH$ is infinite-dimensional, then the FK-determinant coincides with the absolute value of the Fredholm determinant on operators of the form $\id+a$ where $a\in M$ is trace-class. 
\end{example}

The following theorem is proven in \S \ref{sec:determinants}:
\begin{thm}\label{thm:determinants}
With notation as above, for a.e. $x \in X$, if $\tau$ is finite, then
$$\lim_{n\to\infty} \left(\Delta |c(n,x)|\right)^{1/n} = \Delta \La(x).$$
\end{thm}


\subsubsection{Growth rates}

Assume the notation of Theorem \ref{thm:main}.

\begin{defn}\label{D:growth1}
Given $\xi\in \rL^2(M,\tau)$, let $\Si(\xi)$ be the set of all sequences $(\xi_n)_n \subset \rL^2(M,\tau)$ such $\lim_{n\to\infty} \|\xi - \xi_n\|_2 = 0$. Define the {\bf upper and lower smooth growth rates of the system $(X,\mu,f,c)$ with respect to $\xi$ at $x\in X$} by
\begin{eqnarray*}
\underline{\Gr}( x | \xi) &=& \inf\left\{ \liminf_{n\to\infty} \| c(n,x) \xi_n\|_2^{1/n} :~ (\xi_n)_n \in \Si(\xi)  \right\}  \\
\overline{\Gr}( x | \xi) &=& \inf\left\{ \limsup_{n\to\infty} \| c(n,x) \xi_n\|_2^{1/n} :~ (\xi_n)_n \in \Si(\xi)\right\}.
\end{eqnarray*}



\end{defn}
The following theorem is proven in \S \ref{sec: growth rates}.
\begin{thm}\label{thm:main2}
Assume the hypotheses of Theorem \ref{thm:main}. Then for a.e. $x\in X$ and every $\xi \in \rL^2(M,\tau)$,
$$\underline{\Gr}(  x | \xi)  = \lim_{n\to\infty} \|\La(x)^n\xi\|_2^{1/n}= \overline{\Gr}( x | \xi).$$
\end{thm}

\begin{remark}
In \S \ref{sec:abelian} we give an explicit example in which a strict inequality
$$\liminf_{n\to\infty} \|c(n,x) \xi\|_2^{1/n} > \lim_{n\to\infty} \|\La(x)^n \xi\|_2^{1/n}$$
occurs. This explains why we take the infimum over all sequences in Definition \ref{D:growth1}. This infimum over sequences also appears in \cite[Definition 3.1]{MR2511586} which handles the special case in which $c(1,x)$ is a constant not depending on $x$ (see also \S \ref{S:single}).
\end{remark}

\begin{conj}\label{conj:2}
If $(M,\tau)$ is finite then Theorem \ref{thm:main2} can be strengthened to: for a.e. $x \in X$ there exists an essentially dense subspace $\cH_x \subset \rL^2(M,\tau)$ such that for every $\xi \in \cH_x$,
$$\lim_{n\to\infty} \|c(n,x) \xi\|_2^{1/n} = \lim_{n\to\infty} \|\La(x)^n \xi\|_2^{1/n}.$$
Essentially dense subspaces are reviewed in \S \ref{sec:essentially}.
\end{conj}

\begin{remark}
In \S \ref{sec:invariance}, we prove the conjecture with $\liminf$ in place of $\lim$. To be precise: for a.e. $x \in X$ there exists an essentially dense subspace $\cH_x \subset \rL^2(M,\tau)$ such that for every $\xi \in \cH_x$,
$$\liminf_{n\to\infty} \|c(n,x) \xi\|_2^{1/n} = \lim_{n\to\infty} \|\La(x)^n \xi\|_2^{1/n}.$$
\end{remark}




\subsection{The abelian case}

As in previous examples, suppose $M=\rL^\infty(Y,\nu)$. In \S \ref{sec:abelian}, we show that with this choice of $(M,\tau)$, Theorem \ref{thm:main}, along with Conjectures \ref{conj:1} and \ref{conj:2}, follows readily from Birkhoff's Pointwise Ergodic Theorem. We also provide explicit examples where the limit operator $\La(x)$ has continuous spectrum, where convergence to the limit operator does not occur in operator norm, and where there exist vectors $\xi$ satisfying the strict inequality
$$\liminf_{n\to\infty} \|c(n,x) \xi\|_2^{1/n} > \lim_{n\to\infty} \|\La(x)^n \xi\|_2^{1/n}.$$
The section \S \ref{sec:abelian} can be read independently of the rest of the paper.

\subsection{Powers of a single operator}\label{S:single}

As above, let $(M,\tau)$ be a finite von Neumann algebra and let $T \in M$. It is a famous open problem to determine whether $T$ admits a proper invariant subspace. The main results of \cite{MR2511586} show that the limit $\lim_{n\to\infty} |T^n|^{1/n} = \La$ exists in the Strong Operator Topology (SOT) and moreover, if $\cH_t=1_{[0,t]}(\Lambda)(\rL^2(M,\tau))$ then $\cH_t$ is an invariant subspace. The spectral measure of $\La$ is the same as the Brown measure of $T$ radially projected to the positive real axis. Moreover, if the Brown measure of $T$ is not a Dirac mass then there exists a proper invariant subspace.

Now suppose that $T$ has a bounded inverse $T^{-1} \in M$. Regardless of the dynamics, we may choose to define the cocycle $c$ by $c(n,x)=T^n$. Theorems \ref{thm:main} and \ref{thm:main2} then recover the main results of \cite{MR2511586} with the exception that our results say nothing of the Brown measure and they only apply to the invertible case. Our methods are completely different. In particular, we do not use \cite{MR2511586}.

\subsection{Proof overview}\label{sec:proof overview}
We will make use of a general Multiplicative Ergodic Theorem due to Karlsson-Margulis based on non-positive curvature (see also \cite{MR947327} which seems to be the first paper that develops this geometric approach). To accommodate their cocycle convention (which is different from ours), let us say that a measurable map  $\check{c}:\N \times X \to G$ is a {\bf reverse cocycle} if
$$\check{c}(n+m,x) = \check{c}(n,x)\check{c}(m, f^nx)$$
for any $n,m \in \N$ (where $G$ is a group).

The following is a special case of the Karlsson-Margulis Theorem.
\begin{thm}[\cite{MR1729880}]\label{thm:km}
Let $(X,\mu)$ be a standard probability space, $f:X \to X$ an ergodic measure-preserving invertible transformation, $(Y,d)$ a complete CAT(0) space, $y_0 \in Y$ and $\check{c}:\N \times X \to \Isom(Y,d)$ a measurable reverse cocycle taking values in the isometry group of $(Y,d)$, where measurable means with respect to the compact-open topology on $\Isom(Y,d)$. Assume that
$$\int_X d(y_0, \check{c}(1,x)y_0) ~d\mu(x) < \infty.$$
Then for almost every $x\in X$, the following limit exists:
$$\lim_{n\to\infty} \frac{d(y_0, \check{c}(n,x)y_0)}{n} = D.$$
Moreover, if $D>0$ then for almost every $x$ there exists a unique unit-speed geodesic ray $\g(\cdot, x)$ in $Y$ starting at $y_0$ such that
$$\lim_{n\to\infty} \frac{1}{n} d(\gamma(Dn,x), \check{c}(n,x)y_0) =0.$$
\end{thm}

As remarked in \cite{MR1729880}, this result implies the classical MET as follows. Let $P(n,\R)$ be the space of positive definite $n\times n$ matrices. Then $\GL(n,\R)$ acts on $P(n,\R)$ by $g.p := g p g^*$. The tangent space to $p \in P(n,\R)$, denoted $T_p(P(n,\R))$, is naturally identified with $S(n,\R)$, the space of $n\times n$ real symmetric matrices. Define an inner product on $T_p(P(n,\R))$ by
$$\langle x,y \rangle_p := \textrm{trace}(p^{-1}xp^{-1}y).$$
This gives a complete Riemannian metric on $P(n,\R)$. All sectional curvatures are non-positive and so $P(n,\R)$ is CAT(0). Moreover the $\GL(n,\R)$ action is isometric and transitive. Every geodesic ray from $\id$ (the identity matrix) has the form $t \mapsto \exp(tx)$ for $x \in S(n,\R)$.

Substitute $Y=P(n,\R)$ and $y_0=\id$ (the identity matrix) in the Karlsson-Margulis Theorem to obtain the classical multiplicative ergodic theorem.

Our proof of Theorem \ref{thm:main} follows in a similar way from the Karlsson-Margulis Theorem. In \cite{MR2254561}, Andruchow and Larotonda construct a Riemannian metric on the positive cone $\cP^\infty(M)$ of a finite von Neumann algebra. They prove that it is non-positively curved. We go over the needed facts from their construction in \S \ref{sec:symmetric}.

However, $\cP^\infty(M)$ is not metrically complete. We prove that its metric completion can naturally be identified with $\cP$, as mentioned earlier in \S \ref{sec: intro regular rep}, and $\GL^2(M,\tau)$ acts transitively and is a subgroup of the isometry group of $\cP$. This partially answers a question raised in \cite[Remark 3.21]{MR2643829} which asks to identify the metric completion of the space of positive definite operators $\cP^\infty(M)$ with respect to the metric $d_p(x,y) = \|\log(x^{1/2}y^{-1}x^{1/2})\|_p$ $(1\le p <\infty)$. We obtain a characterization in the special case $p=2$. It is possible that our proof can be modified to handle the general case; we did not attempt it.


\subsection{Organization}

\begin{itemize}
\item \S \ref{sec:abelian} proves a number of results about the special case in which $M$ is abelian. This is not needed for the rest of the paper. It is included for illustration only. 

\item \S \ref{sec:prelim} provides necessary background on spectral measures, the regular representation, the algebra of affiliated operators, etc.

\item \S \ref{sec:GL2} proves that $\GL^2(M,\tau)$ is a group.

\item \S \ref{sec:positive} proves Theorem \ref{thm:positivecone} on the geometry of $\cP$. 

\item \S \ref{sec:proofs of main results} proves the main Theorem \ref{thm:main} and the theorems on growth rates, determinants and Oseledets subspaces.

\item Appendix section \S \ref{sec:diffuse} has general facts about diffuse von Neumann algebras, to illustrate the main theorem.

\item Appendix section \S \ref{sec:examples-more} provides more examples of tracial von Neumann algebras.

\item Appendix section \S \ref{sec:glossary} is a glossary of terms used in the paper.

\end{itemize}

\noindent {\bf Acknowledgements}. L. Bowen would like thank IPAM and UCLA for their hospitality. The initial ideas for this projects were obtained while L. Bowen was attending the Quantitative Linear Algebra semester at IPAM.

\section{The abelian case}\label{sec:abelian}

As in \S \ref{intro-abelian}, let $M=\rL^\infty(Y,\nu)$ and define the trace $\tau$ on $M$ by $\tau(\phi)= \int \phi~\dee\nu$. This section studies the MET under the hypothesis that the cocycle $c$ takes values in $M$. It serves as motivation and can be read independently of the rest of the paper.

This special case might seem trivial and indeed, we will see that the conclusions of Theorem \ref{thm:main} are implied by the Pointwise Ergodic Theorem. However, there are curious features not present in previous versions of the MET. Below we will give examples in which $\La(x)$ has continuous spectrum and examples where $|c(n,x)|^{1/n}$ converges in $\rL^2$-norm to $\La(x)$ but not in operator norm. We will also show that growth rates do not necessarily exist for every vector, but do exist for an essentially dense subspace of vectors.

\subsection{Theorem \ref{thm:main} from the Pointwise Ergodic Theorem}\label{S:mean}

\begin{thm}
Assume the hypotheses of Theorem \ref{thm:main}. In addition, let $(Y,\nu)$ be a standard probability measure space, $M=\rL^\infty(Y,\nu)$ and let the trace $\tau$ be given by $\tau(\phi)=\int \phi~\dee\nu$ ($\phi \in M_+$). Also assume that the cocycle is uniformly bounded: $\exists R>0$ such that
$$R^{-1} \le |c(1,x)(y)| \le R$$
for a.e. $(x,y)$. Then the conclusion of Theorem \ref{thm:main} follows from the Pointwise Ergodic Theorem. 
\end{thm}

\begin{proof}
Define
\begin{eqnarray*}
F:X \times Y \to X\times Y&& \quad F(x,y)=(f(x),y), \\
\phi \in \rL^1(X\times Y, \mu \times \nu)&& \quad \phi(x,y)= \log |c(1,x)(y)|, \\
 A_n(x,y) \in \rL^1(X\times Y, \mu \times \nu) &&\quad A_n(x,y) = \frac{1}{n} \sum_{k=0}^{n-1} \phi(F^k(x,y)) = \frac{1}{n} \log |c(n,x)(y)|\\
 A_n(x) \in \rL^1(Y,\nu) && \quad A_n(x)(y) = A_n(x,y).
\end{eqnarray*}


 Because we  assume $\int \|\log(|c(1,x)|)\|_2~d\mu(x)<\infty$, it follows that $\phi \in \rL^1(X\times Y, \mu \times \nu)$ as claimed above.

The first conclusion of Theorem \ref{thm:main} is: for a.e. $x \in X$, $\|A_n(x)\|_2$ converges as $n\to\infty$. It is easier to work with the $\rL^1$-norm in place of the $\rL^2$-norm. Let $\phi^+(x,y)=\max(\phi(x,y),0)$ and $\phi^-(x,y)=\max(-\phi(x,y),0)$. Then both $\phi^+, \phi^-$ are in $\rL^1$ and
$$\|A_n(x)\|_1 = \int |A_n(x,y)| ~\dee\nu(y) =  \left|\frac{1}{n} \sum_{k=0}^{n-1} \int \phi^+(F^k(x,y))~\dee\nu(y) - \frac{1}{n} \sum_{k=0}^{n-1} \int \phi^-(F^k(x,y))~\dee\nu(y)\right|.$$
The Pointwise Ergodic Theorem applied to $x\mapsto \int \phi^{+}(x,y)~\dee\nu(y)$ implies that the averages $\frac{1}{n} \sum_{k=0}^{n-1} \int \phi^+(F^k(x,y))~\dee\nu(y)$ converge for a.e. $x$ as $n\to\infty$. Since the same is true with $\phi^+$ replaced with $\phi^-$, it follows that $\|A_n(x)\|_1$ converges for a.e. $x$ as $n\to\infty$.  Because $(Y,\nu)$ is a finite measure space and $\phi(x,y)$ is essentially bounded, $\rL^1$-convergence implies $\rL^2$-convergence. So for a.e. $x \in X$, $\|A_n(x)\|_2$ converges as $n\to\infty$.

The Pointwise Ergodic Theorem implies $A_n(x,y)$ converges for a.e. $(x,y)$ as $n\to\infty$. So Fubini's Theorem implies that: for a.e. $x$, $A_n(x)$ converges pointwise a.e. as $n\to\infty$. Scheffe's Lemma now implies that for a.e. $x$, $A_n(x)$ converges in $\rL^2(Y,\nu)$. This proves the last conclusion of Theorem \ref{thm:main}.

Let $\log \La(x)$ denote the limit of $A_n(x)$. As explained in Example \ref{abelian1},
$$ \frac{1}{n} d_\cP(|c(n,x)|, \La(x)^n) = \frac{1}{n}\| \log (|c(n,x)|) - \log(\La(x)^n)\|_2 = \left\| \frac{1}{n}\log (|c(n,x)|) - \log(\La(x))\right\|_2.$$
Thus $\frac{1}{n} d_\cP(|c(n,x)|, \La(x)^n) \to 0$ as $n\to\infty$.

Similarly,
$$d_\cP(|c(n,x)|^{1/n}, \La(x))= \left\|\frac{1}{n}\log (|c(n,x)|) - \log(\La(x))\right\|_2.$$
So $|c(n,x)|^{1/n}$ converges to $\La(x)$ in $(\cP,d_\cP)$ as $n\to\infty$ (for a.e. $x$).

To prove $|c(n,x)|^{1/n}$ converges to $\La(x)$ in measure, it suffices to show: for every $\eps>0$
$$\nu\left(\left\{y \in Y:~ \big| |c(n,x)(y)|^{1/n} - \La(x)(y)\big| > \eps\right\}\right)$$
tends to zero as $n\to\infty$ (for a.e. $x$). This is implied by the fact that $n^{-1}\log|c(n,x)|$ converges to $\log \La(x)$ in $\rL^2(Y,\nu)$ for a.e. $x$.


\end{proof}

Similarly, if $M = \textrm{M}_n(\C) \otimes \rL^\infty(Y,\nu)$ where $\textrm{M}_n(\C)$ denotes the algebra of $n\times n$ complex matrices, then the non-ergodic version of the classical Multiplicative Ergodic Theorem implies the conclusions of Theorem \ref{thm:main}.

\subsection{Examples with continuous spectrum}
This example is almost trivial. Let $\psi \in \rL^\infty(Y,\nu)$ be such that $\log |\psi| \in \rL^2(Y,\nu)$. Define $c(n,x)= \psi^n$. Then the limit operator satisfies $\La(x)=|\psi|$ for a.e. $x$ and the spectral measure of $\La$ is the distribution of $|\psi|$. In particular, if $|\psi|$ has continuous distribution then $\La(x)$ has continuous spectrum. 

\subsection{Almost uniform convergence and growth rates}

In this subsection, we prove Conjectures  \ref{conj:1} and \ref{conj:2} in the special case $M=\rL^\infty(Y,\nu)$ and $(Y,\nu)$ is a probability space.
\begin{thm}\label{thm:abelian-strong-growth}
Assume hypotheses as in Theorem \ref{thm:main}.  In addition, let $(Y,\nu)$ be a standard probability space, $M=\rL^\infty(Y,\nu)$ and let the trace $\tau$ be given by $\tau(\phi)=\int \phi~\dee\nu$. In this setting, Conjecture \ref{conj:1} is true.


\end{thm}

\begin{proof}
Define $F,\phi$ and $A_n$ as in \S \ref{S:mean}. By the Pointwise Ergodic Theorem, $A_n(x,y)$ converges to $\log \La(x)(y)$ for a.e. $(x,y)$.  By Fubini's Theorem, there exists a subset $X' \subset X$ with full measure such that for a.e. $x \in X'$, $A_n(x)$ converges pointwise a.e. (as $n\to\infty$) to $\log \La(x)$. Let $\eps>0$. By Egorov's Theorem, for every $x \in X'$ there exists a measurable subset $Z(x) \subset Y$ with $\nu(Z(x)) > 1-\eps$ such that $A_n(x)$ converges uniformly to $\log \La(x)$ on $Z(x)$.

Let $S(x) \subset \rL^2(Y,\nu)$ be the closed subspace of functions that equal zero off of $Z(x)$. The projection operator $p_{S(x)}$ is identified with the characteristic function $1_{Z(x)} \in \rL^\infty(Y,\nu)$. Moreover, $\tau(\id-p_{S(x)}) = \nu(Y \setminus Z(x)) < \eps$. Because $A_n(x)=n^{-1} \log |c(n,x)|$ converges uniformly to $\log \La(x)$ on $Z(x)$, it follows that
 $$\lim_{n\to\infty} n^{-1} \log |c(n,x)| p_{S(x)} = \log(\La(x))p_{S(x)}$$
 in operator norm.

\end{proof}

\begin{prop}
We assume the same hypotheses as Theorem \ref{thm:abelian-strong-growth}. In this setting, Conjecture \ref{conj:2} is true.
\end{prop}

\begin{proof}

Let $\xi \in \rL^2(Y,\nu)$. We first prove $\lim_{n\to\infty} \| \La(x)^n \xi \|^{1/n}_2 = \| \La(x) 1_{support(\xi)} \|_\infty$.

Without loss of generality, we may assume $\|\xi\|_2 = 1$. It is a standard exercise that $\lim_{n\to\infty} \|\phi\|_n = \|\phi\|_\infty$ for $\phi \in \rL^\infty$. So
$\|\La(x)^n\xi\|^{1/n}_2$ tends to $\|\La(x)\|_{\rL^{\infty}(Y, |\xi|^{2} \dee \nu)}$ as $n\to\infty$. The latter is  the same as $\| \La(x) 1_{support(\xi)} \|_{\rL^\infty(Y,\nu)}$.

By the proof of Theorem \ref{thm:abelian-strong-growth}, for every $r \in \N$ and a.e. $x$, there exists a measurable subset $Z_r(x) \subset Y$ such that $\nu(Z_r(x))>1-1/r$ and $n^{-1} \log |c(n,x)\resto Z_r(x)|$ converges uniformly to $\log \La(x) \resto Z_r(x)$ as $n\to\infty$ (where $\resto$ means ``restricted to'').

Let $S_r(x) \subset \rL^2(Y,\nu)$ be the subspace of vectors $\xi$ such that $\xi(y)=0$ for a.e. $y \in Y \setminus Z_r(x)$. Let $S(x) = \cup_{r\in \N} S_r(x) \subset \rL^2(Y,\nu)$. Because $\nu(Z_r(x))>1-1/r$ for all $r$, $S(x)$ is essentially dense.

Let $\xi \in S_r(x)$. Then
\begin{eqnarray*}
\| c(n,x) \xi \|_2 &\le& \| \La(x)^n \xi\|_2  \left\| \left(y \in Z_r(x) \mapsto \frac{|c(n,x)(y)|}{\La(x)^n(y)} \right) \right\|_{\rL^\infty(Z_r(x),\nu)}.
\end{eqnarray*}
Since $n^{-1} \log |c(n,x)\resto Z_r(x)|$ converges to $\log \La(x)$ uniformly on $Z_r(x)$, this implies
$$\limsup_{n\to\infty} \|c(n,x ) \xi \|^{1/n}_2 \le \lim_{n\to\infty} \|\La(x )^n \xi \|^{1/n}_2.$$
Similarly,
\begin{eqnarray*}
 \| \La(x)^n \xi\|_2&\le& \| c(n,x) \xi \|_2  \left\| \left(y \in Z_r(x) \mapsto \frac{\La(x)^n(y)}{|c(n,x)(y)|} \right) \right\|_{\rL^\infty(Z_r(x),\nu)}.
\end{eqnarray*}
So $\liminf_{n\to\infty} \|c(n,x ) \xi \|^{1/n}_2 \ge \lim_{n\to\infty} \|\La(x )^n \xi \|^{1/n}_2.$ This proves 
$$\lim_{n\to\infty} \| c(n,x) \xi \|^{1/n}_2 = \lim_{n\to\infty} \| \La(x)^n \xi \|^{1/n}_2 = \| \La(x) 1_{support(\xi)} \|_\infty$$
for all vectors $\xi \in S_r(x)$. Because $r\in \N$ is arbitrary, the above limits hold for all $\xi \in S(x)$. Because $S(x)$ is essentially dense, this implies Conjecture \ref{conj:2}.

\end{proof}

\begin{remark}
The same result holds if $M = \textrm{M}_n(\C)  \otimes \rL^\infty(Y,\nu)$ with essentially the same proof. One needs only use the non-ergodic version of Oseledet's Multiplicative Ergodic Theorem instead of Birkhoff's Pointwise Ergodic Theorem.
\end{remark}

\subsection{A counterexample}


\begin{thm}
There exist standard probability spaces $(X,\mu), (Y,\nu)$, an ergodic pmp invertible transformation $f:X \to X$, a measurable cocycle $c:\Z \times X \to M=\rL^\infty(Y,\nu)$  satisfying the hypotheses of Theorem \ref{thm:main} and a vector $\xi \in \rL^2(Y,\nu)$ such that
\begin{eqnarray*}
\lim_{n\to\infty} \|c(n,x ) \xi \|^{1/n}_{\rL^2(Y,\nu)} = \lim_{n\to\infty} \|c(n,x ) \|^{1/n}_{\rL^2(Y,\nu)} > \lim_{n\to\infty} \|\La(x )^n \xi \|^{1/n}_{\rL^2(Y,\nu)} = \lim_{n\to\infty} \|\La(x )^n \|^{1/n}_{\rL^2(Y,\nu)}.
\end{eqnarray*}
Moreover, we can choose the cocycle so that $\|c(1,x)\|_{\infty} \le C$ for some constant $C$ and a.e. $x$. Moreover, $n^{-1}\log |c(n,x)|$ does not converge to $\log \La(x)$ in operator norm (for a.e. $x$).
\end{thm}

\begin{proof}
Let $X=\Z_2$ be the compact group of 2-adic integers. An element of $\Z_2$ is written as a formal sum $x = \sum_{i=0}^\infty x_i 2^i$ with $x_i \in \{0,1\}$ and the usual multiplication and addition rules. Let $\mu$ be the Haar probability measure on $X$. There is a bijection between $X$ and $\{0,1\}^{\N \cup \{0\}}$ given by $x \mapsto (x_0,x_1,\ldots)$. This bijection maps the measure $\mu$ to the $(\N \cup \{0\})$-th power of the uniform measure on $\{0,1\}$.

Define $f:X \to X$ by $f(x)=x+1$. It is well-known that a translation on a compact abelian group is ergodic if and only if every orbit is dense. Thus $f$ is an ergodic measure-preserving transformation. Alternatively, $f$ is the standard odometer which is well-known to be ergodic.

Let $(Y,\nu)$ be a probability space that is isomorphic to the unit interval with Lebesgue measure. Let $Y= \sqcup_{n=1}^\infty Y_n$ be a partition of $Y$ into positive measure subsets. We will choose the partition more carefully later. Define the cocycle $c:\Z \times X \to \rL^\infty(Y,\nu)$ by
\begin{displaymath}
c(1,x)(y) = \left\{ \begin{array}{cl}
1 & \textrm{ if } y \in Y_m \textrm{ for some $m$ and } x_m = 0\\
2 & \textrm{ otherwise } \end{array}\right.\end{displaymath}
This extends to a cocycle via $c(n,x)=c(1,f^{n-1}x)\cdots c(1,f(x))c(1,x)$.

For every $y \in Y$,
$$\int \log c(1,x)(y)~d\mu(x)  = (1/2)\log(2).$$
Since $f$ is ergodic, it follows that the limit operator $\La(x)$ defined by $\log \La(x) = \lim_{n\to\infty} n^{-1} \log c(n,x)$ (where convergence is in $\rL^2$ and pointwise a.e.) is the constant function $\La(x)=\sqrt{2}$ for a.e. $x$.



For $n,m \in \N$, let
$$S_{n,m} = \{x \in X:~x_{m} = 1, x_n = 0\}.$$
We claim that if $x \in S_{n,m}$ and $n<m$, $0 \le l \le 2^n +1 $, then $c(l,x)(y) = 2^l~\forall y \in Y_m$. Indeed, note that if $x \in S_{n,m}$ the smallest $k$ such that $c(1,f^kx)(y) = 1$ must be when $k + \sum_{i=0}^m x_i 2^i = 2^{m+1}$.  But
$$2^{m}\le \sum_{i=0}^m x_i 2^i \le 2^{m+1} - 2^n-1.$$
Therefore, $c(1,f^kx)(y) = c(1,x +k)(y)= 2$ for all $0\le k \le 2^n$ and
$$c(l,x)  = c(1,f^{l-1}x)\cdots c(1,x)=2^l.$$

Note that $\mu(S_{n,m})=1/4$. Moreover, if $n_1 \ne n_2$ and $m_1 \ne m_2$ then $S_{n_1,m_1}$ and $S_{n_2,m_2}$ are independent events.  It follows that if $T_n = S_{n,n+10}$ then the events $\{T_n\}_{n=1}^\infty$ are jointly independent and, by Borel-Cantelli, a.e. $x$ is contained in infinitely many of the sets $T_n$.

If $x\in T_n$ then for $0 \le l \le 2^n + 1$
$$\|c(l,x )1_Y\|^2_{\rL^2(Y,\nu)} \ge \|c(l,x)1_{Y_{n+10}}\|^2_{\rL^2(Y,\nu)} \ge \nu(Y_{n+10}) 2^{2l}.$$

We could choose the subsets $\{Y_m\}$ so that $\nu(Y_m) \ge C m^{-2}$ for some constant $C$. With this choice and $x \in T_n$,
$$\|c(l,x )1_Y\|^{2}_{\rL^2(Y,\nu)}\ge C 2^{2l}/(n+10)^{-2}.$$
Since a.e. $x$ is contained in infinitely many $T_n$'s it follows that
$$\lim_{n\to\infty} \|c(n,x )1_Y\|^{1/n}_{\rL^2(Y,\nu)} = 2.$$
On the other hand,
$$\lim_{n\to\infty} \|\La(x )^n 1_Y\|^{1/n}_{\rL^2(Y,\nu)} = \sqrt{2}.$$
This proves the theorem with $\xi=1_Y$. By Theorem \ref{thm:abelian-strong-growth}, $|c(n,x)|^{1/n}$ does not converge to $\La(x)$ in operator norm (for a.e. $x$).
\end{proof}

\begin{remark}
The essential phenomena behind this counterexample is that there is no uniform rate of convergence in the Pointwise Ergodic Theorem. Precisely, while $\frac{1}{n}\sum_{k=0}^{n-1} \log c(1,f^kx)(y)$ converges to $\log(2)/2$ for every $y$ and a.e. $x$, the convergence is not uniform in $y$.
\end{remark}

\section{Preliminaries}\label{sec:prelim}
Throughout these notes, by a \emph{tracial von Neumann algebra} we mean a pair $(M,\tau)$ where $M$ is a von Neumann algebra, and $\tau$ is a faithful, normal, tracial, state.  By a \emph{semi-finite von Neumann algebra} we mean a pair $(M,\tau)$ where $M$ is a von Neumann algebra and $\tau$ is a faithful, normal, and semi-finite trace. Throughout, we assume $M$ is a sub-algebra of the algebra $B(\cH)$ of all bounded operators on a separable Hilbert space $\cH$. This implies $(M,\tau)$ has separable pre-dual. We will consider many constructions that depend on the choice of trace $\tau$ but we suppress this dependence from the notation. 

\subsection{Spectral measures}

Suppose $x$ is a  (bounded or unbounded) self-adjoint operator on $\cH$. By the Spectral Theorem (\cite[Theorem VIII.6]{MR751959}), there exists a projection valued measure $E_x$ on the real line such that
$$x = \int \l ~\dee E_x(\l).$$
The support of $E_x$ is contained in the spectrum of $x$. The projections of the form $E_x(R)$ (for Borel sets $R \subset \R$) are the {\bf spectral projections of $x$}.
If $f: \R \to \R$ is Borel then $f(x)$ is a self-adjoint operator on $\cH$ defined by
$$f(x) := \int f(\l) ~\dee E_x(\l).$$
In the case of unbounded $x$, $f(x)$ has the same domain as $x$. The {\bf absolute value} of $x$ is defined by $|x| = (x^*x)^{1/2} = \int \sqrt{\lambda} ~\dee E_{x^*x}(\l)$ and is equal to $\int |\l| ~\dee E_x(\l)$.

If $x$ is such that all of its spectral projections lie in the von Neumann algebra $M$, then the composition $\tau \circ E_x$ is a Borel probability measure on $\C$ called the {\bf spectral measure of $x$} and denoted by $\mu_x$.  In particular, if $x \in M$ then $\mu_x$ is well-defined.

\begin{example}[The abelian case]
If $M=\rL^\infty(Y,\nu)$ (as in \S \ref{intro-abelian}), then an operator $\phi \in M$ is self-adjoint if and only if it is real-valued. The projection-valued measure $E_\phi$ satisfies: $E_\phi(R)$ is the projection onto the subspace of $\rL^2$-functions with support in $\phi^{-1}(R)$ (for Borel $R \subset \R$). Moreover, $\mu_\phi=\phi_*\nu$ is the distribution of $\phi$.
\end{example}

\subsection{Polar decomposition}

We will frequently have to use the polar decomposition, see \cite[Theorem VIII.32]{MR751959}. We restate it here.

\begin{prop}\label{P:polar decomposition}
Let $x$ be a closed densely defined operator on $\cH$. Then there is a positive self-adjoint operator $|x|$ with $\dom(|x|)=\dom(x)$ and a partial isometry $u$ with initial space $\ker(x)^\perp$ and final space $\overline{\im(x)}$ so that $x=u|x|$ (where $\im(x)$ denotes the image of $x$). Moreover $|x|$ and $u$ are uniquely determined by these properties together with the additional condition $\ker(|x|)=\ker(x)$.
\end{prop}
The expression $x=u|x|$ is called the {\bf polar decomposition} of $x$.


\subsection{The regular representation}
For the remainder of this section, fix a semi-finite von Neumann algebra $(M,\tau).$
Recall from the introduction that $\rL^2(M,\tau)$ is the Hilbert space completion of $\cN = \{x \in M:~ \tau(x^*x)<\infty\}$ with respect to the inner product defined on $\cN$ by
$$\langle x, y \rangle = \tau(x^*y).$$
Let $\|x\|_2 = \langle x,x \rangle^{1/2}$ and $\|x\|_\infty$ be the operator norm of $x$ (as an operator on $\cH$).

For any $x,y \in M$,
$$\|xy\|_2 \le \|x\|_\infty \|y\|_2 \textrm{ and } \|xy\|_2 \le \|x\|_2 \|y\|_\infty.$$
(e.g., \cite[V.2, equation (8)]{MR1873025}). Therefore, the operator $L_x:\cN \to \cN$ defined by $L_x(y)=xy$ admits a unique continuous extension from $\rL^2(M,\tau)$ to itself. Moreover, the operator norm of $L_x$ is bounded by $\|x\|_\infty$. In fact, they are equal.
This follows, for example, from \cite[Corollary I.5.4 and Theorem V.2.22]{MR1873025}.
Similarly, the map $R_x:\cN \to \cN$ defined by $R_x(y)=yx$ admits a unique continuous extension to $\rL^2(M,\tau)$ and the operator norm of $R_x$ is $\|x\|_\infty$.

We will identify $M$ with its image $\{L_x:~x \in M\}$ (viewed as a sub-algebra of the algebra of bounded operators on $\rL^2(M,\tau)$).

\begin{remark}
In much of this paper, we think of $M$ as acting on $\rL^2(M,\tau)$ rather than on $\cH$. There is a difference. For example, the strong operator topologies on $M$ induced from the inclusions $M \subset B(\cH)$ and $M \subset B(\rL^2(M,\tau))$ are different in general \cite[Exercise 1.3]{anantharaman-popa}.  However, in our case, $\cH$ is a left-$M$-module. By \cite[Lemma II.2.5 and Corollary III.3.10]{MR1873025}, this implies that the strong topologies induced from the inclusions $M\subseteq B(\cH)$,$M\subseteq B(\rL^{2}(M,\tau))$ agree on any norm bounded subset of $M$. Assuming $\cH$ is separable, there is a nice structure theory: left-$M$-modules are isomorphic to submodules of multiplies of $\rL^2(M,\tau)$ \cite[Chapter 8]{anantharaman-popa}.


\end{remark}

\subsection{The algebra of affiliated operators}\label{sec:affiliated}

\begin{defn}
The {\bf commutant} of $M$, denoted $M'$, is the algebra of bounded operators $y$ on $\rL^2(M,\tau)$ such that $xy=yx$ for all $x\in M$. Von Neumann's Double Commutant Theorem implies $M=(M')'$.
\end{defn}

\begin{defn}
An unbounded operator $x$ on $\rL^2(M,\tau)$ is {\bf affiliated} with $M$ if for every unitary $u \in M'$, $xu=ux$.  By \cite[Chapter I.1 Exercise 10]{MR641217} or \cite[Proposition 7.2.3]{anantharaman-popa}, if $x$ is a closed densely defined operator and $x=u|x|$ is its polar decomposition, then $x$ is affiliated with $M$ if and only if $u$ and the spectral projections of $|x|$ are in $M$.
\end{defn}

\begin{defn}
A subspace is $V \subset \rL^2(M,\tau)$ is {\bf essentially dense} if for every $\eps>0$ there exists a projection $p \in M$ such that $\tau(\id-p) < \eps$ and $p\rL^2(M,\tau) \subset V$.  Essentially dense subspaces are reviewed in \S \ref{sec:essentially}.
\end{defn}

\begin{defn}[$\textrm{L}^0(M,\tau)$]
 A closed, densely-defined operator affiliated with $(M,\tau)$ is {\bf $\tau$-measurable} if its domain of definition is essentially dense.  Note that when $(M,\tau)$ is finite, then all affiliated operators are $\tau$-measurable. Let $\rL^0(M,\tau)$ denote the set of $\tau$-measurable operators.  By \cite[Theorems IX.2.2, IX.2.5]{TakesakiII}, $\rL^0(M,\tau)$ is closed under adjoint, addition and multiplication and is a $\ast$-algebra under these operations.
\end{defn}

\subsubsection{Domains}
For $x\in \rL^0(M,\tau)$ we write $\dom(x) \subset \rL^2(M,\tau)$ for its domain. We remark now that for $a,b\in \rL^0(M,\tau)$ the sum $a+b$ is defined as  the \emph{closure} of the operator $T$ with $\dom(T)=\dom(a)\cap \dom(b)$ and with $T\xi=a\xi+b\xi$ for $\xi\in\dom(T).$  Similarly $ab$ is defined as the \emph{closure} of the operator $T$ with domain $b^{-1}(\dom(a))\cap \dom(b)$ and $T\xi=a(b\xi)$ for $\xi\in\dom(T).$ Thus, for example, the domain of $ab$ is often larger than $b^{-1}(\dom(a))\cap \dom(b).$ This will occasionally cause us some headaches, and we will try to remark  when it actually presents an issue. Regardless, this paragraph should be taken as a blanket warning that $ab$ is not literally defined to be the composition, and $a+b$ is not the literal sum.

\subsubsection{$\textrm{L}^2(M,\tau) \subset \textrm{L}^0(M,\tau)$}

We can include $\rL^2(M,\tau)$ in $\rL^0(M,\tau)$ as follows. For $x \in \rL^2(M,\tau)$ and $y \in M$, define $L^0_x(y) = R_y(x)=xy$. Then $L^0_x$ is closable but not bounded in general. Let $L_x$ denote the closure of $L^0_x$. The map $x \mapsto L_x$ defines a linear injection from $\rL^2(M,\tau)$ into $\rL^0(M,\tau)$. By abuse of notation, we will identify $\rL^2(M,\tau)$ with its image in $\rL^0(M,\tau)$.

For $x\in \rL^0(M,\tau)$ we set $|x|=(x^{*}x)^{1/2}$ and
\[\|x\|_{2}=\left(\int t^{2}\,d\mu_{|x|}(t)\right)^{1/2}\in [0,\infty].\]
Then $\rL^2(M,\tau)$ is identified with the set of all $x\in \rL^0(M,\tau)$ which have $\|x\|_{2}<\infty.$

While $\rL^2(M,\tau)$ is a subspace of $\rL^0(M,\tau)$, it is not a sub-algebra in general.  For instance, it is not closed under products if $M$ is diffuse (see 
Proposition \ref{prop:facts about GL2/L2 diffuse case}).

\subsubsection{Extending the adjoint}

The anti-linear map $x \mapsto x^*$ on $M$ uniquely extends to an anti-linear isometry $J:\rL^2(M,\tau) \to \rL^2(M,\tau)$. By \cite[Proposition 7.3.3]{anantharaman-popa}, if $x \in \rL^2(M,\tau)$ then the following are equivalent: (1) $x$ is self-adjoint, (2) $Jx=x$, (3) $x$ is in the $\rL^2$-closure of $M_{sa}=\{y \in M:~y^*=y\}$. Let $\rL^2(M,\tau)_{sa}=\{x\in \rL^2(M,\tau):~Jx=x\}$.

\subsubsection{Invertible affiliated operators}
We say an operator $x \in \rL^0(M,\tau)$ is {\bf invertible} if there exists an operator $y \in \rL^0(M,\tau)$ such that $xy=yx=\id$ where, following our abuse of notation, $xy$ and $yx$ denote the closures of the compositions of the operators $x$ and $y$. In this case we write $y=x^{-1}$. Let $\rL^0(M,\tau)^{\times} \subset \rL^0(M,\tau)$ be the set of invertible affiliated operators $x$.

\begin{lem}\label{L:invertible}
If $(M,\tau)$ is semi-finite and $x \in \rL^0(M,\tau)^{\times}$ has polar decomposition $x=u|x|$ then $u$ is unitary, $|x| \in \rL^0(M,\tau)^{\times}$ and $x^* \in \rL^0(M,\tau)^\times$ with $(x^*)^{-1}=(x^{-1})^*$. If $(M,\tau)$ is finite then $x \in \rL^0(M,\tau)$ is invertible if and only if it is injective.
\end{lem}

\begin{proof}
Because $u$ is a partial isometry, $u^*u$ is the orthogonal projection onto $\ker(u)^\perp$. If $x$ is invertible, then $u$ is injective, so $u^*u =\id$. Similarly, $uu^*$ is the projection onto the closure of the image of $x$. So if $x$ is invertible then $uu^*=\id$. This proves $u$ is unitary.

Since $x$ is injective, the equality $\|x\xi\|=\||x|\xi\|$ for $\xi\in \dom(x)$ implies that $|x|$ is injective. Thus $1_{\{0\}}(|x|)=p_{\ker(|x|)}=0,$ and so $|x|^{-1}$ may be defined as a closeable operator in $\rL^0(M,\tau).$ The computation $(x^*)^{-1}=(x^{-1})^*$ is straightforward.


Now suppose $(M,\tau)$ is finite. Without loss of generality, $\tau(\id)=1$. Suppose $x$ is injective. Then $u$ is also injective and $u^*u=\id$. Thus $1=\tau(u^*u) = \tau(u^*u)$. Because the only projection with trace 1 is $\id$ (by faithfulness of $\tau$), $u^*u=\id$ and $u$ is unitary. As above, $|x|$ is invertible. So $|x|^{-1}u^*$ is an inverse to $x$.


\end{proof}


\section{The log-square integrable general linear group}\label{sec:GL2}


For this section, we fix a semi-finite tracial von Neumann algebra $(M,\tau)$. Let
$$\GL^2(M,\tau)=\{a\in \rL^0(M,\tau)^{\times}: \log(|a|)\in \rL^2(M,\tau)\}$$
be the {\bf log-square integrable general linear group} of $(M,\tau)$. For brevity we will write $G=\GL^2(M,\tau)$. Although we call this set a group, it is not at all obvious that $G$ is closed under multiplication. The main result of this section is:

\begin{thm}\label{T:log subgroup}
$G$ is a subgroup of $\rL^0(M,\tau)^{\times}.$  Moreover, for every $a\in G$ we have that $a^{*}\in G$ and additionally:
\[\|\log( |a|)\|_{2}=\|\log(|a^{*}|)\|_{2}=\|\log(|a^{-1}|)\|_{2}.\]
\end{thm}

We start with some basic facts about spectral measures.

\begin{prop}\label{P:transformation of measures}
Let $(M,\tau)$ be a semi-finite von Neumann algebra and $a\in \rL^0(M,\tau)^{\times}.$ Then:
\begin{enumerate}
\item $\mu_{|a|}=\mu_{|a^{*}|},$ \label{I:same measure under *}
\item  $\mu_{|a^{-1}|}=\mu_{|a|^{-1}}=r_{*}(\mu_{|a|}),$ where $r\colon (0,\infty)\to (0,\infty)$ is the map $r(t)=t^{-1}.$  \label{I:change of measure under reflection}
\end{enumerate}

\end{prop}

\begin{proof}

(\ref{I:same measure under *}):
Let $a=u|a|$ be the polar decomposition. Since $a\in \rL^0(M,\tau)^{\times}$ by Lemma \ref{L:invertible} we have that $u\in \mathcal{U}(M)$ (which is the unitary group of $M$). Then $a^{*}=|a|u^{*},$ and $|a^{*}|^{2}=aa^{*}=u|a|^{2}u^{*}.$ From this, it is easy to see that $|a^{*}|=u|a|u^{*},$ because $(u|a|u^*)^2 = u|a|^2u^*$  Thus, for every Borel $E\subseteq [0,\infty)$
\[\mu_{|a^{*}|}(E)=\tau(1_{E}(|a^{*}|))=\tau(1_{E}(u|a|u^{*}))=\tau(u1_{E}(|a|)u^{*})=\tau(1_{E}(|a|))=\mu_{|a|}(E).\]


(\ref{I:change of measure under reflection}):
Again, let $a=u|a|$ be the polar decomposition. So $a^{-1}=|a|^{-1}u^{*}.$ As in (\ref{I:same measure under *}), it is direct to show that $|a^{-1}|=u|a|^{-1}u^{*}.$ The proof then proceeds exactly as in (\ref{I:same measure under *}), using that $r_{*}(\mu_{|a|})=\mu_{|a|^{-1}}$ (which follows from functional calculus).

\end{proof}

Because expressions like $1_{(\lambda,\infty)}(|a|)(\rL^2(M,\tau))$ will show up frequently, it will be helpful to introduce the following notation. Given $a\in \rL^0(M,\tau)$ and $E\subseteq [0,\infty)$ Borel, we let $\mathcal{H}^{a}_{E}=1_{E}(|a|)(\rL^2(M,\tau)). $
It will be helpful for us to derive an alternate expression for $\|\log(|a|)\|_{2}.$ Note
\begin{eqnarray}
\|\log(|a|)\|_{2}^{2}&=&\int_0^\infty t^{2}\,\dee\mu_{|\log(|a|)|}(t)\nonumber \\
&=&\int_0^\infty \left(\int_0^t 2\l\, \dee\l\right) \, \dee\mu_{|\log(|a|)|}(t)\nonumber \\
&=&2 \int_0^\infty \lambda \int_{0}^{\infty} 1_{(\lambda,\infty)}(t)\,\dee\mu_{|\log(|a|)|}(t)\,\dee\lambda=2\int_{0}^{\infty}\lambda\mu_{|\log(|a|)|}(\lambda,\infty)\,\dee\lambda. \label{E:simpler formula for norm}
\end{eqnarray}

Note that we have used Fubini's Theorem. This is valid if $\mu_{|\log(|a|)|}(\l,\infty)<\infty$ for all $\l>0$ since then $\mu_{|\log(|a|)|}$ is sigma-finite. On the other hand, if  $\mu_{|\log(|a|)|}(\l,\infty)=\infty$ for some $\l>0$ then both sides are infinite, so the formula is correct in this case, too.

By functional calculus,  $\mu_{|\log(|a|)|}$ is the pushforward of $\mu_{|a|}$ under the map $t\mapsto |\log(t)|.$ So
$$\|\log(|a|)\|_{2}^{2}=2\int_{0}^{\infty}\lambda\left[\mu_{|a|}(e^{\lambda},\infty)+\mu_{|a|}(0,e^{-\lambda})\right]\,\dee\lambda.$$
By Proposition \ref{P:transformation of measures},
\begin{eqnarray}\label{E:formula for norm}
\|\log(|a|)\|_{2}^{2}=2\int_{0}^{\infty}\lambda [\mu_{|a|}(e^{\lambda},\infty)+\mu_{|a^{-1}|}(e^{\lambda},\infty)]\,\dee\lambda.
\end{eqnarray}

\begin{prop}\label{P:submutliplicative estimate}
Let $(M,\tau)$ be a semi-finite von Neumann algebra and $a, b\in \rL^0(M,\tau).$
Given $\lambda_{1},\lambda_{2},\lambda\in (0,\infty)$ with $\lambda_{1}\lambda_{2}=\lambda,$ we have that
\[\tau(1_{(\lambda,\infty)}(|ab|))\leq \tau(1_{(\lambda_{1},\infty)}(|a|))+\tau(1_{(\lambda_{2},\infty)}(|b|)).\]
\end{prop}

\begin{proof}
The proof is almost immediate from \cite[Lemma 2.5 (vii) and Proposition 2.2]{MR840845}. To explain, for $a \in \rL^0(M,\tau)$, let $\tmu_t(a)$ be the infimum of $\|ap\|_\infty$ over all projections $p \in M$ with $\tau(1-p) \le t$. This is the {\bf $t$-th generalized $s$-number of $a$}. Also let $\tlambda_t(a)=\tau(1_{(t,\infty)}(|a|))$. So $t\mapsto \tlambda_t(a)$ is the {\bf distribution function of $a$}. These invariants are related by \cite[Proposition 2.2]{MR840845} which states
$$\tmu_t(a) = \inf\{s\ge 0:~ \tlambda_s(a)\le t\}.$$
It follows that $\tmu_t(a) \le s$ if and only if $\tlambda_s(a)\le t$. In particular, $\tmu_{\tlambda_t(a)}(a) \le t$ always holds.

 \cite[Lemma 2.5 (vii)]{MR840845} states
$$\tmu_{t+s}(ab) \le \tmu_t(a)\tmu_s(b)$$
for any $t,s >0$ and any $a,b \in \rL^0(M,\tau)$.

Now that the tools above are ready, we return to the proposition we want to prove. The inequality  $\tau(1_{(\lambda,\infty)}(|ab|))\leq \tau(1_{(\lambda_{1},\infty)}(|a|))+\tau(1_{(\lambda_{2},\infty)}(|b|))$ is equivalent to the statement
$$\tlambda_{ts}(ab) \le \tlambda_t(a) + \tlambda_s(b).$$
By \cite[Proposition 2.2]{MR840845}, this is true if and only if
$$\tmu_{\tlambda_t(a) + \tlambda_s(b)}(ab) \le ts.$$
By \cite[Lemma 2.5 (vii)]{MR840845},
$$\tmu_{\tlambda_t(a) + \tlambda_s(b)}(ab) \le \tmu_{\tlambda_t(a)}(a)\tmu_{\tlambda_t(b)}(b).$$
By \cite[Proposition 2.2]{MR840845} again, $\tmu_{\tlambda_t(a)}(a) \le t$ and $\tmu_{\tlambda_s(b)}(b) \le s$. Combining these inequalities proves the proposition.

\end{proof}

\begin{proof}[Proof of Theorem \ref{T:log subgroup}]
The fact that $G$ is closed under inverses and the $*$ operation is obvious from Proposition \ref{P:transformation of measures}. Let $a,b\in G.$ By Proposition \ref{P:submutliplicative estimate}:
\begin{align*}
&\|\log(|ab|)\|_{2}^{2}=2\int_{0}^{\infty}\lambda\left[\mu_{|ab|}(e^{\lambda},\infty)+\mu_{|b^{-1}a^{-1}|}(e^{\lambda},\infty)\right]\,\dee\lambda \\
&\leq 2\int_{0}^{\infty}\lambda\left[\mu_{|a|}(e^{\lambda/2},\infty)+\mu_{|a^{-1}|}(e^{\lambda/2},\infty)\right]\,\dee\lambda+2\int_{0}^{\infty}\lambda\left[\mu_{|b|}(e^{\lambda/2},\infty)+\mu_{|b^{-1}|}(e^{\lambda/2},\infty)\right]\,\dee\lambda\\
&=4\int_{0}^{\infty}\lambda\left[\mu_{|a|}(e^{t},\infty)+\mu_{|a^{-1}|}(e^{t},\infty)\right]\,\dee t+4\int_{0}^{\infty}\lambda\left[\mu_{|b|}(e^{t},\infty)+\mu_{|b^{-1}|}(e^{t},\infty)\right]\,\dee\mu\\
&=2(\|\log(|a|)\|_{2}^{2}+\|\log(|b|)\|_{2}^{2}).
\end{align*}
\end{proof}

We also need the following fact analogous to Proposition \ref{P:submutliplicative estimate}, but whose proof is easier.

\begin{prop}\label{P:subadditive estimate}
Let $(M,\tau)$ be a semi-finite von Neumann algebra, and $a,b\in \rL^0(M,\tau).$ Then for all $\lambda_{1},\lambda_{2}>0$ we have that
\[\mu_{|a+b|}(\lambda_{1}+\lambda_{2},\infty)\leq \mu_{|a|}(\lambda_{1},\infty)+\mu_{|b|}(\lambda_{2},\infty).\]
\end{prop}

\begin{proof}
We use \cite[Lemma 2.5, (v)]{MR840845}, which shows that
$$\tmu_{\mu_{|a|}(\l_1,\infty)+\mu_{|b|}(\l_2,\infty)}(a+b) \leq \tmu_{\mu_{|a|}(\l_1,\infty)}(a)+\tmu_{\mu_{|b|}(\l_2,\infty)}(b) \leq \l_1 + \l_2.$$
Apply \cite[Proposition 2.2]{MR840845} to the inequality above to obtain
$$\tlambda_{\l_1 + \l_2}(|a+b|) \le \mu_{|a|}(\l_1,\infty)+\mu_{|b|}(\l_2,\infty).$$
By definition of $\tlambda$, $\tlambda_{\l_1 + \l_2}(|a+b|) = \mu_{|a+b|}(\lambda_{1}+\lambda_{2},\infty)$ so this finishes the proof.

\end{proof}

\section{The geometry of positive definite operators}\label{sec:positive}

Let $\cP=\cP(M,\tau)=\{x \in \GL^2(M,\tau):~x>0\}$ be the positive definite elements of $\GL^2(M,\tau)$. In \S \ref{sec:symmetric}, we review work of Andruchow-Larontonda  \cite{MR2254561} on the geometry of $\cP \cap M$. In \S \ref{sec:measure}, we review the measure topology on $\rL^0(M,\tau)$. By approximating $\cP$ by $\cP \cap M$ (in the measure topology), we show in \S \ref{sec:P1} that $d_\cP$ (as defined in the introduction) is a metric on $\cP$. Moreover, $\GL^2(M,\tau)$ acts transitively and by isometries on $(\cP,d_\cP)$. In \S \ref{sec:P2}, we show that the exponential map $\exp:\rL^2(M,\tau)_{sa} \to \cP$ is a homeomorphism. From this, we conclude that $(\cP,d_\cP)$ is a complete CAT(0) metric space and characterize its geodesics.  In \S \ref{sec:semi-finite case} we generalize results to semi-finite von Neumann algebras.


\subsection{The space $\cP^\infty(M,\tau)$ of bounded positive operators}\label{sec:symmetric}
Throughout this section, we assume $(M,\tau)$ is finite.

Let $M_{sa} \subset M$ be the subspace of self-adjoint elements and
$$\cP^\infty=\cP^\infty(M,\tau):=\{ \exp(x):~ x \in M_{sa} \} \subset M_{sa}$$
be the positive definite elements with bounded inverse. Note $\cP^\infty=\cP \cap M^\times$. This section studies $\cP^\infty$ equipped with a natural metric, as introduced in \cite{MR2254561}. The results in this section are obtained directly from \cite{MR2254561}.

Let $\GL^\infty(M,\tau)=M^\times$ be the group of elements $x \in M$ such that $x$ has a bounded inverse $x^{-1}$ in $M$. This group acts on $M_{sa}$ by
$$g.w := g w g^* ~(\forall g\in \GL^\infty(M,\tau), w \in M_{sa}).$$

For $w \in \cP^\infty$, the {\bf tangent space to $\cP^\infty$ at $w$}, denoted $T_w(\cP^\infty)$, is a copy of $M_{sa}$ with the inner product $\langle \cdot, \cdot \rangle_w$ defined by
$$\langle x,y \rangle_w := \langle w^{-1/2}.x, w^{-1/2}.y\rangle = \tau( w^{-1/2} x^* w^{-1} y w^{-1/2}) = \tau(w^{-1}x w^{-1}y).$$
 Let $\|\cdot \|_{w,2}$ denote the $\rL^2$-norm with respect to this inner product. In the special case that $w=\id$ is the identity, this is just the restriction of the standard inner product to $M_{sa}$.

These inner products induce a Riemannian metric on $\cP^\infty(M,\tau)$. The reader might be concerned that the tangent spaces $T_w(\cP^\infty)$ are not necessarily complete with respect to their inner products. The incompleteness of $M_{sa}$ causes no difficulty in defining the metric on $\cP^\infty$ but it does mean that Theorem \ref{thm:km} cannot be directly applied to $\cP^\infty$.

\begin{example}[The abelian case]
Suppose $M=\rL^\infty(Y,\nu)$ as in \S \ref{intro-abelian}. Then $T_w(\cP^\infty)=M_{sa}$ consists of essentially bounded real-valued functions on $Y$. The completion of $M_{sa}$ is $\rL^2((Y,\nu), \R)$. If $\nu$ is not supported on a finite subset, then $\rL^2(Y,\nu)$ is not contained in $\rL^\infty(Y,\nu)$ and $M_{sa}$ is not complete. More generally, if $M$ contains an embedded copy of $\rL^\infty(Y,\nu)$ for a diffuse measure space $(Y,\nu)$ and then $M_{sa}$ will not be complete. For example, this occurs if $M$ is diffuse. See Proposition \ref{prop:facts about GL2/L2 diffuse case}.

\end{example}

Here is a more detailed explanation of the metric. Let $\g:[a,b] \to \cP^\infty$ be a path. The $\rL^2$-derivative of $\g$ at $t$ is defined by
$$\g'(t) = \lim_{h\to 0} \frac{\g(t+h)-\g(t)}{h}$$
where the limit is taken with respect to the $\rL^2$-metric on $T_{\g(t)}(\cP^\infty)$. Then the length of $\g$ is defined as in the finite-dimensional case:
$$\textrm{length}_\cP(\g) = \int_a^b \|\g'(t)\|_{\g(t),2} ~dt.$$
Define distance on $\cP^\infty(M,\tau)$ by $d_\cP(x,y)= \inf_\g \textrm{length}_\cP(\g)$ where the infimum is taken over all piece-wise smooth curves $\g$ with derivatives in $M$. For this to be well-defined, it needs to be shown that there exists a piecewise smooth curve between any two points of $\cP^\infty$. For any $\exp(x) \in \cP^\infty$, the map $t \mapsto \exp(tx)$ defines a smooth curve from $\id$ to $\exp(x)$. A piecewise smooth curve between any two points can be obtained by concatenating two of these special curves.

\begin{lem}\label{lem:isometries}
The action of $\GL^\infty(M,\tau)$ on $\cP^\infty$ is transitive and by isometries.
\end{lem}
\begin{proof}
The action of $\GL^\infty(M,\tau)$ on $\cP^\infty$  is by isometries since the Frechet derivative of $g$ at $w$ is the map
$$x \in T_w(\cP^\infty) \mapsto g.x =  gxg^* \in T_{g.w}(\cP^\infty)$$
and
\begin{eqnarray*}
\langle g.x, g.y \rangle_{g.w} &=& \tau( (g.w)^{-1} (g.x) (g.w)^{-1} (g.y) ) \\
&=& \tau( (g^*)^{-1}w^{-1}g^{-1} (gxg^*) (g^*)^{-1}w^{-1}g^{-1}  (gyg^*) ) \\
&=& \tau( w^{-1}xw^{-1}y ) = \langle x,y \rangle_w.
\end{eqnarray*}
The action $\GL^\infty(M,\tau) \cc \cP^\infty$ is transitive since for any $w\in \cP^\infty$, $w^{1/2} \in \GL^\infty(M,\tau)$ and
$$w^{1/2}.\id = w.$$
\end{proof}

\begin{lem}\label{lem:commute}\cite[Lemma 3.5]{MR2254561}
For any $a,b \in \cP^\infty$,
$$d_\cP(a,b) = \| \log (b^{-1/2}ab^{-1/2}) \|_2 \ge \|\log(a)-\log(b)\|_2.$$
\end{lem}

\begin{thm}\label{thm:cat}
$\cP^\infty(M,\tau)$ is a CAT(0) space.
\end{thm}

\begin{proof}
This follows from \cite[Lemma 3.6]{MR2254561} and \cite[Chapter II.1, Proposition 1.7 (3)]{bridson-haefliger-book}.
\end{proof}

 \begin{cor}\label{C:contraction}
Let $x,y \in M_{sa}$ and $\s \ge 1$ be a scalar. Then
$$d_\cP(e^{\s x},e^{\s y}) \ge \s d_\cP(e^x,e^y).$$
\end{cor}

\begin{proof}
Let $x', y' \in M_{sa}$ and let $f(t) = d_\cP(e^{t x'},e^{t y'})$. By  \cite[Corollary 3.4]{MR2254561}, $f$ is convex. Therefore,
$$f(t) \le t f(1) + (1-t)f(0) = tf(1)$$
for any $0\le t \le 1$. Set $t=1/\s$, $x' = \s x$, $y' = \s y$ to obtain
$$f(t) = d_\cP(e^{x},e^{y}) \le \s^{-1} d_\cP(e^{\s x},e^{\s y}).$$

\end{proof}

\subsection{The measure topology}\label{sec:measure}

This section reviews the measure topology on $\rL^0(M,\tau)$. The results here are probably well-known but being unable to find them explicitly stated in the literature, we give proofs for completeness. We will need this material in the next two sections.

Let $(M,\tau)$ be a semi-finite von Neumann algebra. By \cite[Theorem IX.2.2]{TakesakiII},  the sets
\[\mathcal{O}_{\varepsilon,\delta}(a)=\left\{b\in \rL^0(M,\tau):\tau(1_{(\varepsilon,\infty)}(|a-b|))<\delta\right\}\]
ranging over $a\in \rL^0(M,\tau)$ and $\varepsilon,\delta>0$ form a basis for a metrizable vector space topology on $\rL^0(M,\tau),$ and this topology turns $\rL^0(M,\tau)$ into a topological $*$-algebra (i.e. the product and sum operations are continuous as a function of two variables, as is the adjoint). We shall call this topology the {\bf measure topology}.
This motivates the following definition.

\begin{defn}
Let $(M,\tau)$ be a semi-finite von Neumann algebra. Given a sequence $(a_{n})_{n}$ in $\rL^0(M,\tau),$ and an $a\in \rL^0(M,\tau)$ we say that {\bf $a_{n}\to a$ in measure} if for every $\varepsilon>0$ we have that  \[\tau(1_{(\varepsilon,\infty)}(|a-a_{n}|))\to_{n\to\infty} 0.\]
\end{defn}

\begin{lem}\label{L:L2tomeasure}
Let $(M,\tau)$ be a semi-finite von Neumann algebra. Suppose $x_1,x_2,\ldots \in \rL^2(M,\tau)$ and $\lim_{n\to\infty} x_n = x_\infty$  in $\rL^2(M,\tau)$. Then $x_n \to_{n\to\infty} x_\infty$ in measure.

\end{lem}
\begin{proof}
For any $\eps>0$,
$$\tau(1_{(\eps,\infty)}(|x_n-x|)) \eps^2 \le \|x_n-x\|_2^2.$$
Since $x_n \to x$ in $\rL^2(M,\tau)$, this shows $x_n \to x$ in measure.
\end{proof}

\begin{prop}\label{P:3topologies}
Let $(M,\tau)$ be a von Neumann algebra with a finite trace.  Let $C>0$ and let $M_C \subset M$ be the set of all elements $x$ with $\|x\|_\infty \le C$. Then the measure, strong operator and $\rL^2$ topologies all coincide on $M_C$.  By Strong Operator Topology we mean with respect to either of the inclusions $M \subset B(\cH)$ or $M \subset B(\rL^2(M,\tau))$.
\end{prop}

\begin{proof}
By \cite[I.4.3. Theorem 2]{MR641217} or \cite[Corollary 2.5.9 and Proposition 2.5.8]{anantharaman-popa}, the topology induced on $M_C$ from the SOT on $B(\cH)$ is the same as the topology it inherits from the SOT on  $B(\rL^2(M,\tau))$.

Let $x \in M_C$ and let $(x_n)_n \subset M_C$ be a sequence. We will show that if $x_n \to x$ in one of the three topologies then $x_n\to x$ in the other topologies. After replacing $x_n$ with $x_n-x$ and $C$ with $2C$ if necessary, we may assume $x=0$. Also without loss of generality we may assume $\tau(\id)=1$.

 Suppose that $x_n \to 0$ in measure. We will show that $x_n \to 0$ in $\rL^2$. For any $\varepsilon>0$,
$$\|x_n\|_2^2 = \tau(x_n^*x_n) \le C^2 \tau(1_{(\varepsilon,\infty)}(|x_{n}|)) + \varepsilon^2.$$
Since $x_n \to 0$ in measure, $\limsup_{n\to\infty} \|x_n\|_2^2 \le \varepsilon^2$. Since $\varepsilon$ is arbitrary, this shows $x_n \to 0$ in $\rL^2$.

Now suppose that $x_n \to 0$ in $\rL^2$. We will show that $x_n \to 0$ in the SOT. So let $\xi \in L^2(M,\tau)$. If $\xi \in M$ then
$$\limsup_{n\to\infty} \|x_n \xi\|_2 \le \limsup_{n\to\infty} \|x_n\|_2 \|\xi\|_\infty = 0.$$
In general, for any $\xi \in L^2(M,\tau)$ and $\eps>0$, there exists $\xi' \in M$ with $\|\xi - \xi'\|_2 < \eps$. Then
$$\limsup_{n\to\infty} \|x_n \xi\|_2 \le \limsup_{n\to\infty} \|x_n \xi'\|_2 +\|x_n (\xi-\xi')\|_2  \le \limsup_{n\to\infty} \|x_n (\xi-\xi')\|_2 \le C \eps.$$
Since $\eps>0$ is arbitrary, this shows $x_n \to 0$ in SOT.

Now suppose that $x_n \to 0$ in SOT.  Since $\id \in L^2(M,\tau)$ and $x_n \id = x_n$, $\|x_n\|_2 \to 0$. This shows $x_n \to 0$ in $\rL^2$.


Now suppose $x_n \to 0$ in $\rL^2$. Let $\varepsilon>0$. Then $\tau(1_{(\varepsilon,\infty)}(|x_{n}|)) \le \varepsilon^{-2} \|x_n\|^2_2$. Since $\|x_n\|^2_2 \to 0$ this implies
$$\limsup_{n\to\infty} \tau(1_{(\varepsilon,\infty)}(|x_{n}|)) = 0.$$
So $x_n \to 0$ in measure.

\end{proof}

\begin{remark}
It is possible that the topology on $M$ inherited from the SOT on $B(\cH)$ is not the same as topology it inherits from the SOT on $B(\rL^2(M,\tau))$ \cite[Exercise 1.3]{anantharaman-popa}.
\end{remark}

\begin{defn}
If $(\mu_n)_n$ is a sequence of Borel probability measures on a topological space $X$ and $\mu$ is another Borel probability measure on $X$ then we write $\mu_n \to \mu$ {\bf weakly} if for every bounded continuous function $f:X \to \C$, $\int f~\dee\mu_n$ converges to $\int f~\dee\mu$ as $n\to\infty$.
\end{defn}

Recall that $C_0(\R)$ denotes continuous functions on $\R$ that vanish at infinity while $C_c(\R) \subset C_0(\R)$ denotes those functions with compact support.

\begin{prop}\label{P:conv in meas implies wk* conv}
Let $(M,\tau)$ be a von Neumann algebra with a finite trace. Suppose that $(a_{n})_n, (b_n)_n\in \rL^0(M,\tau)$, $a_n \to a$ in measure, $b_n \to b$ in measure and $b_n$ is self-adjoint for all $n$.
 Then:
 \begin{enumerate}
 \item  $\mu_{b_{n}}\to \mu_{b}$ weakly. \label{I:wk*conv meas sa}
  \item For all but countably many $\lambda \in \R$ we have that $\mu_{b_{n}}(\lambda,\infty)\to \mu_{b}(\lambda,\infty)$.\label{I:level set convergence}
 \item For every bounded, continuous $f\colon \R\to \R$ we have that $\|f(b_{n})-f(b)\|_{2}\to_{n\to\infty}0.$ \label{I:operate continuously bounded functions}
\item For every continuous $f\colon \R\to \R$ we have that $f(b_{n})\to f(b)$ in measure. \label{I:operate continuously measure}
 \item  $|a_n| \to |a|$ in measure. \label{I:wk* conv meas}

 \end{enumerate}

\end{prop}

\begin{proof}
After scaling if necessary we will assume without loss of generality that $\tau(\id)=1$.

(\ref{I:wk*conv meas sa}):
Let $f\in C_{0}(\R)$ (where $C_0(\R)$ is the space of continuous functions that vanish at infinity). By \cite[Corollary 5.4]{StinespringMeasOps}, we know that
\[\lim_{n\to\infty} \|f(b_{n})-f(b)\|_{2} = 0.\]
Since $|\tau(x)-\tau(y)|\leq \|x-y\|_{2}$ for all $x,y\in M,$ the above convergence shows that
\[\lim_{n\to\infty}\int f\,d\mu_{b_{n}}=\lim_{n\to\infty}\tau(f(b_{n}))=\tau(f(b))=\int f\,d\mu_{b}.\]
Now that we know the integrals $\int f\,d\mu_{b_{n}}$ converge as $n\to\infty$ for $f \in C_0(\R)$, it follows that these integrals converge for all bounded continuous $f:\R \to \R$ because $\mu_{b_{n}},\mu_{b}$ are all probability measures (see, e.g. \cite[Exercise 20 of Chapter 7]{Folland}).

(\ref{I:level set convergence}): This follows from  (\ref{I:wk*conv meas sa}) and the Portmanteau theorem.

(\ref{I:operate continuously bounded functions}): Let $R>0$ be such that $|\phi(t)|\leq R$ for all $t\in \R.$ Let $\varepsilon>0,$ and choose a $T>0$ so that $\mu_{b}(\{t:|t|\geq T\})<\varepsilon.$ Choose a function $\psi\in C_{c}(\R)$ with $\psi(t)=1$ for $|t|\leq T$ and so that $0\leq \psi\leq 1.$ Then,
\begin{align*}
\|\phi(b_{n})-\phi(b)\|_{2}&\leq \|\phi \psi(b_{n})-\phi\psi(b)\|_{2}+\|\phi(1-\psi)(b_{n})\|_{2}+\|\phi(1-\psi)(b)\|_{2}\\
&=\|\phi \psi(b_{n})-\phi\psi(b)\|_{2}+\left(\int |\phi(t)|^{2}(1-\psi(t))^{2}\,d\mu_{b_{n}}(t)\right)^{1/2}\\
&+\left(\int |\phi(t)|^{2}(1-\psi(t))^{2}\,d\mu_{b}(t)\right)^{1/2}.
\end{align*}
By \cite[Corollary 5.4]{StinespringMeasOps} and  (\ref{I:wk*conv meas sa}) we thus have that
\[\limsup_{n\to\infty}\|\phi(b_{n})-\phi(b)\|_{2}\leq 2\left(\int |\phi(t)|^{2}(1-\psi(t))^{2}\,d\mu_{b}(t)\right)^{1/2}<2R\varepsilon.\]
Letting $\varepsilon\to 0$ completes the proof.


(\ref{I:operate continuously measure}):
Let $\phi\in C_{c}(\R).$ By \cite[Theorem 5.5]{StinespringMeasOps}, it suffices to show that
\[\|\lim_{n\to\infty} \phi(f(b_{n}))-\phi(f(b))\|_{2} = 0.\]
Since $\phi\circ f$ is bounded and continuous, this follows from (\ref{I:operate continuously bounded functions}).

(\ref{I:wk* conv meas}):
Since $\rL^0(M,\tau)$ is a topological $\ast$-algebra in the measure topology, $a_{n}^{*}a_{n} \to a^*a$ in the measure topology. Let $g\colon [0,\infty)\to [0,\infty)$ be the function $g(t)=\sqrt{t}.$ Then $|a_{n}|=g(a_{n}^{*}a_{n}).$ So this follows from (\ref{I:operate continuously measure}) with $b_{n}=a_{n}^{*}a_{n}.$

\end{proof}

We can give a more refined improvement of Proposition \ref{P:conv in meas implies wk* conv}.\ref{I:operate continuously measure} which we will need later. We first note the following.

\begin{cor}\label{L:compactness implies uniform integrability}
Let $(M,\tau)$ be a von Neumann algebra with a finite trace, and let $K\subseteq \rL^0(M,\tau)$ have compact closure in the measure topology. Then for every $\varepsilon>0,$ there is an $R>0$ so that
\[\tau(1_{(R,\infty)}(|a|))<\varepsilon\]
for all $a\in K.$
\end{cor}

\begin{proof}
Replacing $K$ with its closure, we may as well assume $K$ is compact. By Proposition \ref{P:conv in meas implies wk* conv} (\ref{I:wk* conv meas}) and (\ref{I:wk*conv meas sa}), the map $\rL^0(M,\tau)\to \Prob(\R)$ sending $x\mapsto \mu_{|x|}$ is continuous if we give $\rL^0(M,\tau)$ the measure topology, and $\Prob(\R)$ the weak topology. So $\{\mu_{|a|}:a\in K\}$ is compact in the weak topology, and thus tight. Tightness means there exists an $R>0$ so that $\mu_{|a|}(R,\infty)<\varepsilon$ for all $a\in K.$ As $\tau(1_{(R,\infty)}(|a|))=\mu_{|a|}(R,\infty),$ we are done.

\end{proof}

\begin{cor}\label{C:joint coninuity of functional calculus}
Let $(M,\tau)$ be a von Neumann algebra with a finite trace. Then the map
\[\mathcal{E}\colon \rL^0(M,\tau)_{sa}\times C(\R,\R)\to \rL^0(M,\tau)_{sa}\]
given by $\mathcal{E}(a,f)=f(a)$ is continuous if we give $\rL^0(M,\tau)_{sa}$ the measure topology and $C(\R, \R)$ the topology of uniform convergence on compact sets.

\end{cor}

\begin{proof}
Suppose we are given sequences $(f_{n})_n \subset C(\R,\R)$, $(a_{n})_n \subset  \rL^0(M,\tau)_{sa}$ and $f\in C(\R)$, $a\in \rL^0(M,\tau)_{sa}$ with $f_{n}\to f$ uniformly on compact sets and $a_{n}\to a$ in measure.

To prove $f_{n}(a_{n})\to f(a)$ in measure, fix $\lambda>0.$ Let $\varepsilon>0$ be given. By Corollary \ref{L:compactness implies uniform integrability}, we may choose an $R>0$ so that
\[\sup_{n\in \N}\tau(1_{(R,\infty)}(|a_{n}|))<\varepsilon, \quad \tau(1_{(R,\infty)}(|a|))<\varepsilon.\]
Let $g\colon \R\to \R$ be a bounded continuous function with $g(t)=t$ for $|t|\leq R,$ and define $h\colon \R\to \R$ by $h(t)=f(t)-f(g(t))$ and  $h_n\colon \R\to \R$ by $h_n(t)=f_n(t)-f_n(g(t))$. Then:
\[f_{n}(a_{n})-f(a)=f_{n}(g(a_{n}))-f(g(a))+h_n(a_{n})-h(a).\]
Then, by Proposition \ref{P:subadditive estimate} we have that:
\begin{align*}
& \tau(1_{(\lambda,\infty)}(|f_{n}(a_{n})-f(a_{n})|)\\
& \leq \tau(1_{(\lambda/4,\infty)}(|h_n(a_{n})|))+\tau(1_{(\lambda/4,\infty)}(|h(a)|))+\tau(1_{(\lambda/2,\infty)}(|f_{n}(g(a_{n}))-f(g(a))|))\\
&\leq  \tau(1_{(\lambda/4,\infty)}(|h_n(a_{n})|))+\tau(1_{(\lambda/4,\infty)}(|h(a)|))+\frac{4}{\lambda^{2}}\|f_{n}(g(a_{n}))-f(g(a))\|_{2}.
\end{align*}
Since $\lambda>0,$ and $h=0$ in $[-R,R]$ it follows that for all $n\in \N:$
\[\tau(1_{(\lambda/4,\infty)}(|h_n(a_{n})|))\leq \tau(1_{(R,\infty)}(|a_{n}|))<\varepsilon.\]
Similarly,
\[\tau(1_{(\lambda/4,\infty)}(|h(a)|))<\varepsilon.\]
For the last term: let $T>0$ be such that $\|g\|_{\infty}\leq T.$ Then:
\begin{align*}
  \|f_{n}(g(a_{n}))-f(g(a))\|_{2}&\leq \|f_{n}(g(a_{n}))-f(g(a_{n}))\|_{2}+\|f(g(a_{n}))-f(g(a))\|_{2}\\
&\leq \|f_{n}(g(a_{n}))-f(g(a_{n}))\|_\infty+\|f(g(a_{n}))-f(g(a))\|_{2}\\
  &\leq \sup_{t\in \R:|t|\leq T}|f_{n}(t)-f(t)|+\|f(g(a_n))-f(g(a))\|_{2}.
\end{align*}
We have that $\sup_{t\in \R:|t|\leq T}|f_{n}(t)-f(t)|\to_{n\to\infty}0$ as $n\to\infty$ since $f_{n}\to f$ uniformly on compact sets. We also have that $\|f(g(a_n))-f(g(a))\|_{2}\to 0$ by Proposition \ref{P:conv in meas implies wk* conv} (\ref{I:operate continuously bounded functions}).
Hence $\|f_{n}(g(a_{n}))-f(g(a))\|_{2}\to_{n\to\infty}0$. 
 Altogether, we have shown that
\[\limsup_{n\to\infty}\tau(1_{(\lambda,\infty)}(|f_{n}(a_{n})-f(a_{n})|))\leq 2\varepsilon.\]
Since $\varepsilon>0$ is arbitrary, we can let $\varepsilon\to 0$ to show that
\[\tau(1_{(\lambda,\infty)}(|f_{n}(a_{n})-f(a)|))\to_{n\to\infty}0.\]
Since this is true for every $\lambda>0,$ we have that $f_{n}(a_{n})\to_{n\to\infty}f(a)$ in measure.

\end{proof}

\subsection{The space $\mathcal{P}(M,\tau)$ of positive log-square integrable operators}\label{sec:P1}

\begin{defn}
Let $(M,\tau)$ be a tracial von Neumann algebra, and let
$$G=\GL^2(M,\tau)=\{a\in \rL^0(M,\tau):\log(|a|)\in \rL^2(M,\tau)\}.$$
Set $\mathcal{P}=\mathcal{P}(M,\tau)=\{a\in G:a> 0\}.$ For $a,b\in \mathcal{P},$ set
\[d_{\mathcal{P}}(a,b)=\|\log(b^{-1/2}ab^{-1/2})\|_{2}.\]
This is well-defined by Theorem \ref{T:log subgroup}. Note that $\cP^\infty = \cP\cap M$ and $d_\cP$ restricted to $\cP^\infty$ agrees with the formula above by Lemma \ref{lem:commute}.
\end{defn}

The main result of this section is:

\begin{thm}\label{T:I heard you like metrics}
Let $(M,\tau)$ be a von Neumann algebra with a finite trace. Then,
\begin{enumerate}
\item $d_{\mathcal{P}}$ is a metric.\label{I:its a metric}
\item The group $G$ acts on $\mathcal{P}$ by isometries  by $g.a=gag^{*}$. \label{I:isometric action}
\item The action $G \cc \cP$ is transitive. \label{I:transitive}
\item $\cP^\infty$ is dense in $\cP$. \label{I:dense}
\end{enumerate}
\end{thm}

We will extend this result to the semi-finite case in \S \ref{sec:semi-finite case}. To prove this theorem, we will approximate elements of $\cP$ by elements of $\cP^\infty$ in the measure topology and then apply results from the previous section on $\cP^\infty$. Because we will use the Dominated Convergence Theorem, we need some basic facts about operator monotonicity. Recall that if $a,b$ are operators then by definition, $a\le b$ if and only if $b-a\in \rL^0(M,\tau)$ is a positive operator (where $b-a$ is defined to be the closure of $b-a$ restricted to $\dom(b)\cap \dom(a)$).


\begin{prop}\label{P:basic facts positive ops affil}
Let $(M,\tau)$ be a semi-finite von Neumann algebra.
\begin{enumerate}
\item Suppose $a,b\in \rL^0(M,\tau)$ and $|a|\leq |b|.$ Then for every $\lambda>0$ we have that \label{I:ineq op ineq meas}
\[\mu_{|a|}(\lambda,\infty)\leq \mu_{|b|}(\lambda,\infty).\]
\item If $a,b\in \rL^0(M,\tau)$ are self-adjoint and $a\leq b,$ then $cac^{*}\leq cbc^{*}$ for every $c\in \rL^0(M,\tau).$ \label{I:this is easy}
\end{enumerate}
\end{prop}

\begin{proof}
(\ref{I:ineq op ineq meas}): This is implied by \cite[Lemma 3.(i)]{BKMeasOps} or \cite[Lemma 2.5(iii)]{MR840845}.

(\ref{I:this is easy}): We may write $b-a=d^{*}d$ for some $d\in \rL^0(M,\tau).$ Then
$cac^{*}-cbc^{*}=(dc^*)^{*}dc^*$.
\end{proof}

The next proposition contains the approximations results we will need.

\begin{prop}\label{P:basic convergence log L2}
Let $(M,\tau)$ be a von Neumann algebra with a finite trace. Suppose that $(a_{n})_n$, $(b_{n})_n$ are sequences in $G,$ that $a\in \rL^0(M,\tau)$ and that $a_{n}^{\pm 1}\to a^{\pm 1},$ $b_{n}^{\pm 1}\to b^{\pm 1}$ in measure. Further assume that there are $A_{1},A_{2},B_{1},B_{2}\in \mathcal{P}$ with $A_{1}\leq |a_{n}|\leq A_{2},$ $B_{1}\leq |b_{n}|\leq B_{2}$ for all $n\in \N.$
\begin{enumerate}
\item Then $a\in G$ and $\|\log(|a_{n}|)\|_{2}\to_{n\to\infty} \|\log(|a|)\|_{2}.$ \label{I:conv in meas log L2}
\item If $a_{n}$ and $b_n \in \mathcal{P}$ for all $n,$ then $d_{\mathcal{P}}(a_{n},b_{n})\to_{n\to\infty}d_{\mathcal{P}}(a,b).$
\label{I:distance convergence}
 \end{enumerate}

\end{prop}

\begin{proof}
(\ref{I:conv in meas log L2}): As in (\ref{E:formula for norm}),
\[\|\log(a_{n})\|_{2}^{2}=2\int_{0}^{\infty}\lambda[\mu_{|a_{n}|}(e^{\lambda},\infty)+\mu_{|a_{n}^{-1}|}(e^{\lambda},\infty)]\,d\lambda.\]
Moreover, since $|a_{n}|\leq A_{2},$ we have that $\mu_{|a_{n}|}(\lambda,\infty)\leq \mu_{A_{2}}(\lambda,\infty).$ Let $a_{n}=u_{n}|a_{n}|$ be the polar decomposition. Since $a_n^{-1}=u_n^{-1} (u_n |a_n|^{-1}u_n^{-1})$, it follows that $|a_{n}^{-1}|=u_{n}|a_{n}|^{-1}u_{n}^{*}$. So by operator monotonicity of inverses, $|a_{n}^{-1}|\leq u_{n}A_{1}^{-1}u_{n}^{*}$ and thus by Proposition \ref{P:basic facts positive ops affil} (\ref{I:ineq op ineq meas}), \[\mu_{|a_{n}^{-1}|}(e^{\lambda},\infty)\leq \mu_{u_{n}A_{1}^{-1}u_{n}^{*}}(e^{\lambda},\infty).\] Since $a_{n}\in L^{0}(M,\tau)^{\times}$ we know that each $u_{n}$ is a unitary, so $\mu_{|a_{n}^{-1}|}(e^{\lambda},\infty)\leq \mu_{A_{1}^{-1}}(e^{\lambda},\infty)$.  Thus
\[\lambda[\mu_{|a_{n}|}(e^{\lambda},\infty)+\mu_{|a_{n}^{-1}|}(e^{\lambda},\infty)]\leq \lambda[ \mu_{A_{2}}(e^{\lambda},\infty)+\mu_{A_{1}^{-1}}(e^{\lambda},\infty)].\]
Since $A_{1},A_{2}\in \mathcal{P},$ the right hand side of this expression is in $L^{1}(\R).$

Since $a_{n}^{\pm 1}\to a^{\pm 1}$ in measure, Proposition \ref{P:conv in meas implies wk* conv} implies that $\mu_{|a_{n}^{\pm 1}|}(\lambda,\infty)\to \mu_{|a^{\pm 1}|}(\lambda,\infty)$ for all but countably many $\lambda.$ So by the Dominated Convergence Theorem,
\begin{align*}
\|\log(|a|)\|_{2}^{2}&=2\int_{0}^{\infty}\lambda[\mu_{|a|}((e^{\lambda},\infty))+\mu_{|a^{-1}|}((e^{\lambda},\infty))]\,d\lambda\\
&=\lim_{n\to\infty}2\int_{0}^{\infty}\lambda[\mu_{|a_{n}|}((e^{\lambda},\infty))+\mu_{|a_{n}^{-1}|}((e^{\lambda},\infty))]\,d\lambda\\
&=\lim_{n\to\infty}\|\log(|a_{n}|)\|_{2}^{2}.
\end{align*}

Moreover, we already saw that
\[2\int_{0}^{\infty}\lambda\left[\mu_{|a_{n}|}(e^{\lambda},\infty)+\mu_{|a_{n}^{-1}|}(e^{\lambda},\infty)\right]\,\dee\lambda\leq 2\int \lambda\left[ \mu_{A_{2}}(e^{\lambda},\infty)+\mu_{A_{1}^{-1}}(e^{\lambda},\infty)\right]\,\dee\lambda<\infty.\]
Thus $\log(|a|)\in \rL^2(M,\tau)$ and we have established that $\|\log(|a_{n}|)\|_{2}\to_{n\to\infty} \|\log(|a|)\|_{2}.$

(\ref{I:distance convergence}):
By definition,
\[d_{\mathcal{P}}(a_{n},b_{n})=\|\log(b_{n}^{-1/2}a_{n}b_{n}^{-1/2})\|_{2},\]
and as in (\ref{I:conv in meas log L2}) we have that
\begin{equation}\label{E:distance integral expression}
d_{\mathcal{P}}(a_{n},b_{n})=2\int_{0}^{\infty}\lambda\left[\mu_{b_{n}^{-1/2}a_{n}b_{n}^{-1/2}}(e^{\lambda},\infty)+\mu_{b_{n}^{1/2}a_{n}^{-1}b_{n}^{1/2}}(e^{\lambda},\infty)\right]\,\dee\lambda.
\end{equation}
By Proposition \ref{P:conv in meas implies wk* conv} (\ref{I:operate continuously measure}) we know that $b_{n}^{-1/2}a_{n}b_{n}^{-1/2}\to b^{-1/2}ab^{-1/2}$ in measure, and similarly $b_{n}^{1/2}a_{n}^{-1}b_{n}^{1/2}\to ba^{-1}b$ measure. Hence by Proposition \ref{P:conv in meas implies wk* conv} (\ref{I:wk* conv meas}) we know that
\begin{equation}\label{I:setting up dct for metric}
    \lim_{n\to\infty}\mu_{b_{n}^{-1/2}a_{n}b_{n}^{-1/2}}(e^{\lambda},\infty)+\mu_{b_{n}^{1/2}a_{n}^{-1}b_{n}^{1/2}}(e^{\lambda},\infty)=\mu_{b^{-1/2}ab^{-1/2}}(e^{\lambda},\infty)+\mu_{b^{1/2}a^{-1}b^{1/2}}(e^{\lambda},\infty),
\end{equation}
for all but countably many $\lambda.$ Moreover, by Proposition \ref{P:submutliplicative estimate} we have that
\[\mu_{b_{n}^{-1/2}a_{n}b_{n}^{-1/2}}(e^{\lambda},\infty)\leq 2\mu_{b_{n}^{-1/2}}(e^{\lambda/4},\infty)+\mu_{a_{n}}(e^{\lambda/2},\infty)=2\mu_{b_{n}^{-1}}(e^{\lambda/2},\infty)+\mu_{a_{n}}(e^{\lambda/2},\infty).\]
By operator monotonicity of inverses, we have that $b_{n}^{-1}\leq B_{1}^{-1}$ and so by Proposition \ref{P:basic facts positive ops affil} (\ref{I:ineq op ineq meas}) we have
\begin{equation}\label{I:first domination metric}
    \mu_{b_{n}^{-1/2}a_{n}b_{n}^{-1/2}}(e^{\lambda},\infty)\leq 2\mu_{B_{1}^{-1}}(e^{\lambda/2},\infty)+\mu_{A_{2}}(e^{\lambda/2},\infty).
\end{equation}
Similarly,
\begin{equation}\label{I:second domination metric}
    \mu_{b_{n}^{1/2}a_{n}^{-1}b_{n}^{1/2}}(e^{\lambda},\infty)\leq 2\mu_{B_{2}}(e^{-\lambda/2},\infty)+\mu_{A_{1}^{-1}}(e^{-\lambda/2},\infty).
\end{equation}
As in the proof of (\ref{I:conv in meas log L2}),
\[\lambda\mapsto \lambda\left[\mu_{B_{1}^{-1}}(e^{\lambda/2},\infty)+\mu_{B_{2}}(e^{-\lambda/2},\infty)+\mu_{A_{1}^{-1}}(e^{-\lambda/2},\infty)+\mu_{A_{2}}(e^{\lambda/2},\infty)\right]\]
is in $L^{1}(\R).$ So by   (\ref{I:first domination metric}),(\ref{I:second domination metric}), and (\ref{I:setting up dct for metric}) we may apply the dominated convergence theorem to (\ref{E:distance integral expression}) to see that
\[\lim_{n\to\infty}d_{\mathcal{P}}(a_{n},b_{n})=2\int_{0}^{\infty}\lambda\left[\mu_{b^{-1/2}ab^{-1/2}}(e^{\lambda},\infty)+\mu_{b^{1/2}a^{-1}b^{1/2}}(e^{\lambda},\infty)\right]\,\dee\lambda=d_{\mathcal{P}}(a,b).\]
\end{proof}

\begin{proof}[Proof of Theorem \ref{T:I heard you like metrics}]

(\ref{I:its a metric}):

We first prove non-degeneracy. So suppose that $a,b\in \mathcal{P}$ and $d_{\mathcal{P}}(a,b)=0.$ Then $\log(a^{-1/2}ba^{-1/2})=0,$ and so $a^{-1/2}ba^{-1/2}=1.$ Multiplying this equation on the left and right by $a^{1/2}$ proves that $b=a.$

For the triangle inequality, we already know by Lemma \ref{lem:commute} that $d_{\mathcal{P}}$ is a metric when restricted to $\GL^\infty(M,\tau)=M^{\times}.$ Here $M^{\times}$ is the set of elements of $M$ with a \emph{bounded} inverse.  Define $f_{n}\colon [0,\infty)\to [0,\infty)$ by
\[f_{n}(t)=\begin{cases}
n,& \textnormal{ if $t>n$}\\
t,& \textnormal{if $\frac{1}{n}\leq t\leq n$}\\
\frac{1}{n},& \textnormal{ if $0\leq t<\frac{1}{n}$.}\\
\end{cases}\]
Given $a,b\in \mathcal{P},$ set $a_{n}=f_{n}(a),b_{n}=f_{n}(b),$ $A=|\log(a)|,B=|\log(b)|$ and observe that:
\begin{itemize}
    \item $a_{n}^{\pm 1}\to a^{\pm 1},$ $b_{n}^{\pm 1}\to b^{\pm 1}$ in measure,
    \item $\exp(-A)\leq a_{n}\leq \exp(A),$ $\exp(-B)\leq b_{n}\leq \exp(B)$ for all $n\in \N.$
\end{itemize}
By Proposition \ref{P:basic convergence log L2} (\ref{I:distance convergence}), 
\[\lim_{n\to\infty}d_{\mathcal{P}}(a_{n},b_{n})=d_{\mathcal{P}}(a,b).\]
 Since $d_{\mathcal{P}}$ is a metric when restricted to $\cP\cap M^{\times},$ and $f_{n}(\mathcal{P})\subseteq \cP\cap M^{\times},$ the above equation implies the triangle inequality for $d_{\mathcal{P}}.$ It also implies $d_\cP$ is symmetric. So it is a metric.

(\ref{I:isometric action}):
It is easy to see that (\ref{I:isometric action}) is true if $g\in \mathcal{U}(M)$ (where $\mathcal{U}(M) \le M^\times$ is the group of unitaries in $M$). Every $g\in G$ can be written as $g=u|g|$ where $u\in \mathcal{U}(M).$ Since $|g|\in \mathcal{P}$ for every $g\in G,$ and (\ref{I:isometric action}) is true when $u\in \mathcal{U}(M),$ it suffices to show (\ref{I:isometric action}) for $g\in \mathcal{P}.$ So we will assume throughout the rest of the proof that $g\in \mathcal{P}.$

We first show that $d_{\mathcal{P}}(gag^{*},gbg^{*})=d_{\mathcal{P}}(a,b)$ for $a,b\in \mathcal{P}\cap M^{\times}.$ Since $g\in \mathcal{P},$ as in (\ref{I:its a metric}) we may find a sequence $g_{n}\in \mathcal{P}\cap M^{\times}$ so that
\begin{itemize}
    \item $g_{n}^{\pm 1}\to g^{\pm 1}$ in measure,
    \item $g_{n}=f_{n}(g)$ for some $f_{n}\colon [0,\infty)\to [0,\infty)$
    \item $\exp(-H)\leq g_{n}\leq \exp(H)$ for some self-adjoint $H\in \rL^2(M,\tau).$
\end{itemize}
Since $\rL^0(M,\tau)$ is a topological $*$-algebra in the measure topology, we have that $g_{n}ag_{n}\to_{n\to\infty}gag$ in measure. Moreover by Proposition \ref{P:basic facts positive ops affil} (\ref{I:this is easy})
\[\|a^{-1}\|^{-1}_{\infty}\exp(-2H)\leq \|a^{-1}\|_{\infty}^{-1}g_{n}g_{n}\leq g_{n}ag_{n}\leq \|a\|_{\infty}g_{n}^{2}\leq \|a\|_{\infty}\exp(2H),\]
and similarly
\[\|b^{-1}\|^{-1}_{\infty}\exp(-2H)\leq g_{n}bg_{n}\leq \|b\|_{\infty}\exp(2H).\]
So as in (\ref{I:its a metric}) we may apply Proposition \ref{P:basic convergence log L2} (\ref{I:distance convergence}) to see that
\begin{equation}\label{E:were halfway there} d_{\mathcal{P}}(gag,gbg)=\lim_{n\to\infty}d_{\mathcal{P}}(g_{n}ag_{n},g_{n}bg_{n}).
\end{equation}
By Lemma \ref{lem:isometries},  $d_{\mathcal{P}}(g_{n}ag_{n},g_{n}bg_{n})= d_\cP(a,b)$. We thus have that
\[d_{\mathcal{P}}(gag,gbg)=d_{\mathcal{P}}(a,b).\]

We now handle the case of general $a,b\in\mathcal{P}.$ As in (\ref{I:its a metric}), we find may sequences $a_{n},b_{n}\in \mathcal{P}\cap M^{\times}$ so that:
\begin{itemize}
    \item $a_{n}^{\pm 1}\to a^{\pm 1},$ $b_{n}\to b^{\pm 1}$ in measure
    \item $\exp(-A)\leq a_{n}\leq \exp(A),$ $\exp(-B)\leq b_n \leq \exp(B)$ for some $A,B\in \rL^2(M,\tau)$.
\end{itemize}
As in (\ref{I:its a metric}), we have that
\begin{equation}\label{E:not this again}
d_{\mathcal{P}}(a,b)=\lim_{n\to\infty}d_{\mathcal{P}}(a_{n},b_{n}).
\end{equation}
\[d_{\mathcal{P}}(gag^{*},gbg^{*})=\lim_{n\to\infty}d_{\mathcal{P}}(ga_{n}g^{*},gb_{n}g^{*}).\]
So combining (\ref{E:not this again}) with the first case shows that
\[d_{\mathcal{P}}(gag^{*},gbg^{*})=d_{\mathcal{P}}(a,b).\]

(\ref{I:transitive}) Let $p,q \in \cP$. Then $p^{-1/2},q^{1/2} \in G= \GL^2(M,\tau)$. Moreover,
$$(q^{1/2}p^{-1/2}) \cdot p = q.$$

(\ref{I:dense}) Let $a \in \cP$ and define $a_n=f_n(a)$ as in (\ref{I:its a metric}). Then $a_n \in \cP^\infty$ and $a_n \to a$ in measure. Apply Proposition \ref{P:basic convergence log L2} with $b_n=a = B_1=B_2$ to obtain $d_\cP(a_n,a) \to d_\cP(a,a) = 0$ as $n\to\infty$. Since $a \in \cP$ is arbitrary, this proves $\cP^\infty$ is dense in $\cP$.

\end{proof}

\subsection{Continuity of the exponential map}\label{sec:P2}

This section proves that the exponential map $\exp:\rL^2(M,\tau)_{sa} \to \cP$ is a homeomorphism and obtains as a corollary that $\cP$ is a complete CAT(0) metric space. We also obtain a formula for the geodesics in $\cP$. First we need the following estimate which extends the $\cP^\infty$ case proven earlier.

\begin{prop}\label{P:continuity of logarithm map}
Let $(M,\tau)$ be a von Neumann algebra with a finite trace. Then for all $a,b\in \rL^2(M,\tau)_{sa}$,
$$\|a-b\|_{2}\leq d_{\mathcal{P}}(e^{a},e^{b}).$$
 If $a$ and $b$ commute then $\|a-b\|_{2}=d_{\mathcal{P}}(e^{a},e^{b}).$
\end{prop}

\begin{proof}
Define a function $f_{n}\colon \R\to \R$ by
\[f_{n}(t)=\begin{cases}
n,& \textnormal{ if $t>n$}\\
t,& \textnormal{if $|t|\leq n$}\\
-n,& \textnormal{ if $t<-n$}.
\end{cases}
\]
Set $a_{n}=f_{n}(a),b_{m}=f_{n}(b).$ Then:
\begin{itemize}
\item $e^{a_{n}}\to e^{a},$ $e^{b_{n}}\to e^{b}$ in measure,
\item $\exp(-A)\leq e^{a_{n}}\leq \exp(A),$ $\exp(-B)\leq e^{b_{n}}\leq \exp(B)$ for all $n\in \N.$
\end{itemize}
So as in Theorem \ref{T:I heard you like metrics} (\ref{I:its a metric}) we have that
\[d_{\mathcal{P}}(e^{a},e^{b})=\lim_{n\to\infty}d_{\mathcal{P}}(e^{a_{n}},e^{b_{n}}).\]
Additionally, it is direct to see from the spectral theorem that
\[ \lim_{n\to\infty} \|a-a_{n}\|_{2}= \lim_{n\to\infty}\|b-b_{n}\|_{2}= 0.\]
So, by Lemma \ref{lem:commute},
\[d_{\mathcal{P}}(e^{a},e^{b})=\lim_{n\to\infty}d_{\mathcal{P}}(e^{a_{n}},e^{b_{n}})\geq \lim_{n\to\infty}\|a_{n}-b_{n}\|_{2}=\|a-b\|_{2}.\]
Suppose $a$ and $b$ commute. By definition
$$d_{\mathcal{P}}(e^{a},e^{b}) = \|\log(e^{-b/2}e^ae^{-b/2})\|_2.$$
Since $a$ and $b$ commute, $ e^{-b/2}e^ae^{-b/2} = e^{a-b}$. So $\|\log(e^{-b/2}e^ae^{-b/2})\|_2 = \|a-b\|_2$.

\end{proof}




\begin{thm}\label{T:continuity of exp}
Let $(M,\tau)$ be a von Neumann algebra with a finite trace. Then the exponential map $\exp\colon \rL^2(M,\tau)_{sa}\to \mathcal{P}$ is a homeomorphism.

\end{thm}

\begin{proof}
By Proposition \ref{P:continuity of logarithm map}, we know that $\log\colon \mathcal{P}\to \rL^2(M,\tau)_{sa}$ is continuous. So it just remains to show that $\exp\colon \rL^2(M,\tau)_{sa}\to \mathcal{P}$ is continuous.

Suppose that $(a_{n})_n$ is a sequence in $\rL^2(M,\tau)$ and $a\in \rL^2(M,\tau)$ with $\|a-a_{n}\|_{2}\to 0.$ Let $\varepsilon>0,$ and for $\lambda>0$, define $f_{\lambda}\colon \R\to \R$ by
\[f_{\lambda}(t)=\begin{cases}
\lambda,& \textnormal{ if $t>\lambda$}\\
t,& \textnormal{ if $|t|\leq \lambda$ }\\
-\lambda,& \textnormal{if $t<-\lambda.$}
\end{cases}\]
If $\lambda>0$ is large enough, then $\|a-f_{\lambda}(a)\|_{2}<\varepsilon.$ Fix such a choice of $\lambda.$

Since $a$ and $f_\l(a)$ commute,
\begin{align}\label{E:cutoff functions estimate}
d_{\mathcal{P}}(e^{a_{n}},e^{a})&\leq d_{\mathcal{P}}(e^{a},e^{f_{\lambda}(a)})+d_{\mathcal{P}}(e^{f_{\lambda}(a_{n})},e^{f_{\lambda}(a)})+d_{\mathcal{P}}(e^{a_{n}},e^{f_{\lambda}(a_{n})})\\ \nonumber
&=\|a-f_{\lambda}(a)\|_{2}+\|a_{n}-f_{\lambda}(a_{n})\|_{2}+d_{\mathcal{P}}(e^{f_{\lambda}(a_{n})},e^{f_{\lambda}(a)}).
\end{align}

Since $a_n \to a$ in $\rL^2(M,\tau)$, $a_n \to a$ in measure. By Proposition \ref{P:conv in meas implies wk* conv}  (\ref{I:operate continuously bounded functions}), 
$\lim_{n\to\infty} \|f_{\lambda}(a_{n})-f_{\lambda}(a)\|_{2} = 0.$
Furthermore, $\max(\|f_{\lambda}(a_{n})\|_{\infty},\|f_{\lambda}(a)\|_{\infty})\leq \lambda$ for all $n\in \N.$

By Proposition \ref{P:conv in meas implies wk* conv} (\ref{I:operate continuously measure}), $e^{-f_{\lambda}(a_{n})/2} \to e^{-f_{\lambda}(a)/2}$ in measure. Since $\rL^0(M,\tau)$ is a topological $\ast$-algebra in the measure topology, $e^{-f_{\lambda}(a_{n})/2}e^{f_{\lambda}(a)}e^{-f_{\lambda}(a_{n})/2} \to 1$ in measure. We claim that
$$\log(e^{-f_{\lambda}(a_{n})/2}e^{f_{\lambda}(a)}e^{-f_{\lambda}(a_{n})/2}) \to 0$$
in measure. To see this, observe that
$$e^{-2\l} \le e^{-f_{\lambda}(a_{n})/2}e^{f_{\lambda}(a)}e^{-f_{\lambda}(a_{n})/2} \le e^{2\l}.$$
Choose a continuous function $\phi:\R \to \R$ with $\phi(x)=\log(x)$ for all $e^{-2\l} \le x \le e^{2\l}$. Then $\phi(e^{-f_{\lambda}(a_{n})/2}e^{f_{\lambda}(a)}e^{-f_{\lambda}(a_{n})/2}) = \log (e^{-f_{\lambda}(a_{n})/2}e^{f_{\lambda}(a)}e^{-f_{\lambda}(a_{n})/2})$. So the claim follows from  Proposition \ref{P:conv in meas implies wk* conv} (\ref{I:operate continuously measure}).

By Proposition \ref{P:3topologies}, the claim above implies $\log(e^{-f_{\lambda}(a_{n})/2}e^{f_{\lambda}(a)/2}e^{-f_{\lambda}(a_{n})/2}) \to 0$ in $\rL^2(M,\tau)$. Since
$$d_{\mathcal{P}}(e^{f_{\lambda}(a_{n})},e^{f_{\lambda}(a)}) = \|\log(e^{-f_{\lambda}(a_{n})/2}e^{f_{\lambda}(a)/2}e^{-f_{\lambda}(a_{n})/2})\|_2$$
this shows
\begin{equation}\label{E:middle term cutoff estimate}
d_{\mathcal{P}}(e^{f_{\lambda}(a_{n})},e^{f_{\lambda}(a)})\to_{n\to\infty}0.
\end{equation}
Since $a_n \to a$ and $f_\l(a_n) \to f_\l(a)$ in $\rL^2(M,\tau)$,
$\|a_{n}-f_{\lambda}(a_{n})\|_{2}\to_{n\to\infty}\|a-f_{\lambda}(a)\|_{2}.$
Combining with (\ref{E:middle term cutoff estimate}), (\ref{E:cutoff functions estimate}) we have shown that
\[\limsup_{n\to\infty}d_{\mathcal{P}}(e^{a_{n}},e^{a})\leq 2\|a-f_{\lambda}(a)\|_{2}<2\varepsilon.\]
Letting $\varepsilon\to 0$ proves that
$d_{\mathcal{P}}(e^{a_{n}},e^{a})\to 0.$

\end{proof}

\begin{cor}\label{C:complete}
Let $(M,\tau)$ be a von Neumann algebra with a finite trace. Then $(\mathcal{P},d_{\mathcal{P}})$ is a complete metric space.

\end{cor}

\begin{proof}
Let $(a_{n})$ be a Cauchy sequence in $\mathcal{P}.$ Set $b_{n}=\log(a_{n}).$ By Proposition \ref{P:continuity of logarithm map}, we know that $(b_{n})$ is Cauchy in $\rL^2(M,\tau).$ By completeness of $\rL^2(M,\tau),$ there is a $b\in \rL^2(M,\tau)$ with $\|b_{n}-b\|_{2}\to_{n\to\infty}0.$ Then $a=e^{b}\in \mathcal{P},$ and by Theorem \ref{T:continuity of exp} we know that $a_{n}=e^{b_{n}}\to e^{b}=a.$
\end{proof}

\begin{cor}\label{C:cat0}
If $(M,\tau)$ is finite then $\cP$ is CAT(0).
\end{cor}

\begin{proof}
Recall that $\cP^\infty$ is CAT(0) by Theorem \ref{thm:cat}. By Theorem \ref{T:I heard you like metrics} $\cP^\infty$ is dense in $\cP$.  Because metric completions of CAT(0) spaces are CAT(0) by \cite[II.3, Corollary 3.11]{bridson-haefliger-book}, this implies $\cP$ is CAT(0). 
\end{proof}

\begin{cor}\label{C:weaker topology}
Let $(M,\tau)$ be a von Neumann algebra with a finite trace. Then the measure topology on $\mathcal{P}(M,\tau)$ is weaker than the $d_{\mathcal{P}}$-topology.

\end{cor}

\begin{proof}
Let $(b_{n})_n$ be a sequence in $\mathcal{P}(M,\tau)$ and $b\in \mathcal{P}(M,\tau)$ with $\lim_{n\to\infty} d_\cP(b_{n},b)=0.$ Let $a_{n}=\log b_{n}$, $a=\log(b)$. Then $\|a_{n}-a\|_{2}\to_{n\to\infty}0,$ since the logarithm map is continuous. So $a_{n}\to a$ in measure. But then by applying the exponential map in Proposition \ref{P:conv in meas implies wk* conv} (\ref{I:operate continuously measure}) we have that $b_{n}\to b$ in measure.
\end{proof}

\begin{cor}\label{C:geodesics}
For $\xi \in \rL^2(M,\tau)_{sa}$, the map $\g_\xi:\R \to \cP$ defined by
$$\g_\xi(t) = \exp(t\xi)$$
is a minimal geodesic with speed $\|\xi\|_2$. Moreover every geodesic $\g$ with $\g(0)=\id$ is equal to $\g_\xi$ for some $\xi$. Moreover, for any $a,b \in \cP$, the unique unit-speed geodesic from $a$ to $b$ is the map $\g:[0,d_\cP(a,b)] \to \cP$ defined by
$$\g(t) = a^{1/2}\g_\xi(t)a^{1/2}$$
where
\begin{eqnarray}\label{xixi}
\xi = \frac{\log (a^{-1/2}ba^{-1/2}) }{\|\log (a^{-1/2}ba^{-1/2})\|_2} = \frac{\log (a^{-1/2}ba^{-1/2}) }{d_\cP(a,b)}.
\end{eqnarray}

\end{cor}

\begin{proof}

 For any $t>0$,
$$d_\cP(\id, \g_\xi(t)) = \| \log \g_\xi(t) \|_2 = t\|\xi\|_2.$$
This proves $\g_\xi$ is a minimal geodesic with speed $\|\xi\|_2$. Because $\cP$ is CAT(0), there is a unique unit-speed geodesic between any two points.  By uniqueness of geodesics, every geodesic $\g$ with $\g(0)=\id$ has the above form.

In particular, if $a,b \in \cP$ and $\xi$ is defined by (\ref{xixi}) then $\g_\xi:[0,d_\cP(a,b)] \to \cP$ is a unit-speed geodesic from $\id$ to $a^{-1/2}ba^{-1/2}$. Because the action of $\GL^2(M,\tau)$ on $\cP$ is by isometries, $\g(t)=a^{1/2}.\g_\xi(t)$ is a unit-speed geodesic from $a=a^{1/2}.\id$ to $b=a^{1/2}.a^{-1/2}ba^{-1/2}$.

\end{proof}

\subsection{The semi-finite case}\label{sec:semi-finite case}

Let $(M,\tau)$ be a semi-finite tracial von Neumann algebra.   Let $G=\GL^2(M,\tau)$ and $\cP = \cP(M,\tau) = \exp(L^2_{sa}(M,\tau))$ as before.  We want to show that $d_\cP(a,b) := \|\log(b^{-1/2}ab^{-1/2})\|_2$ is a distance function which makes $\cP$ into a complete CAT(0) space.  Since we have shown this fact when $\tau(\id)$ is finite, our approach will often involve reducing to the finite case.  To this end we first need to identify the following objects.




For a finite projection $p \in M$, observe that $(pMp,\tau \circ p)$ is a von Neumann algebra with a finite trace. Let $\cP_p = \exp(\rL^2_{sa}(pMp,\tau \circ p)) \subset \rL^0(pMp, \tau \circ p)$.  For $a,b \in \cP_p$ define $d_{\cP_p}(a,b) = ||\log(b^{-1/2}ab^{-1/2})||_{\rL^2(pMp,\tau \circ p)}$. Since $(pMp,\tau \circ p_n)$ is finite, Theorem \ref{T:I heard you like metrics} implies $d_{\cP_p}$ is a metric and Corollary \ref{C:cat0} implies $\cP_p$ is complete CAT(0).

Also define $\tcP_p= \exp(p\rL^2_{sa}(M,\tau)p)\subset \cP(M, \tau)$  and define the metric $d_{\tcP_p}$ of $\tcP_p$ to be the restriction of  $d_\cP$ to  $\tcP_p$.  The next proposition implies Theorem \ref{thm:positivecone}.

\begin{prop}\label{P:semi-finite CAT(0)}
The following are true:
\begin{enumerate}
    \item For every projection $p\in M,$ the inclusion $pMp\hookrightarrow M$ extends to a $*$-isomorphism of topological $*$-algebras $\iota\colon L^{0}(pMp,\tau\big|_{pMp})\cong pL^{0}(M,\tau)p.$  Further $\mu_{|\iota(x)|}=\mu_{|x|},$ so in particular $\iota$ induces an isometry $\cP_{p}\to \tilde{\cP}_{p}.$
\label{I:isom ident}
    \item $d_\cP$ is a metric. \label{I: semi-finite metric}
    \item $G$ acts on $\cP$ by isometries. \label{I: semi-finite isom action}
    \item $G$ acts on $\cP$ transitively.\label{I: semi-finite trans action}
    \item $\cP$ is complete.\label{I: semi-finite completeness}
    \item  $\cP$ is CAT(0).\label{I:semi-finite cat(0)}
    \item Let $\cP^\infty = \exp(M_{sa} \cap \rL^2(M,\tau)_{sa})$.  Then $\cP^\infty$ is dense in $\cP$ and $\cP^\infty=\cP \cap M^\times$. \label{I: P infinity dense}
\end{enumerate}

\end{prop}

\begin{proof}[Proof of Proposition \ref{P:semi-finite CAT(0)}(\ref{I:isom ident})]
Let $j\colon pMp\to M$ the inclusion map. Tautologically,
\[\tau(1_{(\lambda,\infty)}(|j(x)|))=\tau(1_{(\lambda,\infty)}(|x|)),\]
and the above equality implies that $j$ extends to a linear map
\[\iota\colon L^{0}(pMp,\tau\big|_{pMp})\to L^{0}(M,\tau)\]
with closed image,
and that this map is a homeomorphism onto its image. Moreover $j(pMp)=pMp,$ and so passing to closures we have $\iota(L^{0}(pMp,\tau\big|_{pMp}))=pL^{0}(M,\tau)p.$
By uniform continuity and the fact that $\iota$ is a $*$-homomorphism on a dense $*$-subalgebra (namely, $pMp$) it follows that $\iota$ is a $*$-homomorphism. The fact that $\mu_{|\iota(x)|}=\mu_{|x|}$ for all $x\in L^{0}(pMp,\tau\big|_{pMp})$ follows from the fact that it is true for $x\in pMp$ and Proposition \ref{P:conv in meas implies wk* conv} (\ref{I:wk* conv meas}). Since $\mu_{|\iota(x)|}=\mu_{|x|}$, we know
\[\|\log(|x|)\|_{2}=\|\log(|\iota(x)|)\|_{2}\]
for all $x\in pMp.$ So $\iota(\cP_{p})=\tilde{\cP}_{p}.$ Moreover, the fact that $\iota$ is $*$-homomorphism and the above equality implies that $\iota\colon \cP_{p}\to\tilde{\cP}_{p}$ is an isometry.

\end{proof}



To prove (\ref{I: semi-finite metric}) of Proposition \ref{P:semi-finite CAT(0)}, we take an approximation approach similar to that of the finite case.  Although we do not have all the tools available in the finite case such as Proposition \ref{P:conv in meas implies wk* conv}, we still have enough to work with.  We first state the tools that we will be using.




\begin{lem}\label{tikhonov}
Suppose $x_k, x \in \rL^0_{sa}(M,\tau)$ and $x_k \to_{k\to\infty} x$ in measure and $f:\R \to \R$ is a Borel function continuous on the spectrum of $x$ and bounded on bounded subsets of $\R$.  Then $f(x_k) \to_{k\to\infty} f(x)$ in measure.
\end{lem}
\begin{proof}
This is implied by \cite[Theorem 2.4]{MR892008}.
\end{proof}

\begin{lem}\label{cdfcont}
Suppose $x_k,x \in \rL^0(M,\tau)$ and $x_k \to_{k\to\infty}  x$ in measure.  Suppose $\l \mapsto \mu_{|x|}(\l,\infty)$ is continuous at $\l_0$. Then $\mu_{|x_k|}(\l_0,\infty) \to_{k\to\infty}  \mu_{|x|}(\l_0,\infty)$.
\end{lem}
\begin{proof}
Let $\l \mapsto \mu_{|x|}(\l,\infty)$ be continuous at $\l_0$. By Proposition \ref{P:subadditive estimate}, for any $0<\d< \l_0$ and $k \in \N$,
\begin{eqnarray*}
\mu_{|x|}(\l_0+\d,\infty) &\leq& \mu_{|x-x_k|}(\delta,\infty) + \mu_{|x_k|}(\l_0,\infty);\\
\mu_{|x_k|}(\l_0,\infty) &\leq& \mu_{|x_k-x|}(\d,\infty) + \mu_{|x|}(\l_0-\d,\infty).
\end{eqnarray*}
Since $x_k \to x$ in measure (as $k\to\infty$),
\begin{eqnarray*}
 \mu_{|x|}(\l_0+\d,\infty) &\leq& \liminf_{k\to\infty} \mu_{|x_k|}(\l_0,\infty)\\
\limsup_{k\to\infty} \mu_{|x_k|}(\l_0,\infty) &\leq& \mu_{|x|}(\l_0-\d,\infty).
   \end{eqnarray*}
Since $\l_0$ is a point of continuity, this implies
$$\limsup_{k\to\infty} \mu_{|x_k|}(\l_0,\infty) \leq\mu_{|x|}(\l_0,\infty)\leq\liminf_{k\to\infty} \mu_{|x_k|}(\l_0,\infty)$$
which implies the lemma.

\end{proof}

We now prove that $d_\cP$ satisfies the triangle inequality and symmetry properties; the identity property is similar to the finite case.  We do this by approximating via elements from a ``reduced" von Neumann algebra with a finite trace.

\begin{notation}\label{notation1}
For any $x \in \rL^2_{sa}(M,\tau)$ and $n\in \N$, let $p_n^x= 1_{(-\infty, -1/n)\cup (1/n,\infty)}(x)$ and $x_n := p_n^x x p_n^x = xp_n^x$.  Then $x_n$ is an increasing sequence converging in measure to $x$ (as $n\to\infty$). Because $x \in \rL^2_{sa}(M,\tau)$, $p_n^x$ is a finite projection.
\end{notation}

\begin{prop}\label{finite approx of cP}
Suppose $z,w \in \cP$ and we can write $z = e^{x_1}\cdots e^{x_k}$, $w = e^{y_1}\cdots e^{y_l}$ for some $x_i,y_j \in \rL^2_{sa}(M,\tau)$. Let
\begin{eqnarray*}
x_{i,n} = p_n^{x_i} x_i p_n^{x_i}, &\quad& z_n = e^{x_{1,n}}\cdots e^{x_{k,n}}\\
y_{j,n} = p_n^{y_j} y_j p_n^{y_j}, &\quad& w_n = e^{y_{1,n}}\cdots e^{y_{l,n}}.
\end{eqnarray*}
Assume $z_n$ and $w_n$ are positive for all $n$. Then $d_\cP(z_n,w_n) \to d_\cP(z,w)$ and  $d_\cP(z_n,z) \to 0$ as $n \to\infty$.
\end{prop}
\begin{proof}

We will just prove $d_\cP(z_n,w_n) \to_{n\to\infty} d_\cP(z,w)$ since the proof that $d_\cP(z_n,z)_{n\to\infty} \to 0$ is similar.

Let $q_n=z_n^{-1/2}w_nz_n^{-1/2}$ and $q=z^{-1/2}wz^{-1/2}$. Because $x_{i,n} \to x_i$ in measure, the exponential map is continuous and bounded on bounded subsets, and $\rL^0(M,\tau)$ is a topological *-algebra, it follows from Lemma \ref{tikhonov} that $z_n \to z$ in measure as $n\to\infty$.  Similarly $z_n^{-1} \to z^{-1}$, $w_n \to w$ and $w_n^{-1}\to w$ in measure as $n\to\infty$. It follows from Lemma \ref{tikhonov} that $z_n^{-1/2} \to z^{-1/2}$ in measure. Since $\rL^0(M,\tau)$ is a topological *-algebra in the measure topology, $q_n\to q$ in measure.

Next use Proposition \ref{P:submutliplicative estimate} and operator monotonicity (Proposition \ref{P:basic facts positive ops affil}) to get (for any $\l>0$)
\begin{eqnarray}\label{eqn:multiplication-bound}
\mu_{z^{\pm 1}_n}(e^\l,\infty) &\leq&   \sum_{i=1}^k \mu_{e^{\pm x_{i,n}}}(e^{\l/k},\infty) = \sum_{i=1}^k\mu_{ \pm x_{i,n}}(\l/k,\infty) \\
&\leq& \sum_{i=1}^k\mu_{| x_{i,n}|}(\l/k,\infty) = \sum_{i=1}^k\mu_{k| x_{i,n}|}(\l,\infty)
\leq \sum_{i=1}^k \mu_{k|x_i|}(\l,\infty).
\end{eqnarray}
A similar calculation shows that $\mu_{w^{\pm 1}_n}(e^\l,\infty)  \leq \sum_{j=1}^l \mu_{l|y_j|}(\l,\infty)$.

Next use Proposition \ref{P:submutliplicative estimate} to get (for any $\l>0$) 
\begin{eqnarray*}
\mu_{q_n}(e^\l,\infty) &\leq&   2\mu_{z_n^{-1/2}}(e^{\l/3},\infty) + \mu_{w_n}(e^{\l/3},\infty)\\
&=&  2\mu_{z^{-1}_n}(e^{2\l/3},\infty) + \mu_{w_n}(e^{\l/3},\infty)\\
&\leq& \sum_{i=1}^k \mu_{k|x_i|}(2\l/3,\infty) + \sum_{j=1}^l \mu_{l|y_j|}(\l/3,\infty) \\
&=& \sum_{i=1}^k \mu_{(3/2)k|x_i|}(\l,\infty) + \sum_{j=1}^l \mu_{3l|y_j|}(\l,\infty)
\end{eqnarray*}
where the last inequality follows from (\ref{eqn:multiplication-bound}).

Note $q_n^{-1}=z_n^{1/2} w_n^{-1} z_n^{1/2}$. So by a similar computation we obtain
\begin{eqnarray*}
\mu_{q_n^{-1}}(e^\l,\infty) &\leq&    \sum_{i=1}^k \mu_{(3/2)k|x_i|}(\l,\infty) + \sum_{j=1}^l \mu_{3l|y_j|}(\l,\infty).
\end{eqnarray*}

Now since each $x_i \in \rL^2_{sa}(M,\tau)$, by equation (\ref{E:simpler formula for norm}) we conclude that
$$\l \mapsto  2\l\left(\sum_{i=1}^k \mu_{(3/2)k|x_i|}(\l,\infty)  + \sum_{j=1}^l \mu_{3l|y_j|}(\l,\infty)\right)$$
 is integrable.  It follows that $\l \mapsto \l(\mu_{q_n}(e^\l,\infty) + \mu_{q_n^{-1}}(e^\l,\infty))$ is dominated by an integrable function, so that by the Dominated Convergence Theorem, Lemma \ref{cdfcont} and equation (\ref{E:formula for norm}), $\| \log q_n\|_2 \to \| \log q \|_2$. By definition, $\| \log q_n\|_2  = d_\cP(z_n,w_n)$ and $\| \log q\|_2  = d_\cP(z, w)$,  so this implies the proposition.
\end{proof}


We can now prove $d_\cP$ is a metric.

\begin{proof}[Proof of Proposition \ref{P:semi-finite CAT(0)}(\ref{I: semi-finite metric})]
To prove the triangle inequality in $(\cP,d_\cP)$, suppose $e^x,e^y,e^v \in \cP$ with $x,y,v \in \rL^2_{sa}(M,\tau)$. Define $x_n,y_n,v_n \in M$ as above. Let $p_n = p_n^x \vee p_n^y \vee p_n^v$ be the smallest projection dominating $p_n^x, p_n^y, p_n^v$. By \cite{MR641217}[Part III, Ch. 2, Prop. 5], $p_n$ is a finite projection in $M$. So $x_n,y_n,v_n$ are in the finite von Neumann sub-algebra $p_nMp_n$.  By Theorem \ref{T:I heard you like metrics}(\ref{I:its a metric}),
$$d_{\cP}(e^{x_n},e^{v_n}) \leq d_{\cP}(e^{x_n},e^{y_n}) + d_{\cP}(e^{y_n},e^{v_n}).$$
By Proposition \ref{finite approx of cP}, this implies the triangle inequality $d_{\cP}(e^{x},e^{v}) \leq d_{\cP}(e^{x},e^{y}) + d_{\cP}(e^{y},e^{v}).$

Similarly, the symmetry $d_{\cP}(e^{x},e^{v}) = d_{\cP}(e^{v},e^{x})$ follows from by taking the limit as $n\to\infty$ in   $d_{\cP}(e^{x_n},e^{v_n}) = d_{\cP}(e^{v_n},e^{x_n})$. Lastly, if $d_\cP(e^x,e^y)=0$ then, by definition, $\log(e^{-x/2}e^ye^{-x/2})=0$ which implies $e^{-x/2}e^ye^{-x/2}=1$ which implies $e^x=e^y$.

\end{proof}



\begin{cor}\label{finite approx}
Let $\Proj \subset M$ denote the set of finite projections in $M$. Then $\cup_{p \in \Proj}\cP_{p}$ is dense in $\cP$.
\end{cor}
\begin{proof}
For any $x \in \rL^2(M,\tau)_{sa}$ and $n\in \N$, $e^{x_n} \in  \cup_{p \in \Proj}\cP_{p}$ and $\lim_{n\to\infty} e^{x_n} = e^x$ in $(\cP,d_\cP)$ by Proposition \ref{finite approx of cP}.
\end{proof}

To prove Proposition \ref{P:semi-finite CAT(0)}(\ref{I: semi-finite isom action}), we first show that unitary elements in $M$ act by isometries on $\cP$.  Then, by polar decomposition, it suffices to consider the action of $\cP$ on $\cP$.

\begin{lem}\label{lem:unitary-action}
Let  $x,y \in \rL^2_{sa}(M,\tau)$ and $u \in M$ is unitary. Then
\begin{enumerate}
\item  $(ue^x u^*)^{-1}  = ue^{-x} u^*$;
\item  $(ue^{x} u^*)^{1/2}  = ue^{x/2} u^*$;
\item $e^{uxu^*} = u e^x u^*$;
\item $d_\cP(ue^x u^*, u e^y u^*) = d_\cP(e^x,e^y)$.
\end{enumerate}
\end{lem}

\begin{proof}
The first claim is obvious. The second follows from observing that $ue^{x/2} u^*$ is positive and its square is $ue^{x} u^*$.

For the third claim, first consider a sequence $x_k \in M$ converging to $x$ in $\rL^2$.  For each $x_k$, because $M$ is a unital Banach algebra,
$$e^{ux_ku^*} = \sum_{n=0}^\infty \frac{ (ux_ku^*)^n }{n!} =  \sum_{n=0}^\infty \frac{ ux_k^n u^* }{n!} = u e^{x_k} u^*.$$

By Lemma \ref{L:L2tomeasure},  $x_k \to x$ in measure as $k \to\infty$. By Lemma \ref{tikhonov}, $ue^{x_k}u^* \to ue^xu^*$ in measure, and also $e^{ux_ku^*} \to e^{uxu^*}$.  But since $ue^{x_k}u^* = e^{ux_ku^*}$, by uniqueness of limits $ue^xu^* = e^{uxu^*}$.

The last claim now follows using the previous three claims:
\begin{eqnarray*}
d_\cP(ue^x u^*, u e^y u^*) &=& \| \log [ (ue^x u^*)^{-1/2}u e^y u^*(ue^x u^*)^{-1/2}] \|_2 \\
&=& \| \log [ (ue^{-x/2} u^*) u e^y u^*(ue^{-x/2} u^*)] \|_2 \\
&=& \| \log [ ue^{-x/2}  e^y e^{-x/2} u^*] \|_2 \\
&=& \| u \log [ e^{-x/2}  e^y e^{-x/2} ] u^* \|_2 \\
&=&  \tau(u \log [ e^{-x/2}  e^y e^{-x/2} ]^* \log [ e^{-x/2}  e^y e^{-x/2} ] u^*)^{1/2} \\
&=&  \tau( \log [ e^{-x/2}  e^y e^{-x/2} ]^* \log [ e^{-x/2}  e^y e^{-x/2} ] )^{1/2}  \\
&=& \| \log [ e^{-x/2}  e^y e^{-x/2} ]  \|_2 \\
&=& d_\cP(e^x,e^y).
\end{eqnarray*}
The first equality is by definition of $d_\cP$. The second follows from the first two claims above. The third equality uses $uu^*=1$. The fourth follows from the third item of this lemma. The fifth is by definition of $\| \cdot \|_2$. The sixth holds because $\tau$ is a trace.
\end{proof}

\begin{proof}[Proof of Proposition \ref{P:semi-finite CAT(0)}(\ref{I: semi-finite isom action})]
By Lemma \ref{lem:unitary-action} and polar decomposition it suffices to consider the action of $\cP$ on $\cP$.  Let $g, a,b \in \cP$, where $g = e^h, a = e^x, b = e^y$, $h, x,y \in \rL^2_{sa}(M,\tau)$. We want to show that $d_{\cP}(gag^*,gbg^*) = d_{\cP}(e^x,e^y)$.  As before consider reduced versions $h_n, x_n, y_n$ of $h,x,y$.
Let $g_n = e^{h_n}$, $a_n = e^{x_n}$, $b_n = e^{y_n}$.
By Proposition \ref{finite approx of cP}, $d_\cP(g_na_ng_n^*,g_nb_ng_n^*) \to d_\cP(gag^*,gbg^*)$.



Now $d_\cP(g_na_ng_n,g_nb_ng_n) = d_\cP(a_n,b_n) \to d_\cP(a,b)$ (equality because we are again in a von Neumann algebra with a finite trace and so Theorem \ref{T:I heard you like metrics}(\ref{I:isometric action}) applies, and convergence by Proposition \ref{finite approx of cP}), so it must be that $d_\cP(gag,gbg) = d_\cP(a,b)$.
\end{proof}

The proof of Proposition \ref{P:semi-finite CAT(0)}(\ref{I: semi-finite trans action}) is the same argument as in the finite case.

To prove $(\cP,d_\cP)$ is complete, we first show that $\exp$ is a homeomorphism from $\rL^2(M,\tau)_{sa}$ to $\cP$.




\begin{lem}\label{semi-finite cont of log}
Let $x,y \in \rL^2(M,\tau)_{sa}$.  Then $\|x-y\|_2 \le d_\cP(e^x,e^y)$.
\end{lem}
\begin{proof}
As above, consider $x_n = p^x_nxp^x_n$ and $y_n=p^y_nyp^y_n$. Let $p_n = p_n^x \vee p_n^y$ be the supremum of $p_n^x$ and $p_n^y$. By \cite{MR641217}[Part III, Ch. 2, Prop. 5], $p_n$ is a finite projection in $M$. Proposition \ref{P:continuity of logarithm map} applies to $(p_nMp_n, \tau \circ p_n)$. So $\|x_n-y_n\|_2 \le d_\cP(e^{x_n},e^{y_n})$.

By Proposition \ref{finite approx of cP}, $d_\cP(e^{x_n},e^{y_n}) \to d_\cP(e^x,e^y)$ as $n\to\infty$.  It remains to show that $\|x_n-y_n\|_2 \to \|x-y\|_2$ as $n\to\infty$.  Now we know that $x_n -y_n \to x-y$ in measure.  Furthermore we can write $\|x_n-y_n\|_2^2$, in a similar fashion as equation (\ref{E:formula for norm}), as $2\int_{0}^{\infty}\lambda\mu_{|x_n-y_n|}(\lambda,\infty)\dee\lambda$, and by Lemma \ref{cdfcont}, $\mu_{|x_n-y_n|}(\l,\infty) \to \mu_{|x-y|}(\l,\infty)$, while by Proposition \ref{P:subadditive estimate} and Proposition \ref{P:basic facts positive ops affil} $\mu_{|x_n-y_n|}(\l,\infty) \leq \mu_{|x|}(\l,\infty) + \mu_{|y|}(\l,\infty)$.  So by the Dominated Convergence Theorem $\|x_n-y_n\|_2 \to \|x-y\|_2$.

\end{proof}

\begin{lem}\label{L: approx converg}
Let $f:\R \to \R$ be a bounded continuous function. Suppose there is an open neighborhood $\cO \subset \R$ of $0$ such that $f(t)=0$ for all $t \in \cO$. Suppose $a_1,a_2,\ldots$ is a sequence in $\rL^2(M,\tau)_{sa}$ that converges to $a \in \rL^2(M,\tau)_{sa}$.  Then $d_\cP(e^{f(a_k)},e^{f(a)}) \to 0$ as $k\to\infty$.
\end{lem}

\begin{proof}[Proof of Lemma \ref{L: approx converg}]
By Lemma \ref{L:L2tomeasure},  $a_k$ converges to $a$ in measure as $k\to\infty$. It follows by Lemma \ref{tikhonov} that $f(a_k)$ converges to $f(a)$ in measure, and also $e^{f(a_k)}$ converges to $e^{f(a)}$ in measure.  Since multiplication is jointly continuous with respect to the convergence in measure topology, $z_k:= e^{-f(a_k)/2}e^{f(a)}e^{-f(a_k)/2} \to 1$ in measure.

Let $\l>0$ be such that $\sup_{t\in \R} |f(t)| \le \l$ and $f(t)=0$ for all $|t|\le 1/\l$.  Now since $ e^{-2\l} \leq e^{-f(a_k)/2}e^{f(a)}e^{-f(a_k)/2} \leq e^{2\l}$ and $\log$ is a continuous function on the spectrum of $z_k$, by Lemma \ref{tikhonov} $\log z_k \to 0$ in measure.


We now show that $\log z_k$ converges to $0$ in $\rL^2$.  Now $-2\l \leq \log z_k \leq 2\l$ is uniformly bounded by $2\l$, $\log z_k$ is also in $\rL^2_{sa}(M,\tau)$.
\begin{claim}\label{cofinite kernel}
$\sup_k \tau(1_{(0,\infty)}(|\log z_k|)) < \infty$.
\end{claim}

\begin{proof}
Note that $\ker(f(a_k)) \cap \ker(f(a)) \subset \ker(\log z_k)$, so $\ker(\log z_k)^\perp \leq (\ker(f(a_k)) \cap \ker(f(a)))^\perp = \overline{\textrm{span}(\ker(f(a_k))^\perp \cup \ker(f(a_k))^\perp)}$.  Equivalently, $1_{(0,\infty)}(|\log z_k|) \leq 1_{(0,\infty)}(|f(a_k)|) \vee 1_{(0,\infty)}(|f(a)|) \leq 1_{(0,\infty)}(|f(a_k)|) + 1_{(0,\infty)}(|f(a)|) $.   Now  $\tau(1_{(0,\infty)}|f(a_k)|) = \tau(1_{(1/2\l,\infty)}|a_k|) \leq 4\l^2\|a_k\|_2^2$.  Since $a_k$ is converging to $a$ in $\rL^2$, the right hand side is bounded independently of $k$.  The claim follows.
\end{proof}

Since $|\log z_k| \le 2\l$, the claim implies that for any $\eps>0$,
\begin{eqnarray*}
\|\log z_k\|_2^2 &\le& \eps^2\mu_{|\log z_k|}(0,\eps] + (2\l)^2\mu_{|\log z_k|}(\eps,\infty)\le \eps^2 K + (2\l)^2\mu_{|\log z_k|}(\eps,\infty)
\end{eqnarray*}
where $K=\sup_k \tau(1_{(0,\infty)}(|\log z_k|))$ is constant. Since $\eps>0$ is arbitrary and $|\log z_k| \to |\log z|$ in measure (as $k\to\infty)$, it follows that $\log z_k \to 0$ in $\rL^2$. This implies $d_\cP(e^{f(a_k)},e^{f(a)}) \to_{k\to\infty} 0$.

\end{proof}

\begin{prop}\label{C:exp cont semi-finite}
$\exp: \rL^2(M,\tau)_{sa} \to \cP$ is continuous.
\end{prop}
\begin{proof}
We use a strategy similar to that used in the finite trace setting.
Suppose $(a_k)$ is a sequence in $\rL^2(M,\tau)_{sa}$ converging to $a$ in $\rL^2$.

Let $\l >0$.  Let $f_\l: \R \to \R$ be a continuous nondecreasing function such that for $f_\l(x) = 0$ on $[-1/2\l, 1/2\l]$, $f_\l(x) = \l$ for $x > \l$, $f_\l(x) = -\l$ for $x < -\l$, and $f_\l(x) = x$ on $[-\l,-1/\l] \cup [1/\l, \l]$.

Because $a_k$ and $f_\l(a_k)$ commute, $d_\cP(e^{a_k},e^{f_\l(a_k)}) = \|a_k - f_\l(a_k)\|_2$. Similarly,  $d_\cP(e^{a},e^{f_\l(a)}) = \|a - f_\l(a)\|_2$. So two applications of the triangle inequality yield
\begin{align}\label{E:cutoff functions estimate2}
d_{\mathcal{P}}(e^{a_{k}},e^{a})&\leq d_{\mathcal{P}}(e^{a_k},e^{f_{\lambda}(a_k)})+d_{\mathcal{P}}(e^{f_{\lambda}(a_{k})},e^{f_{\lambda}(a)})+d_{\mathcal{P}}(e^{f_\l(a)},e^a)\\ \nonumber
&=\|a_{k}-f_{\lambda}(a_{k})\|_{2}+d_{\mathcal{P}}(e^{f_{\lambda}(a_{k})},e^{f_{\lambda}(a)})+ \|a-f_{\lambda}(a)\|_{2}\\ \nonumber
&\leq d_{\mathcal{P}}(e^{f_{\lambda}(a_{k})},e^{f_{\lambda}(a)})+ \|a_k-a\|_{2} +2\|a-f_{\lambda}(a)\|_{2} + \|f_\l(a)-f_{\lambda}(a_{k})\|_{2}.
\end{align}
By Lemmas \ref{semi-finite cont of log} and \ref{L: approx converg},
\begin{align*}
\limsup_{k\to\infty} d_{\mathcal{P}}(e^{a_{k}},e^{a})&\leq  2\|a-f_{\lambda}(a)\|_{2}
\end{align*}
Since
$$\|a - f_\l(a)\|^2_2 = \int (t-f_\l(t))^2~\dee \mu_a(t),$$
and $(t-f_\l(t))^2 \le t^2$, the Dominated Convergence Theorem implies $\|a - f_\l(a)\|_2 \to 0$ as $\l \to \infty$.  Combined with the previous inequality, this implies $d_\cP(e^{a_k},e^a) \to 0$ as $k\to\infty$.
\end{proof}

The proof that $(\cP,d_\cP)$ is complete now follows from the same argument as in Corollary \ref{C:complete}. This finishes the proof of Proposition \ref{P:semi-finite CAT(0)}(\ref{I: semi-finite completeness}).



\begin{proof}[Proof of Proposition \ref{P:semi-finite CAT(0)}(\ref{I:semi-finite cat(0)})]
We use arguments similar to those found in \cite[Theorem II.3.9]{bridson-haefliger-book} in order to apply \cite[Proposition II.1.11]{bridson-haefliger-book}. A {\bf sub-embedding} of a 4-tuple $(x_1,y_1,x_2,y_2)$ of points in a metric space $(X,d_X)$ is a 4-tuple of points $(\bx_1,\by_1,\bx_2,\by_2)$ in the Euclidean plane $\E^2$ such that $d_X(x_i,y_j)=\|\bx_i- \by_j\|$ ($\forall i,j \in \{1,2\})$ and $d_X(x_1,x_2) \le \|\bx_1-\bx_2\|$ and $d_X(y_1,y_2) \le \|\by_1-\by_2\|$. A pair of points $x,y \in X$ is said to have {\bf approximate midpoints} if for every $\d>0$ there exists $m \in X$ such that
$$\max\{ d_X(x,m), d_X(m,y)\} \le \frac{1}{2} d_X(x,y) + \d.$$
According to \cite[Proposition II.1.11]{bridson-haefliger-book}, a metric space $(X,d_X)$ is CAT(0) if and only if every 4-tuple of points in $X$ admits a sub-embedding into $\E^2$ and every pair of points $x,y \in X$ admits approximate midpoints. We will verify that $(\cP,d_\cP)$ satisfies this condition.

Now let $a_i \in \cP$, $1 \leq i \leq 4$ and consider the reduced versions $a_{i,n} = \exp(p_n \log(a_i) p_n) \in \cP_{p_n}$, where $p_n = \vee_{i=1}^4 p_n^{\log(a_i)}$ is as defined in Notation \ref{notation1}.   For each $n$, $(\cP_{p_n}, d_{\cP})$ is CAT(0) by Corollary \ref{C:cat0}. So each 4-tuple $(a_{i,n})_i$ has a sub-embedding in Euclidean space $\E^2$: a 4-tuple of points $(\ba_{i,n})_i$ such that $d_{\cP_n}(a_{1,n}, a_{2,n}) \leq \|\ba_{1,n}-\ba_{2,n}\|$, $d_{\cP_n}(a_{3,n}, a_{4,n}) \leq \|\ba_{3,n}-\ba_{4,n}\|$, $d_{\cP_n}(a_{i,n}, a_{j,n}) = \|\ba_{i,n}-\ba_{j,n}\|$ for all other $i,j$.

By translation invariance of the standard metric on $\E^2$, we can assume $\ba_{1,n} = \ba_1$ is the same for all $n$.  We have shown in Proposition \ref{finite approx of cP} that for each $i,j$, $d_{\cP_{p_n}}(a_{i,n},a_{j,n}) \to d_\cP(a_i,a_j)$ as $n\to\infty$.  So the sub-embedding condition and triangle inequality show that $(\ba_{i,n})$ is contained in a compact set as $i$ and $n$ vary.  In particular, by passing to a subsequence if necessary, we can assume $\ba_{i,n}$ converges to some $\ba_i$ in $\E^2$ for each $i$.  It follows that $(\ba_i)_{i=1}^4$ is a sub-embedding of $(a_i)_{i=1}^4$.


For the approximate midpoint condition, let $x,y \in \cP$.  Let $x_n, y_n \in \cP_{p_n}$ be the reduced versions where $p_n$ is now redefined to be $p_n=p_n^{\log(x)} \vee p_n^{\log(y)}$. For example, $x_n = \exp(p_n \log(x)p_n)$. Let $\d>0$.

By Proposition \ref{finite approx of cP}, there exists $n$ such that
$$|d_{\cP}(x_{n},y_{n}) -d_\cP(x,y)|<\d/3, ~d_\cP(x,x_n) \le \d/3, ~d_\cP(y,y_n) \le \d/3.$$
Because $(\cP_{p_n},d_\cP)$ is CAT(0), there exists $m \in \cP_{p_n}$ with
$$\max\{ d_\cP(x_n,m), d_\cP(m,y_n)\} \le \frac{1}{2} d_\cP(x_n,y_n) + \d/3.$$
By the triangle inequality,
\begin{eqnarray*}
\max\{ d_\cP(x,m), d_\cP(m,y)\} &\le& \max\{ d_\cP(x_n,m), d_\cP(m,y_n)\} + \d/3 \\
&\le& \frac{1}{2} d_\cP(x_n,y_n) + 2\d/3 \le \frac{1}{2} d_\cP(x,y) + \d.
\end{eqnarray*}
Thus $x,y$ have approximate midpoints. This completes the verification of the conditions in \cite[Proposition II.1.11]{bridson-haefliger-book}.

\end{proof}

\begin{proof}[Proof of Proposition \ref{P:semi-finite CAT(0)}(\ref{I: P infinity dense})]
That $\cP^\infty$ is dense in $\cP$ follows from $M_{sa} \cap \rL^2(M,\tau)_{sa}$ being dense in $\rL^2(M,\tau)_{sa}$ and Proposition \ref{C:exp cont semi-finite} (that $\exp: \rL^2(M,\tau) \to \cP$ is continuous).  Now if $x \in \cP^{\infty}$, then $x = e^y$ for $y \in M_{sa}$.  Then $\|e^y\|_\infty \leq e^{\|y\|_\infty}$ and similarly for $x^{-1} = e^{-y}$, so $x \in M^\times$.  Thus $\cP^\infty \subset \cP \cap M^\times$.  Conversely if $x \in \cP \cap M^\times$ then $\log x \in \rL^2_{sa}(M,\tau) \cap M = M_{sa}$, so $x \in \cP^\infty$.
\end{proof}

\begin{cor}\label{semi-finite weaker top and geodesics}
Corollaries \ref{C:weaker topology} and \ref{C:geodesics} also hold for $(M,\tau)$ semi-finite.
\end{cor}
\begin{proof}
Using Proposition \ref{P:semi-finite CAT(0)}, the proofs are similar to the finite case.
\end{proof}

\section{Proofs of the main results}\label{sec:proofs of main results}

\subsection{The limit operator}
This subsection proves a generalization of Theorem \ref{thm:main}. We first need a lemma.
\begin{lem}\label{lem:contraction}
Let $(M,\tau)$ be a semi-finite tracial von Neumann algebra. For any $a,b\in \cP$ and $\s \ge 1$,
$$d_\cP(a^\s, b^\s) \ge \s d_\cP(a,b).$$
\end{lem}

\begin{proof} First, assume $\tau$ is a finite trace.  Let $x=\log a$, $y=\log b$. Recall that $M_{sa} \subset M$ is the set of self-adjoint elements in $M$. Because $M_{sa}$ is dense in $\rL^2(M,\tau)_{sa}$, there exist $x_n, y_n \in M_{sa}$ with $x_n \to x$ and $y_n \to y$ in $\rL^2(M,\tau)_{sa}$ as $n\to\infty$. Thus
\begin{eqnarray*}
 d_\cP(a^\s, b^\s) &=& d_\cP(e^{\s x}, e^{\s y}) = \lim_{n\to\infty} d_\cP (e^{\s x_n}, e^{\s y_n}) \\
 &\ge & \lim_{n\to\infty} \s d_\cP (e^{ x_n}, e^{ y_n}) = \s d_\cP(e^x,e^y) = \s d_\cP(a,b)
 \end{eqnarray*}
 where the second and third equalities follow from continuity of the exponential map (Theorem \ref{T:continuity of exp}) and the inequality follows from Corollary \ref{C:contraction}.

 Next we consider the general semi-finite case.  Let $x_n,y_n$ be the reduced versions of $x,y \in \rL^2_{sa}(M,\tau)$ as in Notation \ref{notation1}.  Then
\begin{eqnarray*}
d_\cP(a^\s, b^\s) &=& d_\cP(e^{\s x}, e^{\s y}) = \lim_{n\to\infty} d_\cP (e^{\s x_n}, e^{\s y_n}) \\
 &\ge & \lim_{n\to\infty} \s d_\cP (e^{ x_n}, e^{ y_n}) = \s d_\cP(e^x,e^y) = \s d_\cP(a,b).
 \end{eqnarray*}

Where the second and second-to-last equalities follows from Proposition \ref{finite approx of cP} and the inequality follows from the above finite case.
\end{proof}

We can now prove a slight generalization of Theorem \ref{thm:main} by expanding the range of the cocycle.
\begin{thm}\label{thm:main-general}
Let $(X,\mu)$ be a standard probability space, $f:X \to X$ an ergodic measure-preserving transformation, $(M,\tau)$ a semi-finite von Neumann algebra with faithful normal trace $\tau$. Let $c:\Z \times X \to \GL^2(M,\tau)$ be a cocycle:
$$c(n+m,x) = c(n, f^mx)c(m,x)\quad \forall n,m \in \Z, ~\mu-a.e. ~x\in X.$$
Let $\pi:\GL^2(M,\tau) \to \Isom(\cP)$ be the map $\pi(g)x = gxg^*$ where $\Isom(\cP)$ is the group of isometries of $\cP$. Suppose $\pi \circ c$ is measurable with respect to the compact-open topology on $\Isom(\cP)$ and
$$ \int_X \| \log(|c(1,x)|^2)\|_2 ~d\mu(x) = \int_X d_\cP(1,|c(1,x)|^2)~d\mu(x)< \infty.$$
Then for almost every $x\in X$, the following limit exists:
$$\lim_{n\to\infty} \frac{\| \log (| c(n,x)|^2 )\|_2}{n} = D.$$
Moreover, if $D>0$ then for a.e. $x$, there exists $\La(x) \in \rL^2(M,\tau)$ with $\La(x)\ge 0$ such that
$$\log \La(x):=\lim_{n\to\infty} \log \left([c(n,x)^*c(n,x)]^{1/2n}\right) \in \rL^2(M,\tau)$$
exists for a.e. $x$ and
$$\lim_{n\to\infty} \frac{1}{n}d_\cP(\La(x)^n, |c(n,x)|)=0.$$
\end{thm}

\begin{proof}
We will use Theorem \ref{thm:km}. So let $(Y,d)=(\cP,d_\cP)$. By Corollaries \ref {C:complete} and \ref{C:cat0} for the finite case and Proposition \ref{P:semi-finite CAT(0)} for the semi-finite case, $(\cP,d_\cP)$ is a complete CAT(0) metric space. Let $y_0 = \id \in Y$. Observe that the map
$$\N \times X \to \GL^2(M,\tau), \quad (n,x) \mapsto c(n,x)^*$$
is a reverse cocycle. Also
$$d_{\cP}(y_0, c(1,x)^*.y_0) = \| \log(c(1,x)^* c(1,x))\|_2.$$
So
$$\int_X d_{\cP}(y_0, c(1,x)^*.y_0) ~d\mu(x) < \infty.$$
Theorem \ref{thm:km} implies:  for almost every $x\in X$, the following limit exists:
$$\lim_{n\to\infty} \frac{d_{\cP}(y_0, c(n,x)^*.y_0)}{n}  =D.$$
Moreover, if $D>0$ then for almost every $x$ there exists a unique unit-speed geodesic ray $\g(\cdot, x)$ in $\cP$ starting at $\id$ such that
$$\lim_{n\to\infty} \frac{1}{n} d_{\cP}(\gamma(Dn,x), c(n,x)^*.y_0) =0.$$
By Corollaries  \ref {C:geodesics} for the finite case and \ref{P:semi-finite CAT(0)} for the semi-finite case,
$$\g(t,x) = \exp(t\xi(x))$$
for some unique unit norm element $\xi(x) \in \rL^2(M,\tau)_{sa}$. Let $\La(x)=\exp(D\xi(x)/2)$. Thus we have
$$\lim_{n\to\infty} \frac{1}{n} d_{\cP}(\La(x)^{2n}, c(n,x)^*c(n,x)) =0.$$
Equivalently,
$$\lim_{n\to\infty} \frac{1}{n} \|\log(\La(x)^{-n} c(n,x)^*c(n,x) \La(x)^{-n})\|_2 = \lim_{n\to\infty} \frac{1}{n} L(\La(x)^{-n} c(n,x)^*c(n,x) \La(x)^{-n}) = 0. $$
Observe that
\begin{eqnarray*}
\lim_{n\to\infty} \left\| \log \La(x) - \log\left([c(n,x)^*c(n,x)]^{1/2n}\right) \right\|_2 &\le & \lim_{n\to\infty} d_\cP(\La(x),  [c(n,x)^*c(n,x)]^{1/2n} ) \\
&\le & \lim_{n\to\infty} \frac{1}{n} d_{\cP}(\La(x)^{2n}, c(n,x)^*c(n,x)) =0
\end{eqnarray*}
where the first inequality follows from Proposition \ref{P:continuity of logarithm map} for the finite case and Lemma \ref{semi-finite cont of log} for the semi-finite case. The second inequality follows from Lemma \ref{lem:contraction}.
This concludes the proof.

\end{proof}

In order to show that Theorem \ref{thm:main-general} implies Theorem \ref{thm:main}, we need to show how SOT-measurability of the cocycle $c$ in Theorem \ref{thm:main} implies that $\pi\circ c$ is measurable with respect to the compact-open topology.

We will need the next few lemmas to clarify the measurability hypothesis on the cocycle. The next lemma is probably well-known.

\begin{lem}\label{lem:pointwise}
Let $(Y,d)$ be a metric space. Then the pointwise convergence topology on the isometry group $\Isom(Y,d)$ is the same as the compact-open topology.
\end{lem}

\begin{proof}
It is immediate that the pointwise convergence topology is contained in the compact-open topology. To show the opposite inclusion, let $K \subset Y$ be compact, $O \subset Y$ be open and suppose $g \in \Isom(Y,d)$ is such that $gK \subset O$. Let $g_n \in \Isom(Y,d)$ and suppose $g_n \to g$ pointwise. It suffices to show that $g_n K \subset O$ for all sufficiently large $n$.

Because $K$ is compact, there are a finite subset $F \subset K$ and for every $x \in F$, a radius $\eps_x>0$ such that if $B(x,\eps_x) \subset Y$ is the open ball of radius $\eps_x$ centered at $x$ then
$$gK \subset \cup_{x \in F} B(gx,\eps_x) \subset O.$$
By compactness again, there exist $0<\eps'_x<\eps_x$ such that
$$K \subset \cup_{x \in F} B(x,\eps'_x).$$
Since $g_n \to g$ pointwise, there exists $N$ such that $n>N$ implies $d(g_n x, gx) \le \eps_x-\eps'_x$ for all $x \in F$. Therefore,
$$g_n K \subset \cup_{x \in F} B(g_n x,\eps'_x) \subset \cup_{x \in F} B(g x,\eps_x) \subset O$$
as required.
\end{proof}

Let $\tcH \in \{\cH, \rL^2(M,\tau)\}$ be one of the two Hilbert spaces under consideration. We use $\tcH-SOT$ to denote the Strong Operator Topology with respect to the embedding of $M$ in $B(\tcH)$. Similarly, $\tcH-WOT$ refers to the Weak Operator Topology. If we write SOT or WOT without the $\tcH$-prefix then the default assumption is that we have  chosen $\tcH=\cH$. Of course, it is possible that $\cH=\rL^2(M,\tau)$, so SOT by itself refers to both cases.

Some of the results of the next Theorem appear in \cite{MR1123654}.
\begin{thm}\label{thm:SOT}
Suppose $(M,\tau)$ is $\s$-finite, semi-finite and $\cH$ is separable.  Then
\begin{enumerate}
\item the operator norm $M \to \R$, $T \mapsto \|T\|_\infty$ is SOT-Borel; \label{I:op-norm}
\item a subset $E \subset M$ is $\cH$-SOT-Borel if and only if it is $\rL^2(M,\tau)$-SOT-Borel; \label{I:embedding}
\item  the inverse operator norm $M^\times \to \R$, $T \mapsto \|T^{-1}\|_\infty$ is SOT-Borel; \label{I:inverse op-norm}
\item a subset $E \subset M$ is SOT-Borel if and only if it is WOT-Borel; \label{I:WOT}
\item the adjoint $M\to M$, $T \mapsto T^*$ is SOT-Borel; \label{I:adjoint}
\item the multiplication map $M \times M \to M$, $(S,T) \mapsto ST$ is SOT-Borel; \label{I:multiplication}
\item the map $\cP \cap M^\times \to M$ defined $T \mapsto \log T$ is SOT-Borel; \label{I:logarithm}
\item the map $M \cap \rL^2(M,\tau)$, $T \mapsto \|T\|_2$ is SOT-Borel; \label{I:l2 norm}
\item for any $x,y \in \cP \cap M^\times$ the map $M^\times \cap \GL^2(M,\tau) \to \R$ defined by $T \mapsto d_\cP( TxT^*, y)$ is SOT-Borel. \label{I:distance}
\end{enumerate}
\end{thm}

\begin{proof}

(\ref{I:op-norm}), (\ref{I:WOT}), (\ref{I:adjoint}), (\ref{I:multiplication})  follow from \cite[Proposition 52.2(c) and Proposition 52.5]{ConwayOT}.


(\ref{I:embedding}) Let $M_C = \{T \in M:~ \|T\|_\infty \le C\}.$ By  \cite[I.4.3. Theorem 2]{MR641217} or \cite[Corollary 2.5.9 and Proposition 2.5.8]{anantharaman-popa}, the $\cH$-SOT-topology on $M_C$ is the same as the $\rL^{2}(M,\tau)$-SOT-topology on $M_C$. Since the operator norm is SOT-Borel by (\ref{I:op-norm}), this implies that the $\cH$-SOT-Borel sigma-algebra is the same as the $\rL^2(M,\tau)$-SOT-Borel sigma-algebra.

(\ref{I:inverse op-norm})
Let $\cH_{0}$ be a countable dense subset of $\{h\in \cH:\|h\|=1\}.$  Then for every invertible $T,$ we have
\[\|T^{-1}\|_{\infty}^{-1}=\inf_{h\in \cH_{0}}\|T(h)\|.\]
Since $T\mapsto \|T(h)\|$ is SOT-continuous for every $h\in \cH,$ this proves that $T\mapsto \|T^{-1}\|_{\infty}$ is Borel.

(\ref{I:logarithm}) Note that $\cP \cap M$ is SOT-Borel since, by (\ref{I:adjoint}), $M_{sa}$ is SOT-Borel and
$$\cP\cap M=\{a \in M_{sa}:~ \langle ah, h \rangle \ge 0 ~\forall h \in \cH_0\}.$$
For $D>0$, let
$$M_D = \{T \in \cP:~ \|T\|_\infty \le D \textrm{ and } \|T^{-1}\|_\infty \le D\}.$$
By items (\ref{I:op-norm}) and (\ref{I:inverse op-norm}), $M_D$ is SOT-Borel. So it suffices to show that the map $M_D \to M$ given by $T \mapsto \log T$ is SOT-continuous. Fortunately, it was proven in \cite[Corollary on page 232]{MR50181} that if $h: \R \to \R$ is any continuous bounded function then the map on self-adjoint operators given by $a \mapsto f(a)$ is strongly continuous. Since $\log$ is bounded on $[D^{-1},D]$, this implies $T \mapsto \log T$ is SOT-continuous on $M_D$.

(\ref{I:l2 norm}) Let $p_1,p_2,\ldots$ be a sequence of pairwise-orthogonal finite projections $p_i \in M$ with $\lim_{n\to\infty} \sum_{i=1}^n p_n = \id$ in $\rL^2(M,\tau)$-SOT. Then for any $T \in \rL^2(M,\tau)$,  $\sum_{i=1}^n p_iT$ converges to $T$ in $\rL^2(M,\tau)$.

Now $\langle p_i T, p_j T\rangle = 0$ for all $i \ne j$ and $\sum_{i=1}^\infty p_i T = T$ (where convergence is in $\rL^2(M,\tau)$). Therefore,
$$\|T\|_2^2 = \sum_{i=1}^\infty \|p_i T\|_2^2 =  \sum_{i=1}^\infty \|T^* p_i \|_2^2$$
where the last equality follows from the tracial property of $\tau$.  So it suffices to prove that for any fixed finite projection $p \in M$, the map $T \mapsto \|T^*p\|_2^2$  is SOT-Borel. This follows from item (\ref{I:adjoint})  which states that the adjoint map is SOT-Borel (since $p \in \rL^2(M,\tau)$).

(\ref{I:distance}) By definition,
$$ d_\cP( TxT^*, y) = \| \log(y^{-1/2}TxT^* y^{-1/2})\|_2.$$
So this item follows from the previous items.
\end{proof}

\begin{cor}\label{C: SOT-PCT Borel}
Suppose $(M,\tau)$ is $\s$-finite and semi-finite. Then $\pi:M^\times \to \Isom(\cP)$ is Borel as a map from $M^\times$ with the SOT to $\Isom(\cP)$ with the pointwise convergence topology.
\end{cor}

\begin{proof}
By definition of the pointwise convergence topology, it suffices to show  that for every $x,y \in \cP$, the map $T \mapsto d_\cP( TxT^*, y)$ is SOT-Borel. By (8) of Proposition \ref{P:semi-finite CAT(0)}, $M^\times \cap \cP$ is dense in $\cP$. So it suffices to show that for every $x,y \in M^\times \cap \cP$, the map $T \mapsto d_\cP( TxT^*, y)$ is SOT-Borel. This is item (\ref{I:distance}) of Theorem \ref{thm:SOT}.
\end{proof}

\begin{cor}\label{C:hypotheses}
The hypotheses of Theorem \ref{thm:main} imply the hypotheses of Theorem \ref{thm:main-general}. In particular, Theorem \ref{thm:main} is true.
\end{cor}

\subsection{Determinants}\label{sec:determinants}

This section proves Theorem \ref{thm:determinants}. So we assume $(M,\tau)$ is finite. Following \cite[Definition 2.1]{MR2339369}, we let $M^\Delta$ be the set of all $x\in \rL^0(M,\tau)$ such that
$$\int_0^\infty \log^+ (t)~d\mu_{|x|}(t)<\infty.$$
For $x \in M^\Delta$, the  integral $\int_0^\infty \log t~d\mu_{|x|}(t) \in [-\infty, \infty)$ is well-defined. The {\bf Fuglede-Kadison determinant of $x$} is
$$\Delta(x):=\exp \left(\int_0^\infty \log t~d\mu_{|x|}(t)\right) \in [0,\infty).$$
For the sake of context, we mention that by \cite[Lemma 2.3 and Proposition 2.5]{MR2339369}, if $x,y \in M^\Delta$ and $\Delta(x)>0$ then $x^{-1} \in M^\Delta$, $\Delta(x^{-1})=\Delta(x)^{-1}$, $xy \in M^\Delta$, and $\Delta(xy)=\Delta(x)\Delta(y)$.



\begin{prop}\label{P:determinant}
Suppose $\tau$ is a finite trace. Then $\GL^2(M,\tau) \subset M^\Delta$. Moreover, $\Delta:\cP \to (0,\infty)$ is continuous.
\end{prop}

\begin{proof}
Let $x \in \GL^2(M,\tau)$. By definition, $\log |x| \in \rL^2(M,\tau)$. Since $\tau$ is a finite trace, $\rL^2(M,\tau) \subset \rL^1(M,\tau)$. Thus $\log |x| \in  \rL^1(M,\tau)$ and therefore $\log^+ |x| \in \rL^1(M,\tau)$. So $x \in M^\Delta$.

Now let $(x_n)_n \subset \cP$ and suppose $\lim_{n\to\infty} x_n = x \in \cP$. By Proposition \ref{P:continuity of logarithm map}, $\log |x_n|$ converges to $\log |x|$ in $\rL^2(M,\tau)$. Therefore, $\log |x_n|$ converges to $\log |x|$ in $\rL^1(M,\tau)$. But the trace $\tau:\rL^1(M,\tau)\to \C$ is norm-continuous. So $\tau(\log|x_n|) \to \tau(\log|x|)$. Since $\exp:\R \to \R$ is continuous  and $\Delta(x)=\exp(\tau(\log |x|))$, this finishes the proof.
\end{proof}

Theorem \ref{thm:determinants} follows immediately from Proposition \ref{P:determinant} and Theorem \ref{thm:main}.

\subsection{Growth rates}\label{sec: growth rates}


In this subsection, we prove Theorem \ref{thm:main2}. The proof uses Theorem \ref{thm:main} as a black-box. The extra ingredients needed to prove the theorem are general approximation results for powers of a single operator. These results will also be needed in later subsections to prove Theorem \ref{thm:invariance}.

\begin{defn}
Let $a \in \rL^0(M,\tau)$ be a positive operator and $\xi \in \dom(a)\subset \rL^2(M,\tau)$. By the Spectral Theorem  there exists a unique positive measure $\nu$ on $\C$ such that $\nu([0,\infty))=\|\xi\|_{2}^{2}$ and for every bounded, Borel function $f\colon [0,\infty)\to \C$,
\[\ip{f(a)\xi,\xi}=\int f(s)\,\dee\nu(s).\]
Moreover, for a Borel function $f\colon [0,\infty)\to\C$ we have that $\xi\in \dom(f(a))$ if and only if $\int |f(s)|^{2}\,\dee\nu(s)<\infty,$ and $\ip{f(a)\xi,\xi}=\int f\,\dee\nu$ if $\xi\in \dom(f(a)).$ The measure $\nu$ is called the {\bf spectral measure of $a$ with respect to $\xi$}. Let $\rho(\nu) \in [0,\infty]$ be the smallest number such that $\nu$ is supported on the interval $[0,\rho(\nu)]$. \end{defn}

\begin{lem}\label{L:existence of growth rates limiting object}
Let $a \in \rL^0(M,\tau)_{sa}$, $\xi\in \bigcap_{n=1}^{\infty}\dom(a^n)$ and let $\nu$ be the spectral measure of $a$ with respect to $\xi$. Then
$$\rho(\nu)=\lim_{n\to\infty} \|a^n \xi\|_2^{1/n} \in [0,\infty].$$
Moreover, $\xi \in 1_{[0,t]}(a)(\rL^2(M,\tau))$ if and only if $\lim_{n\to\infty} \|a^n \xi\|_2^{1/n} \le t$ (for any $t \in [0,\infty]$).
\end{lem}

\begin{proof}
Since $\xi \in \dom(a^{n})$, $\int s^{2n}\,\dee\nu(s)<\infty$ for every $n\in \N.$ Thus
\begin{eqnarray}\label{E:thingy}
\| a^{n}\xi\|_{2}^{1/n}=\ip{a^{2n}\xi,\xi}^{1/2n}=\left(\int s^{2n}\,\dee\nu(s)\right)^{1/2n}.
\end{eqnarray}
It is a standard measure theory exercise that the limit of $\left(\int s^{2n}\,\dee\nu(s)\right)^{1/2n}$ as $n\to\infty$ exists and equals $\rho(\nu)$.

Now suppose that $t>0$ and that $\lim_{n\to\infty}\|a^{n}\xi\|_{2}^{1/n}\leq t.$ Then, by the above comment we have that $\nu$ is supported on $[0,t].$ Thus:
\[\|\xi-1_{[0,t]}(a)\xi\|_{2}^{2}=\int |1-1_{[0,t]}(s)|^{2}\,\dee\nu(s)=0.\]
So $\xi\in 1_{[0,t]}(a)(\rL^2(M,\tau)).$

For the converse, suppose $\xi\in 1_{[0,t]}(a)(\rL^2(M,\tau)).$ Then
\begin{eqnarray*}
\lim_{n\to\infty}\|a^{n}\xi\|_{2}^{1/n} &=& \lim_{n\to\infty}\|a^{n}1_{[0,t]}(a)\xi\|_{2}^{1/n} =\lim_{n\to\infty}\left(\int_0^t s^{2n}\,\dee\nu(s)\right)^{1/2n} \le t.
\end{eqnarray*}

\end{proof}

\begin{lem}\label{semi-finite 2top}
Suppose $(M,\tau)$ is semi-finite.  Let $C>0$ and suppose $x_n, x \in M_C$, where $M_C = \{x \in M \cap \rL^2(M,\tau):~ \|x\|_\infty \leq C\}$ and $x_n \to x$ in measure as $n\to\infty$.  Then $x_n \to x$ in SOT.
\end{lem}
\begin{proof}
As in Proposition \ref{P:3topologies} we can assume that $x = 0$.  To begin, let $\xi \in M \cap \rL^2(M,\tau)$.  We want to show that $x_n\xi \to 0$ in $\rL^2(M,\tau)$.


Let $p_{n,k} = 1_{(1/k,\infty)}(|x_n|)$ and $x_{n,k} =  x_n p_{n,k}$. By the triangle inequality,
\begin{eqnarray*}
\|x_n\xi\|_2 &\leq& \|x_{n,k}\xi\|_2 + \|(x_n-x_{n,k})\xi\|_2 \\
&\leq& \|x_{n,k}\|_2 \|\xi\|_\infty + \|(x_n-x_{n,k})\|_\infty \|\xi\|_2 \\
&\leq& C \mu_{x_{n,k}}(0,\infty)^{1/2}  \|\xi\|_\infty + (1/k) \|\xi\|_2.
\end{eqnarray*}
Since $x_n \to_{n\to\infty} 0$ in measure, $\mu_{x_{n,k}}(0,\infty)^{1/2}= \mu_{x_{n,k}}(1/k,\infty)^{1/2} \to_{n\to\infty} 0$ for any fixed $k$. Therefore,
$$\limsup_{n\to\infty} \|x_n\xi\|_2 \leq (1/k) \|\xi\|_2.$$
Since $k$ is arbitrary, this proves $x_n\xi \to_{n\to\infty} 0$ in $\rL^2(M,\tau)$.

Now suppose $\xi \in \rL^2(M,\tau)$ and let $\eps>0$. Because $M \cap \rL^2(M,\tau)$ is dense in $\rL^2(M,\tau)$, there exists $\xi' \in M \cap \rL^2(M,\tau)$ with $\|\xi - \xi'\|_2 \le \eps$. So
\begin{eqnarray*}
\limsup_{n\to\infty} \|x_n\xi\|_2 &\leq& \limsup_{n\to\infty} \|x_n\xi'\|_2 + \|x_n(\xi-\xi')\|_2 \\
&\leq& \limsup_{n\to\infty} \|x_n(\xi-\xi')\|_2 \leq \limsup_{n\to\infty} \|x_n\|_\infty \|\xi-\xi'\|_2 \leq C \eps.
\end{eqnarray*}
Since $\eps>0$ is arbitrary, this proves $x_n\xi \to_{n\to\infty} 0$ in $\rL^2(M,\tau)$. Since $\xi$ is arbitrary, this proves $x_n \to 0$ in SOT.

\end{proof}

\begin{lem}\label{L:weak convergence of spectral measures}
For $n \in \N$, let $a, b_n \in \rL^0(M,\tau)_{sa}$ and $\xi, \xi_n \in \rL^2(M,\tau)$ with $\xi \in \dom(a)$, $\xi_n \in \dom(b_n)$. Assume:
\begin{itemize}
\item $b_n \to a$ in measure,
\item $\|\xi_n - \xi\|_2 \to 0$ as $n\to \infty$.
\end{itemize}
Let $\nu$ be the spectral measure of $a$ with respect to $\xi$ and let $\nu_n$ be the spectral measure of $b_n$ with respect to $\xi_n$. Then $\nu_n \to \nu$ weakly as $n\to\infty$.


\end{lem}

\begin{proof}

Let $f\in C(\R)$ be bounded.  By Lemma \ref{tikhonov}, $f(b_n) \to f(a)$ in measure.
Since $\|f(b_n)\|_{\infty},\|f(a)\|_{\infty}\leq \|f\|_\infty = \sup_{x\in \R}|f(x)|<\infty,$ Lemma \ref{semi-finite 2top} implies that $f(b_n)\to f(a)$ in the Strong Operator Topology.
Hence,
\begin{eqnarray*}
&&\left| \int f(s)\,\dee\nu_{n}(s) -\int f(s)\,\dee\nu(s)\right| = \left| \ip{f(b_n)\xi_n,\xi_n} - \ip{f (a)\xi,\xi} \right| \\
&\le&  \left| \ip{f(b_n)\xi_n,\xi_n} - \ip{f (b_n)\xi_n,\xi} \right| + \left| \ip{f(b_n)\xi_n,\xi} - \ip{f (b_n)\xi,\xi} \right| +  \left| \ip{f (b_n)\xi,\xi} - \ip{f (a)\xi,\xi} \right| \\
&\le&  \|f\|_\infty \|\xi_n - \xi\|_2 (\|\xi_n\|_2 + \|\xi\|_2)  +  \left| \ip{f (b_n)\xi,\xi} - \ip{f (a)\xi,\xi} \right|.
\end{eqnarray*}
Since $\|\xi_n - \xi\|_2 \to 0$ by assumption and $f(b_n)\to f(a)$ in the SOT, this shows $\int f(s)\,\dee\nu_{n}(s) \to \int f(s)\,\dee\nu(s)$ as $n\to\infty$. Since $f$ is arbitrary, $\nu_{n}\to \nu$ weakly.
\end{proof}

\begin{defn}
If $a \in \rL^0(M,\tau)$ and $\xi \in \rL^2(M,\tau) \setminus \dom(a)$, then let $\|a \xi\|_2 = +\infty$.
\end{defn}
The next definition generalizes Definition \ref{D:growth1}.
\begin{defn}
Given $\xi\in \rL^2(M,\tau)$, let $\Si(\xi)$ be the set of all sequences $(\xi_n)_n \subset \rL^2(M,\tau)$ such $\lim_{n\to\infty} \|\xi - \xi_n\|_2 = 0$. Given a sequence $\bfc=(c_n)_n \subset \rL^0(M,\tau)$ and $\xi\in \rL^2(M,\tau)$, define the {\bf upper and lower smooth growth rates} of $\bfc$ with respect to $\xi$ by
\begin{eqnarray*}
\underline{\Gr}( \bfc | \xi) &=& \inf\big\{ \liminf_{n\to\infty} \| c_n \xi_n\|_2^{1/n} :~ (\xi_n)_n \in \Si(\xi),\ \xi_n \in \dom(c_n)  \big\}  \\
\overline{\Gr}( \bfc | \xi) &=& \inf\big\{ \limsup_{n\to\infty} \| c_n \xi_n\|_2^{1/n} :~ (\xi_n)_n \in \Si(\xi),\ \xi_n \in \dom(c_n) \big\}.
\end{eqnarray*}

\end{defn}

\begin{lem}\label{L:inclusions of candidate Osoledets spaces}
For $n \in \N$, let $c_n \in \rL^0(M,\tau)$, $a\in \rL^0(M,\tau)$ with $a\ge 0$ and $\xi \in \rL^2(M,\tau)$. Let $\bfc = (c_n)_n$. Assume:
\begin{itemize}
\item $|c_n|^{1/n}\to a$ in measure as $n\to\infty$.
\item $\xi\in \bigcap_{n=1}^{\infty}\dom(a^n)$.
\end{itemize}
Then
$$\underline{\Gr}(  \bfc | \xi)  = \lim_{n\to\infty} \|a^n\xi\|_2^{1/n}= \overline{\Gr}(  \bfc | \xi).$$
\end{lem}

\begin{proof}
Let $\nu$ be the spectral measure of $a$ with respect to $\xi$. Let $\rho(\nu) \ge 0$ be the smallest number such that $\nu$ is supported on $[0,\rho(\nu)]$. By Lemma \ref{L:existence of growth rates limiting object}, $\rho(\nu)=\lim_{n\to\infty} \|a^n\xi\|_2^{1/n}$.

It is immediate that $\underline{\Gr}( \bfc | \xi) \le \overline{\Gr}( \bfc | \xi).$ So it suffices to show $\rho(\nu)\le \underline{\Gr}( \bfc | \xi)$ and $\overline{\Gr}( \bfc | \xi) \le \rho(\nu)$.

We first show $\rho(\nu)\le \underline{\Gr}( \bfc | \xi)$. Let $b_n=|c_n|^{1/n}$. Let $(\xi_n)_n \in \Si(\xi)$ with $\xi_n \in \dom(b_n)$. By hypothesis, $b_n \to a$ in measure. Let $\nu_n$ be the spectral measure of $\xi_n$ with respect to $b_n$. By Lemma \ref{L:weak convergence of spectral measures}, $\nu_n \to \nu$ weakly. Along with (\ref{E:thingy}) and Fatou's Lemma we have for every $m\in \N:$
\begin{align*}
\|a^{m}\xi\|_{2}^{1/m} &= \left(\int s^{2m}\,\dee\nu(s)\right)^{1/2m}=\left(2m\int \lambda^{2m-1}\nu(\lambda,\infty)\,\dee\lambda\right)^{1/2m}\\
&\leq \left(2m\liminf_{n\to\infty}\int \lambda^{2m-1}\nu_{n}(\lambda,\infty)\,\dee\lambda\right)^{1/2m}=\left(\liminf_{n\to\infty}\int s^{2m}\,\dee\nu_{n}(s)\right)^{1/2m}\\
&\leq \liminf_{n\to\infty}\left(\int s^{2n}\,\dee\nu_{n}(s)\right)^{1/2n}=\liminf_{n\to\infty}\|c_n\xi_n\|_{2}^{1/n}.
\end{align*}
So
\[\sup_{m} \|a^{m}\xi\|_{2}^{1/m} \leq\liminf_{n\to\infty}\|c_n\xi_n\|_{2}^{1/n}.\]
By Lemma \ref{L:existence of growth rates limiting object},
 \[\rho(\nu) =\lim_{n\to\infty}\|a^{n}\xi\|_{2}^{1/n}\leq \liminf_{n\to\infty}\|c_n\xi_n\|_{2}^{1/n}.\]
Since $(\xi_n)_n$ is arbitrary, this shows $\rho(\nu) \le \underline{\Gr}( \bfc | \xi)$.

Next we will show $\overline{\Gr}( \bfc | \xi) \le \rho(\nu)$. So let $\varepsilon>0$. Choose a continuous function $f:[0,\infty) \to [0,1]$ such that $f(t)=1$ for all $t \in [0,\rho(\nu)]$ and $f(t)=0$ for all $t \ge \rho(\nu)+\varepsilon$. Let $\xi_n=f(b_n)\xi$.

We claim that $\xi_n \to \xi$ in $\rL^2(M,\tau)$. First observe that
$$\langle f(a)\xi,\xi\rangle = \int f~\dee\nu = \int 1 ~\dee\nu = \|\xi\|_2^2.$$
Since $\|f(a)\|_\infty \le 1$, we must have $f(a)\xi= \xi$.

Next, let $\nu'_n$ be the spectral measure of $b_n$ with respect to $\xi$.  
Note
\begin{eqnarray*}
\| \xi_n - \xi\|^2_2 &=& \|(1-f(b_n))\xi \|_2^2 = \langle (1-f(b_n))\xi, (1-f(b_n)) \xi\rangle \\
&=&  \langle (1-f(b_n))^2 \xi, \xi\rangle = \int (1-f)^2~\dee\nu'_n.
\end{eqnarray*}
By Lemma  \ref{L:weak convergence of spectral measures}, $\nu'_n \to \nu$ weakly. So $\int (1-f)^2~\dee\nu'_n \to \int (1-f)^2~\dee\nu$ as $n\to\infty$. Since $\nu$ is supported on $[0,\rho(\nu)]$ and $1-f=0$ on $[0,\rho(\nu)]$, it follows that $\int (1-f)^2~\dee\nu'_n \to 0$ as $n\to\infty$. This proves that $\| \xi_n - \xi\|_2 \to 0$ as $n\to\infty$. Thus, $(\xi_n)_n \in \Si(\xi)$.

Let $\nu_n$ be the spectral measure of $b_n$ with respect to $\xi_n$. We claim that $\dee \nu_n = f^2 \dee \nu'_n$. To see this, let $g:[0,\infty) \to \R$ be a continuous bounded function. Then
\begin{eqnarray*}
\int g ~\dee\nu_n &=& \langle g(b_n) \xi_n, \xi_n \rangle =  \langle g(b_n) f(b_n)\xi, f(b_n)\xi\rangle \\
&=& \langle f(b_n) g(b_n) f(b_n) \xi, \xi\rangle = \int g f^2~\dee\nu'_n.
\end{eqnarray*}
Since $g$ is arbitrary, this proves $\dee \nu_n = f^2 \dee \nu'_n$.

By Lemma  \ref{L:weak convergence of spectral measures}, $\nu_n \to \nu$ weakly.  So
\begin{align*}
\overline{\Gr}( \bfc | \xi) &\le \limsup_{n\to\infty}\|c_n \xi_{n}\|_{2}^{1/n}=\limsup_{n\to\infty}\ip{|c_n|^{2}\xi_{n},\xi_{n}}^{1/2n}\\
& =\limsup_{n\to\infty}\ip{b_n^{2n}\xi_{n},\xi_{n}}^{1/2n} =\limsup_{n\to\infty}\left(\int t^{2n}\,\dee\nu_{n}(t)\right)^{1/2n}\\
&=\limsup_{n\to\infty}\left(\int t^{2n}f(t)^2 \,\dee\nu'_{n}(t)\right)^{1/2n}\leq \rho(\nu)+\varepsilon.
\end{align*}
The last inequality occurs because $f(t)=0$ for all $t > \rho(\nu)+\varepsilon$.  Since $\varepsilon$ is arbitrary, $\overline{\Gr}( \bfc | \xi) \le \rho(\nu)$.

\end{proof}

\begin{cor}\label{C:smooth growth rates}
Let $X,\mu,f,M,\tau,c,\La$ be as in Theorem \ref{thm:main-general}. Then for a.e. $x\in X$ and every $\xi \in \rL^2(M,\tau)$ with $\xi \in \cap_n \dom(a^n)$,
\begin{eqnarray*}
\lim_{n\to\infty} \| \La(x)^n \xi\|_2^{1/n} &=&\overline{\Gr}( \bfc(x) | \xi) = \underline{\Gr}( \bfc(x) | \xi)
\end{eqnarray*}
where $\bfc(x) = (c(n,x))_n$. In particular, Theorem \ref{thm:main2} is true.
\end{cor}

\begin{proof}
Apply Lemma \ref{L:inclusions of candidate Osoledets spaces} with $a = \La(x)$, $c_n =c(n,x)$. Theorem \ref{thm:main2} now follows from  Corollary \ref{C:hypotheses}.
\end{proof}

\subsection{Essentially dense subspaces}\label{sec:essentially}

In this section, we review the notion of an essentially dense subspace. This is used in the last section to prove Theorem \ref{thm:invariance}.

\begin{defn}
Let $(M,\tau)$ be a semi-finite von Neumann algebra. A linear subspace $V\subseteq L^{2}(M,\tau)$ is called {\bf right-invariant} if $Vx\subseteq V$ for all $x\in M.$ We say that a right-invariant subspace $\mathcal{D}$ of $L^{2}(M,\tau)$ is {\bf essentially dense} if for every $\varepsilon>0,$ there is a projection $p\in M$ so that $\tau(I-p)\leq \varepsilon,$ and $\mathcal{D}\supseteq pL^{2}(M,\tau).$  If $\cK$ is a closed subspace of $L^2(M,\tau)$ and $W \subseteq \cK$ is a right-invariant subspace, we say that $W$ is essentially dense in $\cK$ if there exists $\cD$ essentially dense in $L^2(M,\tau)$ such that $\cD \cap \cK = W$.


\end{defn}

By definition of $\rL^0(M,\tau)$, if $a\in L^{0}(M,\tau)$, then $\dom(a)$ is essentially dense. It is an exercise to check that  the intersection of countably many essentially dense subspaces is essentially dense.

If $\mathcal{K}\subseteq L^{2}(M,\tau)$ is closed and right-invariant then the orthogonal projection onto $\mathcal{K},$ denoted $p_{\mathcal{K}},$ is in $M$ as a consequence of the Double Commutant Theorem.

Technically, our definition of \emph{essentially dense in $\mathcal{K}$} is different from the one in \cite[Definition 8.1]{Luck}. The next lemma shows that they are in fact equivalent.

\begin{lem}\label{L:equivalence ess dense}
Let $(M,\tau)$ be a semi-finite tracial von Neumann algebra, let $\mathcal{K}\subseteq L^{2}(M,\tau)$ be a closed, right-invariant subspace, and let $W\subseteq \mathcal{K}$ be a right-invariant subspace. Then the following are equivalent:
\begin{enumerate}
    \item $W$ is essentially dense in $\mathcal{K},$ \label{I:essentially dense 1 defn}
    \item there is an increasing sequence of projections $p_{n}\in M$ so that $p_{n}\to p_{\mathcal{K}}$ in the Strong Operator Topology, $\tau(p_\cK - p_n) \to 0$, and $W\supseteq p_{n}L^{2}(M,\tau).$ \label{I:essentially dense 2 defn}
\end{enumerate}

\end{lem}

\begin{proof}

(\ref{I:essentially dense 2 defn}) implies (\ref{I:essentially dense 1 defn}): Let $\mathcal{D}=W+(\id-p_{\mathcal{K}})L^{2}(M,\tau),$ then clearly $\mathcal{D}\cap \mathcal{K}=W.$  Let $q_n = \id-p_\cK + p_n$.






Then $\mathcal{D}\supseteq q_nL^{2}(M,\tau).$ Since $p_{n}(\id-p_{\mathcal{K}})=0,$ we also have that $q_n$ is an orthogonal projection. Also $\tau(\id-q_n) \to 0$.  Thus $\mathcal{D}$ is essentially dense.

(\ref{I:essentially dense 1 defn}) implies (\ref{I:essentially dense 2 defn}): Write $W=\mathcal{D}\cap \mathcal{K},$ where $\mathcal{D}$ is essentially dense in $L^{2}(M,\tau).$ By assumption, for every $n\in \N,$ we find a projection $f_{n}$ in $M$ so that $\tau(\id- f_{n})\leq 2^{-n}$ and $\mathcal{D}\supseteq f_{n}L^{2}(M,\tau).$ Set $q_{n}=\bigwedge_{m=n}^{\infty}f_{m},$ and $p_{n}=p_\cK\wedge q_{n}.$

Observe that
\[\tau(\id-q_{n})=\tau\left(\bigvee_{m=n}^{\infty}\id-f_{m}\right)\leq \sum_{m=n}^{\infty}\tau(\id-f_{n})\leq 2^{-n+1},\]
where in the first inequality we use that $\tau$ is normal.

By definition, $p_\cK - p_n=p_\cK - p_\cK\wedge q_n$ is the orthogonal projection onto $\cK \cap \textrm{range}(q_n)^\perp$, while $\id-q_n$ is the orthogonal projection onto $\textrm{range}(q_n)^\perp$. It follows that $p_\cK - p_n \le \id-q_n$. Therefore, $\tau(p_\cK - p_n) \leq 2^{-n+1} \to 0$ as $n\to\infty$.

Because $q_{n}$ is increasing in $n$, it also true that $p_{n}$ is increasing in $n$.
 So if $p_{\infty} = \sup_n p_n$ then $p_\infty$ is the SOT-limit of $p_n$ as $n\to\infty$.  
By definition of least upper bound, $p_{\infty}\leq p_{\mathcal{K}}$ and since $\tau(p_\cK - p_\infty) \leq \tau(p_\cK -p_n) \to 0$, the fact that $\tau$ is faithful implies $p_{\infty}=p_{\mathcal{K}}.$ So $p_n \to p_\cK$ in the Strong Operator Topology.  For each $n\in \N,$ we have that
\[p_{n}L^{2}(M,\tau)=\mathcal{K}\cap q_{n}L^{2}(M,\tau)\subseteq \mathcal{K}\cap f_{n}L^{2}(M,\tau)\subseteq \mathcal{K}\cap \mathcal{D}=W.\]

\end{proof}

\begin{lem}\label{L:basic facts injective ops}
Let $(M,\tau)$ be a tracial von Neumann algebra, and let $\mathcal{K}\subseteq L^{2}(M,\tau)$ be a closed and right-invariant subspace, fix an $a\in L^{0}(M,\tau).$
\begin{enumerate}
\item We have that $\dim_{M}(\overline{a(\mathcal{K}\cap \dom(a))})\leq \dim_{M}(\mathcal{K}),$ with equality if $\ker(a) \cap \cK = \{0\}$. \label{I:inj keeps dim}
\item We have that $(a^{-1}(\mathcal{K}))^{\perp}=\overline{(a^{*})(\mathcal{K}^{\perp}\cap \dom(a^{*}))}.$ \label{I:ortho complement formula}
\end{enumerate}
\end{lem}

\begin{proof}
Throughout, let $p$ be the orthogonal projection onto $\mathcal{K}.$


(\ref{I:inj keeps dim}): Let $ap=v|ap|$ be the polar decomposition. Then $v^{*}v=p_{\ker(ap)^{\perp}}$, $vv^{*}=p_{\overline{\im(ap)}}.$ Clearly $\ker(ap)\supseteq (1-p)(L^{2}(M,\tau))$ so
\[v^{*}v=p_{\ker(ap)^{\perp}}\leq p.\]
So:
\[\dim_{M}(\mathcal{K})=\tau(p)\geq \tau(v^{*}v)=\tau(vv^{*})=\dim_{M}(\overline{a(\mathcal{K}\cap \dom(a))}).\]

If $\ker(a) \cap \cK= \{0\}$, then in fact $v^*v=p_{\ker(ap)^\perp} = p$ so we have equality throughout.

(\ref{I:ortho complement formula}):
This is \cite[Lemma 3.4]{StinespringMeasOps}.


\end{proof}

\subsection{Invariance}\label{sec:invariance}

In this section, we prove Theorem \ref{thm:invariance}.
\begin{lem}\label{L:osoledets spaces lemma}
For $n \in \N$, let $c_n \in \rL^0(M,\tau)$ and $a\in \rL^0(M,\tau)$ with $a > 0$. Let $T_{n}=a^{-n}|c_n|^{2}a^{-n},$ and $S_{n}=T_{n}^{1/2n}.$ Suppose $S_{n}\to \id$ in measure and $|c_n|^{1/n} \to a$ in measure. For $t \in [-\infty,\infty)$, let
\begin{eqnarray*}
\mathcal{V}_{t} &=& \left\{\xi \in \rL^2(M,\tau):~ \liminf_{n\to\infty}\frac{1}{n}\log \|c_n\xi\|_2 \le t\right\} \\
\mathcal{H}_{t} &=& \left\{\xi \in \rL^2(M,\tau):~ \liminf_{n\to\infty} \frac{1}{n}\log \|a^n\xi\|_2 \le t\right\} = 1_{(-\infty,t]}(\log a)(\rL^2(M,\tau)).
\end{eqnarray*}
Then there exists an essentially dense subspace $\mathcal{D}$ of $\rL^2(M,\tau)$ such that
   \[\mathcal{D}\cap \mathcal{V}_{t}=\mathcal{D}\cap \mathcal{H}_{t}.\]
    In particular, we have that $\mathcal{H}_{t}=\overline{\mathcal{V}_{t}}.$
\end{lem}

\begin{proof}
Choose a decreasing sequence $(\varepsilon_{k})_k$ of positive real numbers tending to zero. Since $S_{n}\to \id$ in measure, there is an increasing sequence $(n_k)_k$
\[\sum_{k}\mu_{S_{n_k}}(1+\varepsilon_{k},\infty)<\infty.\]
By functional calculus, $\mu_{T_{n_k}}\left((1+\varepsilon_{k})^{2n_k},\infty\right) = \mu_{S_{n_k}}(1+\varepsilon_{k},\infty)$. So
\[\sum_{k}\mu_{T_{n_k}}\left((1+\varepsilon_k)^{2n_k},\infty\right)<\infty.\]
Let $r_k \in M$ denote orthogonal projection onto $\cR_{k}:=\overline{ a^{-n_k}1_{[0,(1+\varepsilon_k)^{2n_k}]}(T_{n_k})(\rL^2(M,\tau))}$. 
Let
\[\mathcal{D}=\bigcup_{l=1}^{\infty}\bigcap_{k=l}^{\infty}\Big(\dom(a^{n_k})\cap \cR_{k} \cap \dom(c_{n_k})\Big).\]
We claim that $\cD$ is essentially dense. 

Because $\dom(a^{n_k})$ and $\dom(c_{n_k})$ are essentially dense,
$$\dom(a^{n_k})\cap \dom(c_{n_k})\cap \cR_{k}$$  is essentially dense in $\cR_{k}$. So there exist projections $p_{k}\in M$ satisfying
\begin{itemize}
    \item $p_k \rL^2(M,\tau)\subseteq\dom(a^{n_k})\cap \cR_{k} \cap\dom(c_{n_k}) $
    \item $\tau(r_k - p_{k})  \leq \mu_{T_{n_k}}\left((1+\varepsilon_k)^{2n_k},\infty\right).$
\end{itemize}
For $l\in \N,$ set $q_{l}=\bigwedge_{k=l}^{\infty}p_{k}.$ Then
for every $l\in \N,$ we know that $\mathcal{D}\supseteq q_l\rL^2(M,\tau)$ and
\[\tau(\id-q_{l})\leq \tau(\id - r_l) + \tau(r_l-q_l) \le 2\sum_{k=l}^{\infty}\mu_{T_{n_k}}\left((1+\varepsilon_k)^{2n_k},\infty\right)\to_{l\to\infty}0.\]
So we have shown that $\mathcal{D}$ is essentially dense.

Now suppose that $\xi\in \mathcal{D}.$ Without loss of generality, $\|\xi\|_{2}=1.$ Then:
\begin{align*}
    \liminf_{n\to\infty}\|c_n \xi\|_{2}^{1/n}&\leq \liminf_{k\to\infty}\|c_{n_k}\xi\|_{2}^{1/n_k}=\liminf_{k\to\infty}\ip{|c_{n_k}|^{2}\xi,\xi}^{1/2n_k}\\
    &=\liminf_{k\to\infty}\left\langle a^{-n_k}|c_{n_k}|^{2}a^{-n_k}a^{n_k}\xi, a^{n_k}\xi\right\rangle^{\frac{1}{2n_k}} \\
    &=\liminf_{k\to\infty}\left\langle T_{n_k}a^{n_k}\xi, a^{n_k}\xi\right\rangle^{\frac{1}{2n_k}}.
\end{align*}
By choice of $\mathcal{D},$ we have that $a^{n_k}\xi\in 1_{[0,(1+\varepsilon_k)^{2n_k}]}(T_{n_k})(\rL^2(M,\tau)).$
So
\begin{align*}
    \liminf_{n\to\infty}\|c_n \xi\|_{2}^{1/n}  &\leq \liminf_{k\to\infty} (1+\varepsilon_k)\ip{a^{n_k}\xi,a^{n_k}\xi}^{\frac{1}{2n_k}}\\
    &=\liminf_{k\to\infty}(1+\varepsilon_k)\|a^{n_k}\xi\|_{2}^{1/n_k}=\lim_{n\to\infty}\|a^{n}\xi\|_2^{1/n},
\end{align*}
where the last equality holds by Lemma \ref{L:existence of growth rates limiting object}. So by Lemma \ref{L:inclusions of candidate Osoledets spaces},
\[\liminf_{n\to\infty}\|c_n \xi\|_{2}^{1/n}=\lim_{n\to\infty}\|a^{n}\xi\|_{2}^{1/n}.\]

Thus $\mathcal{D}\cap \mathcal{V}_{t}=\mathcal{D}\cap \mathcal{H}_{t}.$
Since Lemma \ref{L:inclusions of candidate Osoledets spaces} also shows that $\mathcal{V}_{t}\subseteq \mathcal{H}_{t},$ it follows that $\overline{\mathcal{V}_{t}}=\mathcal{H}_{t}$ since essentially dense subspaces are dense.

Lemma \ref{L:existence of growth rates limiting object} also shows that
$\cH_t = 1_{(0,e^t]}(a)(\rL^2(M,\tau))=1_{(-\infty,t]}(\log a)(\rL^2(M,\tau)).$
\end{proof}

\begin{cor}\label{C:Osoledets space cocycle moving}
Let $X,\mu,f,M,\tau,c,\La$ be as in Theorem \ref{thm:main-general}.
For $x\in X$, let
\begin{eqnarray*}
\mathcal{H}_{t}(x) &=& \left\{\xi \in \rL^2(M,\tau):~ \liminf_{n\to\infty} \frac{1}{n} \log \|\La(x)^n\xi\|_2 \le t\right\} = 1_{(-\infty,t]}(\log \La(x))(\rL^2(M,\tau)).
\end{eqnarray*}
Then the Oseledets subspaces and Lyapunov distributions are invariant in the following sense. For a.e. $x\in X$,
\begin{eqnarray*}
c(1,x)\mathcal{H}_{t}(x) &=& \mathcal{H}_{t}(f(x)), \\
\mu_{\log \La(f(x))} &=& \mu_{\log \La(x)}.
\end{eqnarray*}
In particular, Theorem \ref{thm:invariance} is true.
\end{cor}

\begin{proof}
For $x\in X$, let
\begin{eqnarray*}
\mathcal{V}_{t}(x) &=& \left\{\xi \in \rL^2(M,\tau):~ \liminf_{n\to\infty} \frac{1}{n} \log \|c(n,x)\xi\|_2 \le t\right\}.
\end{eqnarray*}
We claim that  $c(1,x)\mathcal{V}_{t}(x)= \mathcal{V}_{t}(f(x)).$ To see this, let $\xi \in \cV_t(x)$. Then for any $\eps>0$ and $N\in \N$ there exists $n>N$ such that $\frac{1}{n} \log \|c(n,x)\xi\|_2 \le t+\eps$. Since $c(n-1,f(x))c(1,x)=c(n,x)$, we can rewrite this as
$$\frac{1}{n} \log \|c(n-1,f(x))c(1,x)\xi\|_2 \le t+\eps.$$
Since this is true for every $\eps>0$ and $N\in \N$, it follows that $c(1,x)\xi \in \mathcal{V}_{t}(f(x))$. Thus $c(1,x)\mathcal{V}_{t}(x)\subseteq \mathcal{V}_{t}(f(x)).$

Conversely, if $\xi \in \mathcal{V}_{t}(f(x))$ then for every $\eps>0$ and $M\in \N$, there exists $m>M$ such that $\frac{1}{m} \log \|c(m,f(x))\xi\|_2 \le t+\eps$. By the cocycle equation, $c(m,f(x))=c(m+1,x)c(1,x)^{-1}$.  So we can rewrite this as
$$\frac{1}{m} \log \|c(m+1,x)c(1,x)^{-1}\xi\|_2 \le t+\eps.$$
Since this is true for every $\eps>0$ and $M\in \N$, it follows that $c(1,x)^{-1}\xi \in \mathcal{V}_{t}(x)$. Thus $c(1,x)^{-1}\mathcal{V}_{t}(f(x))\supseteq \mathcal{V}_{t}(x).$ This proves $c(1,x)\mathcal{V}_{t}(x)= \mathcal{V}_{t}(f(x)).$ Applying the ``in particular" part of Lemma \ref{L:osoledets spaces lemma} shows that $c(1,x)\mathcal{H}_{t}(x)=\mathcal{H}_{t}(f(x)).$




To prove invariance of the Lyapunov distribution, let $s<t$ be real numbers. Fix $x\in X$ and define $\phi \in M$ by
$$\phi = (p_{\cH_t(f(x))} - p_{\cH_s(f(x))})c(1,x).$$
Because the kernel of $p_{\cH_t(f(x))} - p_{\cH_s(f(x))}$ (viewed as a self-map of $\cH_t(f(x))$) is $\cH_s(f(x))$, the kernel of $\phi$ restricted to $\cH_t(x)$ is $c(1,x)^{-1}\cH_s(f(x)) = \cH_s(x)$. Thus, the restriction of $\phi$ to $\cH_t(x) \cap \cH_s(x)^\perp$ is one-to-one. Moreover,
$$\phi(\cH_t(x) \cap \cH_s(x)^\perp)=\cH_t(f(x)) \cap \cH_s(f(x))^\perp.$$
Because injective elements of $M$ preserve von Neumann dimension (by Lemma \ref{L:basic facts injective ops}),
$$\dim_M(\cH_t(x) \cap \cH_s(x)^\perp)=\dim_M(\cH_t(f(x)) \cap \cH_s(f(x))^\perp).$$
By definition, $\mu_{\log\La(x)}(s,t] = \dim_M(\cH_t(x) \cap \cH_s(x)^\perp)$. Since $s<t$ are arbitrary, this shows $\mu_{\log \La(f(x))}=\mu_{\log \Lambda(x)}$.

Theorem \ref{thm:invariance} now follows from Corollary \ref{C:hypotheses}.
\end{proof}

\appendix

\section{Diffuse finite von Neumann algebras}\label{sec:diffuse}

We recall some definitions from the introduction.

\begin{defn}
Let $M\subseteq B(\mathcal{H}),$ $N\subseteq B(\mathcal{K})$ be von Neumann algebras. A linear map $\theta\colon M\to N$ is \emph{normal} if $\theta\big|_{\{x\in M:\|x\|\leq 1\}}$ is weak operator topology-weak operator topology continuous.
\end{defn}

For linear functionals, this agrees with our previous notion of normality introduced in Section \ref{sec: vNa intro}. See \cite[Theorem 7.1.12]{KadisonRingroseII}.

It can be shown that a $C^{*}$-algebra is a von Neumann algebra if and only if $M$ has a predual (i.e. is isometrically isomorphic to the dual of a Banach space), and that moreover this predual is unique \cite[Theorem III.3.5 and Corollary III.3.9]{MR1873025}. It is also known that a linear map is normal if and only if it is weak$^{*}$-weak$^{*}$ continuous. This explains why normality is the correct continuity condition for maps between von Neumann algebras: it is \emph{intrinsic} to the algebra and does not depend upon how it is represented.

Suppose that $(M,\tau)$ is a finite tracial von Neumann algebra. We leave it as an exercise to verify that a sequence (or more generally a net) $x_{n}\in M$ with $\|x_{n}\|\leq 1$ tends to $x$ in the weak operator topology acting on $\rL^{2}(M,\tau)$ if and only if
\begin{equation}\label{E:WOT convergence by traces}
\tau(x_{n}a)\to_{n\to\infty}\tau(xa)\textnormal{ for all $a\in M$.}
\end{equation}
In particular, normality of the trace implies that the regular representation of $M$ on $\rL^{2}(M,\tau)$ is normal.
It follows from the fact that (\ref{E:WOT convergence by traces}) characterizes convergence in the WOT on the unit ball of $M$ that a trace-preserving $*$-homomorphism between tracial finite von Neumann algebras is normal.

\begin{defn}
Let $M$ be a von Neumann algebra. A projection $p\in M$ is a \textbf{minimal projection} if whenever $q\in M$ is a projection with $q\leq p$ we have either $q=p$ or $q=0$. We say that $M$ is \textbf{diffuse} if it has no nonzero minimal projections.
\end{defn}

\begin{prop}\label{prop:diffuse vNa TFAE}
Let $(M,\tau)$ be a finite tracial von Neumann algebra. The following are equivalent:
\begin{enumerate}
\item $M$ is diffuse, \label{item:diffuse appendix}
\item  \label{item: embedding of the unit interval appendix} there is an atomless standard probability space $(X,\mu)$ and an injective, trace-preserving $*$-homomorphism $\iota\colon \rL^{\infty}(X,\mu)\to M$, (here $\rL^{\infty}(X,\mu)$ is equipped with the trace $\int \cdot\,d\mu$),
\item\label{item:RL lemma appendix} there is a sequence $u_{n}$ of unitaries in $M$ with $u_{n}\to 0$ in the weak operator topology. 
\end{enumerate}

\end{prop}

\begin{proof}

(\ref{item:diffuse appendix}) implies (\ref{item: embedding of the unit interval appendix}):
Let $A\subseteq M$ be a unital, abelian von Neumann algebra which is maximal under inclusion among all abelian von Neumann subalgebras of $M$. The existence of such an $A$ follows from Zorn's Lemma.
We start with the following claim.

\emph{Claim: $A$ is diffuse.}

To prove the claim suppose, for contradiction, that $p\in A$ is a nonzero minimal projection. Since every von Neumann algebra is the norm closed linear span of its projections by \cite[Proposition IX.4.8]{Conway}, it follows that the von Neumann algebra $Ap$ is $1$-dimensional. Since $M$ is diffuse, there exists a nonzero projection $q\in M$ so that $q\leq p$ and $q\ne p$. We claim that $q$ commutes with $A$. To see this, let $a\in A$. Since $Ap$ is $1$-dimensional, there is a $\lambda\in \C$ so that $ap=\lambda p.$ So
\[aq=(ap+a(1-p))q=(\lambda p+a(1-p))q=\lambda q,\]
the last equation following as $q\leq p$ implies that $pq=q$ and $(1-p)q=0$. Similarly,
\[qa=q(ap+a(1-p))=q(\lambda p+(1-p)a)=\lambda q,\]
where in the second equality we are using that $A$ is abelian and that $p\in A$. So $q$ commutes with $A$. Now set
\[B=\overline{\Span(A\cup \{qa:a\in A\})}^{WOT}.\]
Because $q$ commutes with $A$, we know that $\Span(A\cup \{qa:a\in A\})$ is an abelian $*$-algebra. So $B$ is an abelian von Neumann algebra containing $A$. By maximality, this forces $B=A$ and so $q\in A$. But then $p$ is not a minimal projection in $A$, which gives a contradiction.  This proves the claim.

Having shown the claim, note that by \cite[Theorem III.1.18]{MR1873025} we have that $(A,\tau)\cong (\rL^{\infty}(Y,\nu),\int \cdot\,d\nu)$ where $Y$ is a compact Hausdorff space and $\nu$ is a probability measure on $Y$ (we can choose $Y$ to be metrizable if and only if $\rL^{2}(Y,\nu)$ is separable). Saying that $\rL^{\infty}(Y,\nu)$ is diffuse is equivalent to saying that for every measurable $E\subseteq Y$ with $\nu(E)>0$ there is a measurable $F\subseteq E$ with $\nu(F)<\nu(E)$. By a standard measure theory exercise this forces $\{\nu(F):F\subseteq E \textnormal{ is measurable}\}=[0,\nu(E)]$ for every measurable $E\subseteq Y$. By a recursive construction, this implies that for every $n\in \N$ and for every $\sigma\in \{0,1\}^{n}$ there is a measurable $E_{\sigma}\subseteq Y$ which satisfy the following properties:
\begin{itemize}
\item $\nu(E_{\sigma})=2^{-n}$ for every $n\in \N$ and every $\sigma\in \{0,1\}^{n}$,
\item for every $n\in \N$ and every $\sigma,\omega\in \{0,1\}^{n}$ with $\sigma\ne \omega$ we have $E_{\sigma}\cap E_{\omega}=\varnothing$,
\item $X=E_{0}\cup E_{1},$
\item for every $\sigma\in \{0,1\}^{n}$ we have $E_{\sigma}=E_{\sigma 0}\cup E_{\sigma 1}$ where $\sigma 0=(\sigma_{1},\cdots,\sigma_{n},0)$, $\sigma 1=(\sigma_{1},\cdots,\sigma_{n},1)\in \{0,1\}^{n+1}$.
\end{itemize}
It follows from the above properties that there is a unique map $\pi\colon Y\to \{0,1\}^{\N}$ with the property that for all $y\in Y,n\in \N$ we have that $(\pi(y)_{1},\cdots,\pi(y)_{n})=\sigma$ if and only if $y\in E_{\sigma}$. Moreover, the above properties imply that $\mu=\pi_{*}\nu$ is the infinite power of the uniform measure on $\{0,1\}$. The map $\pi$ induces a trace-preserving $*$-homomorphism
\[\theta \colon \rL^{\infty}(\{0,1\}^{\N},\mu)\to \rL^{\infty}(Y,\nu)\cong A\]
by $\theta(f)=f\circ \pi$. Combining with the inclusion of $A$ into $M$ we have a trace-preserving $*$-homomorphism $\rL^{\infty}(\{0,1\}^{\N},\mu)\to M$. Since $(\{0,1\}^{\N},\mu)$ is an atomless standard probability space, we are done.

(\ref{item: embedding of the unit interval appendix}) implies (\ref{item:RL lemma appendix}): By \cite[Theorem V.1.22]{MR1873025} we may, and will, assume that $(X,\mu)=([0,1],m)$ where $m$ is Lebesgue measure. Define unitaries $v_{n}$ in $\rL^{\infty}([0,1],m)$ by $v_{n}(x)=e^{2\pi i nx}$. Let $\iota:\rL^{\infty}([0,1],m) \to B(\rL^{2}([0,1],m))$ be the inclusion map. So for $\theta \in \rL^{\infty}([0,1],m)$ and $f \in \rL^{2}([0,1],m)$, $\iota(\theta)(f)(x) = \theta(x)f(x).$

It follows from the Riemann-Lebesgue Lemma that $\iota(v_{n})\to 0$ in the weak operator topology (as $n\to\infty$). By normality, we have that $v_{n}\to 0$ in the weak operator topology as well.

(\ref{item:RL lemma appendix}) implies (\ref{item:diffuse appendix}):
Suppose that $p\in M$ is a minimal projection. Let $z$ be the \textbf{central support} of $M$, namely the smallest projection in the center of $M$ which dominates $p$. Since $M$ is a minimal projection, it follows by \cite[Proposition 6.4.3 and Corollary 6.5.3]{KadisonRingroseII} that we have a normal isomorphism of von Neumann algebras $Mz\cong B(\mathcal{K})$ for some Hilbert space $\mathcal{K}$. Since $Mz$ (and thus $B(\mathcal{K})$) has a faithful, finite, normal tracial state it follows that $\mathcal{K}$ is finite-dimensional. So $Mz$ is finite-dimensional. Note that $u_{n}z\to 0$ in the weak operator topology. Since there is only one Hausdorff vector space topology on a finite-dimensional space, we know that $\|u_{n}z\|\to 0$. But by unitarity we know that $\|z\|=\|u_{n}z\|$. So $z=0$, and the fact that $p\leq z$ implies that $p=0$. So $M$ has no nonzero minimal projection, and thus $M$ is diffuse.

\end{proof}

As we remarked in the introduction, for most of the tracial von Neumann algebras we are interested in, the limiting operator in our Multiplicative Ergodic Theorem cannot be compact (unless it is zero). This is because of the following result.

\begin{prop}\label{prop: no nonzero compacts appendix}
Let $M\subseteq B(\mathcal{H})$ be a diffuse von Neumann algebra. Then $M$ does not contain any nonzero compact operators.
\end{prop}

\begin{proof}
Suppose that $x\in M$ is a nonzero compact operator. Then $x^{*}x$ is also a nonzero compact operator, and is self-adjoint. By the spectral theorem for compact normal operators, there is a $\lambda\in (0,\infty)$ which is a eigenvalue for $x^{*}x$ with finite dimensional eigenspace. Let $p$ be the projection onto the kernel of $\lambda I-x^{*}x$. By functional calculus, we know that $p=1_{\{\lambda\}}(x^{*}x)$. So it follows by \cite[Proposition IX.8.1]{Conway} that $p\in M.$ Since $p\mathcal{H}$ is finite-dimensional, we may choose a nonzero projection $q\leq p$ so that $\dim(q\mathcal{H})\leq \dim(e\mathcal{H})$ whenever $e\in M$ is a nonzero projection with $e\leq p.$  It is direct to see that $q$ is a minimal projection in $M$, and this contradicts our assumption that $M$ is diffuse.

\end{proof}

We close with some facts about $\rL^{2}(M,\tau)$ and $\GL^{2}(M,\tau)$ when $M$ is diffuse. We need a few preliminaries. Suppose that $(N,\widetilde{\tau})$, $(M,\tau)$ are tracial-von Neumann algebras and that $\iota\colon N\to M$ is a trace-preserving $*$-homomorphism. Then $\mu_{|\iota(x)|}=\mu_{|x|}$ for all $x\in N$ and thus $\iota$ is uniformly continuous for the measure topology (with respect to the unique translation-invariant uniform structure on a topological vector space). By \cite[Lemma 2.3 and Theorem 2.5]{TakesakiII}, we may regard $\rL^{0}(M,\tau)$ and $\rL^{0}(N,\widetilde{\tau})$ as the measure topology completions of $M$ and $N$, and so there is a unique measure topology continuous extension $\rL^{0}(N,\widetilde{\tau})\to \rL^{0}(M,\tau)$ of $\iota$. We will still use $\iota$ for this extension. Since $\rL^{0}(M,\tau)$, $\rL^{0}(N,\widetilde{\tau})$ are $*$-algebras extending $M, N$, the map $\iota\colon \rL^{0}(N,\widetilde{\tau})\to \rL^{0}(M,\tau)$ is a $*$-homomorphism. By measure topology continuity of $\iota$, density of $N$ in $\rL^{0}(N,\widetilde{\tau})$, and Proposition \ref{P:conv in meas implies wk* conv} we have $\iota(f(|x|))=f(\iota(|x|))$ for all continuous, compactly supported $f\colon [0,\infty)\to \R$. In particular
\[\int f\,d\mu_{|\iota(x)|}=\tau(f(|\iota(x)|))=\widetilde{\tau}(f(|x|))=\int f\,d\mu_{|x|},\]
for all continuous, compactly supported $f\colon [0,\infty)\to \R$.
Thus $\mu_{|\iota(x)|}=\mu_{|x|}$. 

\begin{prop}\label{prop:facts about GL2/L2 diffuse case}
Let $(M,\tau)$ be a semifinite tracial von Neumann algebra and suppose that $M$ is diffuse. Then:
\begin{enumerate}
\item $\rL^{2}(M,\tau)$ is not closed under products,\label{item: not closed under products appendix}
\item there is an unbounded operator in $\rL^{2}(M,\tau)$, \label{item: containing an unbounded operator appendix}
\item $\mathcal{P}^{\infty}(M,\tau)$ is not complete in the metric  $d_{\mathcal{P}}$. \label{item: not complete appendix}
\end{enumerate}

\end{prop}

\begin{proof}

By Proposition \ref{prop:diffuse vNa TFAE} and \cite[Proposition V.1.40]{MR1873025} there is a nonempty set $J$ and a trace-preserving map $\iota$ of $(\rL^{\infty}([0,1]\times J),\int \cdot \,d(m\otimes \eta))$ into $(M,\tau)$ where $m$ is Lebesgue measure and $\eta$ is counting measure on the set $J$.

(\ref{item: not closed under products appendix}): Let $g\colon [0,1]\to \C$ be a measurable function so that $g\in \rL^{2}([0,1])$ but $g\notin \rL^{4}([0,1])$. E.g. we can take $g(x)=1_{(0,1]}(x)x^{-1/4}$. Fix $j_{0}\in J$ and define $f\colon X\times J\to \C$ by $f(x,j)=1_{\{j_{0}\}}(j)g(x)$. Since $\iota$ is trace-preserving, we have that
\[\mu_{|\iota(f)|}=\mu_{|f|}=|g|_{*}(m).\]
So $\|\iota(f)\|_{2}=\|g\|_{2}<\infty$ and similarly $\|\iota(f)^{2}\|_{2}=\|\iota(f^{2})\|_{2}=\|g\|_{4}=\infty$. So $\iota(f)\in \rL^{2}(M,\tau)$ and $\iota(f)^{2}\notin \rL^{2}(M,\tau)$. So $\rL^{2}(M,\tau)$ is not closed under products. 

(\ref{item: containing an unbounded operator appendix}): Let $f\in \rL^{2}([0,1]\times J)$ with $f\notin \rL^{\infty}([0,1]\times J)$. As in (\ref{item: not closed under products appendix}) we have that $\iota(f)\in \rL^{2}(M,\tau)$ and $\|\iota(f)\|_{\infty}=\|f\|_{\infty}=\infty$, so $\iota(f)$ is not bounded.

(\ref{item: not complete appendix}): Let $g\in \rL^{2}([0,1]\times J)$ with $g\notin \rL^{\infty}([0,1]\times J)$. Let  $f=\exp(|g|)$. Since $\mu_{|\iota(f)|}=\mu_{|f|}$, we have that $\log|\iota(f)|\in \rL^{2}(M,\tau)$ but $\log|\iota(f)|\notin M$. This $\mathcal{P}(M,\tau)\neq \mathcal{P}^{\infty}(M,\tau)$.  So $\mathcal{P}^{\infty}(M,\tau)$ is a proper dense subspace of a complete metric space, and is thus not complete. 

\end{proof}

\section{More examples}\label{sec:examples-more}

\subsection{Example: the hyperfinite factor}\label{intro-hyperfinite}

Let $M_n(\C)$ denote the algebra of $n\times n$ square matrices with entries in $\C$. Let $\phi_n:M_n(\C) \to M_{2n}(\C)$ be the homomorphism
\begin{displaymath}
A \mapsto  \left( \begin{array}{cc}
A & 0 \\
0 & A
\end{array}\right).
\end{displaymath}
Let $\tau_n:M_n(\C) \to \C$ be the normalized trace defined by $\tau_n(A) = \frac{1}{n} \sum_{i=1}^n A_{ii}$. Note that $\phi_n$ preserves normalized traces in the sense that $\tau_n(A) = \tau_{2n}(\phi_n(A))$.

Let $\cA:=\cup_n M_{2^n}(\C)$ be the direct limit of the matrix algebras $M_{2^n}(\C)$ under the family of maps $\phi_n$ (in other words, $\cA$ is the disjoint union of $M_{2^n}(\C)$ after quotienting out by the equivalence relation generated by $A \sim \phi_{2^n}(A)$ for all $A \in M_{2^n}(\C)$). Because the maps $\phi_n$ preserve normalized traces, there is a trace $\tau:\cA \to \C$ satisfying $\tau(A) = \tau_{2^n}(A)$ for all $A \in M_{2^n}(\C)$.

Define the inner product $\langle A, B \rangle := \tau(A^*B)$ for all $A,B \in \cA$. The Hilbert space completion of this inner product is a Hilbert space, denoted by $\rL^2(\cA,\tau)$. Moreover, each operator $A \in \cA$ acts on $\rL^2(\cA,\tau)$ by left-composition. Thus we have embedded $\cA$ into the algebra $B(\rL^2(\cA,\tau))$ of bounded operators. Let $R$ denote the weak operator closure of $\cA$ in $B(\rL^2(\cA,\tau))$. The trace $\tau$ admits a unique normal extension to $R$. The pair $(R,\tau)$ is called the {\bf hyperfinite II$_1$-factor}.  For more details, see \cite[Theorem 11.2.2]{anantharaman-popa}.

\subsection{Example: infinite tensor powers}\label{intro-tensor}
Let $(M,\tau)$ be a finite tracial von Neumann algebra with $M \subset B(\cH)$ SOT-closed. Then there is a natural embedding of $B(\ell^2(\N))\otimes M \to B(\ell^2(\N) \otimes \cH)$. Let $M_\infty:=B(\ell^2(\N)) \botimes M$ be the SOT-closure of $B(\ell^2(\N))\otimes M$ in $B(\ell^2(\N) \otimes \cH)$. We can think of elements of $M_\infty$ as $\N\times \N$ matrices with entries in $M$. Define a trace $\tau_\infty$ on $M_\infty$ by
$$\tau_\infty(x)=\sum_{i\in \N} \tau(x_{ii}).$$
Equipped with this trace, $M_\infty$ is a semi-finite tracial von Neumann algebra. Moreover, this is the typical example of a non-finite semi-finite tracial von Neumann algebra. See \cite[Chapter 8]{anantharaman-popa} for details.

\section{Glossary}\label{sec:glossary}
Let $\cH$ be a separable Hilbert space and $B(\cH)$ be the algebra of all bounded linear operators on $\cH$. 
\begin{itemize}
\item A {\bf von Neumann algebra} $M$ is a sub-algebra of $B(\cH)$ satisfying: $1 \in M$, $M$ is WOT-closed and $*$-closed.
\item A {\bf tracial von Neumann algebra} is a pair $(M,\tau)$ where $M$ is a von Neumann algebra and $\tau$ is a faithful, normal, semi-finite trace. If $\tau(\id)<\infty$ then  $(M,\tau)$ is a {\bf finite tracial von Neumann algebra}.
\item An operator $x \in B(\cH)$ is {\bf positive} if $\langle x\xi,\xi\rangle \ge 0$ $\forall \xi \in \cH$.
\item $x \ge y$ iff $x-y$ is positive.
\item $M_+ =\{x \in M:~x \ge 0\}$. 
\item $M^\times = \GL^\infty(M,\tau)$ is the group of operators $x\in M$ with $x^{-1}\in M$.
\item $M_{sa} = \{x \in M:~ x=x^*\}$. 
\item $\cP^\infty = M_+ \cap M^\times=\exp(M_{sa})=\cP \cap M^\times$.
\end{itemize}

\begin{itemize}
\item $\cN =\{x \in M:~ \tau(x^*x)<\infty\}$.
\item $\rL^2(M,\tau)$ is the Hilbert space completion of $\cN$ with respect to the inner product $\langle x,y \rangle = \tau(x^*y)$. Also $\rL^2(M,\tau) = \{x \in \rL^0(M,\tau):~ \tau(x^*x)<\infty\}$. 
\item $\rL^2(M,\tau)_{sa} = \{x \in \rL^2(M,\tau):~x =x^*\}$. 
\item $\rL^0(M,\tau)$ is the algebra of $\tau$-measurable operators affiliated with $M$. 
\item $\rL^0(M,\tau)_+ = \{x \in \rL^0(M,\tau):~x \ge 0\}$. 
\item $\rL^0(M,\tau)^\times$ is the group of operators $x\in \rL^0(M,\tau)$ with $x^{-1}\in \rL^0(M,\tau)$. 
\item $\rL^0(M,\tau)_{sa} = \{x \in \rL^0(M,\tau):~x =x^*\}$. 
\item $\GL^2(M,\tau) = \{x\in \rL^0(M,\tau)^\times:~ \log|x| \in \rL^2(M,\tau)\}$. 
\item $\cP= \rL^0(M,\tau)_+ \cap \GL^2(M,\tau) = \exp(\rL^2(M,\tau)_{sa})$.

\end{itemize}

\bibliography{biblio}

\def\cprime{$'$} \def\cprime{$'$}
  \def\cfudot#1{\ifmmode\setbox7\hbox{$\accent"5E#1$}\else
  \setbox7\hbox{\accent"5E#1}\penalty 10000\relax\fi\raise 1\ht7
  \hbox{\raise.1ex\hbox to 1\wd7{\hss.\hss}}\penalty 10000 \hskip-1\wd7\penalty
  10000\box7} \def\cprime{$'$} \def\cprime{$'$} \def\cprime{$'$}
  \def\cprime{$'$} \def\cprime{$'$} \def\cprime{$'$} \def\cprime{$'$}
\begin{thebibliography}{GTQ15}

\bibitem[AL06]{MR2254561}
E.~Andruchow and G.~Larotonda.
\newblock Nonpositively curved metric in the positive cone of a finite von
  {N}eumann algebra.
\newblock {\em J. London Math. Soc. (2)}, 74(1):205--218, 2006.

\bibitem[AP16]{anantharaman-popa}
Claire Anantharaman and Sorin Popa.
\newblock An introduction to {$II_1$} factors.
\newblock {\em book in progress}, 2016.

\bibitem[BH99]{bridson-haefliger-book}
Martin~R. Bridson and Andr{\'e} Haefliger.
\newblock {\em Metric spaces of non-positive curvature}, volume 319 of {\em
  Grundlehren der Mathematischen Wissenschaften [Fundamental Principles of
  Mathematical Sciences]}.
\newblock Springer-Verlag, Berlin, 1999.

\bibitem[BK90]{BKMeasOps}
Lawrence~G. Brown and Hideki Kosaki.
\newblock Jensen's inequality in semi-finite von {N}eumann algebras.
\newblock {\em J. Operator Theory}, 23(1):3--19, 1990.

\bibitem[Blu16]{MR3485402}
Alex Blumenthal.
\newblock A volume-based approach to the multiplicative ergodic theorem on
  {B}anach spaces.
\newblock {\em Discrete Contin. Dyn. Syst.}, 36(5):2377--2403, 2016.

\bibitem[CL10]{MR2643829}
Cristian Conde and Gabriel Larotonda.
\newblock Spaces of nonpositive curvature arising from a finite algebra.
\newblock {\em J. Math. Anal. Appl.}, 368(2):636--649, 2010.

\bibitem[Con76]{MR0454659}
A.~Connes.
\newblock Classification of injective factors. {C}ases {$II_{1},$} {$II_{\infty
  },$} {$III_{\lambda },$} {$\lambda \not=1$}.
\newblock {\em Ann. of Math. (2)}, 104(1):73--115, 1976.

\bibitem[Con90]{Conway}
John~B. Conway.
\newblock {\em A course in functional analysis}, volume~96 of {\em Graduate
  Texts in Mathematics}.
\newblock Springer-Verlag, New York, second edition, 1990.

\bibitem[Con00]{ConwayOT}
John~B. Conway.
\newblock {\em A course in operator theory}, volume~21 of {\em Graduate Studies
  in Mathematics}.
\newblock American Mathematical Society, Providence, RI, 2000.

\bibitem[Dix81]{MR641217}
Jacques Dixmier.
\newblock {\em von {N}eumann algebras}, volume~27 of {\em North-Holland
  Mathematical Library}.
\newblock North-Holland Publishing Co., Amsterdam-New York, 1981.
\newblock With a preface by E. C. Lance, Translated from the second French
  edition by F. Jellett.

\bibitem[FK52]{MR0052696}
Bent Fuglede and Richard~V. Kadison.
\newblock Determinant theory in finite factors.
\newblock {\em Ann. of Math. (2)}, 55:520--530, 1952.

\bibitem[FK86]{MR840845}
Thierry Fack and Hideki Kosaki.
\newblock Generalized {$s$}-numbers of {$\tau$}-measurable operators.
\newblock {\em Pacific J. Math.}, 123(2):269--300, 1986.

\bibitem[Fol99]{Folland}
Gerald~B. Folland.
\newblock {\em Real analysis}.
\newblock Pure and Applied Mathematics (New York). John Wiley \& Sons, Inc.,
  New York, second edition, 1999.
\newblock Modern techniques and their applications, A Wiley-Interscience
  Publication.

\bibitem[GTQ15]{MR3400385}
Cecilia Gonz\'{a}lez-Tokman and Anthony Quas.
\newblock A concise proof of the multiplicative ergodic theorem on {B}anach
  spaces.
\newblock {\em J. Mod. Dyn.}, 9:237--255, 2015.

\bibitem[HS07]{MR2339369}
Uffe Haagerup and Hanne Schultz.
\newblock Brown measures of unbounded operators affiliated with a finite von
  {N}eumann algebra.
\newblock {\em Math. Scand.}, 100(2):209--263, 2007.

\bibitem[HS09]{MR2511586}
Uffe Haagerup and Hanne Schultz.
\newblock Invariant subspaces for operators in a general {${\rm II}_1$}-factor.
\newblock {\em Publ. Math. Inst. Hautes \'{E}tudes Sci.}, (109):19--111, 2009.

\bibitem[Joh93]{MR1123654}
G.~W. Johnson.
\newblock The product of strong operator measurable functions is strong
  operator measurable.
\newblock {\em Proc. Amer. Math. Soc.}, 117(4):1097--1104, 1993.

\bibitem[Kap51]{MR50181}
Irving Kaplansky.
\newblock A theorem on rings of operators.
\newblock {\em Pacific J. Math.}, 1:227--232, 1951.

\bibitem[Kau87]{MR947327}
V.~A. Kau{\i}manovich.
\newblock Lyapunov exponents, symmetric spaces and a multiplicative ergodic
  theorem for semisimple {L}ie groups.
\newblock {\em Zap. Nauchn. Sem. Leningrad. Otdel. Mat. Inst. Steklov. (LOMI)},
  164(Differentsial\cprime naya Geom. Gruppy Li i Mekh. IX):29--46, 196--197,
  1987.

\bibitem[KM99]{MR1729880}
Anders Karlsson and Gregory~A. Margulis.
\newblock A multiplicative ergodic theorem and nonpositively curved spaces.
\newblock {\em Comm. Math. Phys.}, 208(1):107--123, 1999.

\bibitem[KR97]{KadisonRingroseII}
Richard~V. Kadison and John~R. Ringrose.
\newblock {\em Fundamentals of the theory of operator algebras. {V}ol. {II}},
  volume~16 of {\em Graduate Studies in Mathematics}.
\newblock American Mathematical Society, Providence, RI, 1997.
\newblock Advanced theory, Corrected reprint of the 1986 original.

\bibitem[LL10]{MR2674952}
Zeng Lian and Kening Lu.
\newblock Lyapunov exponents and invariant manifolds for random dynamical
  systems in a {B}anach space.
\newblock {\em Mem. Amer. Math. Soc.}, 206(967):vi+106, 2010.

\bibitem[L{\"u}c02]{Luck}
Wolfgang L{\"u}ck.
\newblock {\em {$L^2$}-invariants: theory and applications to geometry and
  {$K$}-theory}, volume~44 of {\em Ergebnisse der Mathematik und ihrer
  Grenzgebiete. 3. Folge. A Series of Modern Surveys in Mathematics [Results in
  Mathematics and Related Areas. 3rd Series. A Series of Modern Surveys in
  Mathematics]}.
\newblock Springer-Verlag, Berlin, 2002.

\bibitem[Mn83]{MR730286}
Ricardo Ma\~n\'e.
\newblock Lyapounov exponents and stable manifolds for compact transformations.
\newblock In {\em Geometric dynamics ({R}io de {J}aneiro, 1981)}, volume 1007
  of {\em Lecture Notes in Math.}, pages 522--577. Springer, Berlin, 1983.

\bibitem[Pad67]{MR212581}
A.~R. Padmanabhan.
\newblock Convergence in measure and related results in finite rings of
  operators.
\newblock {\em Trans. Amer. Math. Soc.}, 128:359--378, 1967.

\bibitem[RS80]{MR751959}
Michael Reed and Barry Simon.
\newblock {\em Methods of modern mathematical physics. {I}}.
\newblock Academic Press, Inc. [Harcourt Brace Jovanovich, Publishers], New
  York, second edition, 1980.
\newblock Functional analysis.

\bibitem[Rue82]{MR647807}
David Ruelle.
\newblock Characteristic exponents and invariant manifolds in {H}ilbert space.
\newblock {\em Ann. of Math. (2)}, 115(2):243--290, 1982.

\bibitem[Sch91]{MR1178957}
Kay-Uwe Schauml\"{o}ffel.
\newblock Multiplicative ergodic theorems in infinite dimensions.
\newblock In {\em Lyapunov exponents ({O}berwolfach, 1990)}, volume 1486 of
  {\em Lecture Notes in Math.}, pages 187--195. Springer, Berlin, 1991.

\bibitem[Seg53]{MR54864}
I.~E. Segal.
\newblock A non-commutative extension of abstract integration.
\newblock {\em Ann. of Math. (2)}, 57:401--457, 1953.

\bibitem[Sti59]{StinespringMeasOps}
W.~Forrest Stinespring.
\newblock Integration theorems for gages and duality for unimodular groups.
\newblock {\em Trans. Amer. Math. Soc.}, 90:15--56, 1959.

\bibitem[Tak02]{MR1873025}
M.~Takesaki.
\newblock {\em Theory of operator algebras. {I}}, volume 124 of {\em
  Encyclopaedia of Mathematical Sciences}.
\newblock Springer-Verlag, Berlin, 2002.
\newblock Reprint of the first (1979) edition, Operator Algebras and
  Non-commutative Geometry, 5.

\bibitem[Tak03]{TakesakiII}
M.~Takesaki.
\newblock {\em Theory of operator algebras. {II}}, volume 125 of {\em
  Encyclopaedia of Mathematical Sciences}.
\newblock Springer-Verlag, Berlin, 2003.
\newblock Operator Algebras and Non-commutative Geometry, 6.

\bibitem[Thi87]{MR877991}
P.~Thieullen.
\newblock Fibr\'es dynamiques asymptotiquement compacts. {E}xposants de
  {L}yapounov. {E}ntropie. {D}imension.
\newblock {\em Ann. Inst. H. Poincar\'e Anal. Non Lin\'eaire}, 4(1):49--97,
  1987.

\bibitem[Tik87]{MR892008}
O.~E. Tikhonov.
\newblock Continuity of operator functions in topologies connected with a trace
  on a von {N}eumann algebra.
\newblock {\em Izv. Vyssh. Uchebn. Zaved. Mat.}, (1):77--79, 1987.

\end{thebibliography}
\bibliographystyle{alpha}

\end{document}